\documentclass[a4paper,11pt,reqno]{amsart} 
\usepackage{hyperref}

\usepackage{amsmath,amsfonts,amssymb,amsthm}
\usepackage{graphics}
\usepackage{epsfig, esint}
\usepackage{graphics,color}
\usepackage{xstring} 

\usepackage[nocompress]{cite} 

\usepackage[margin=1.78cm]{geometry}

\providecommand{\dy}{\, \mathrm{d} y}
\newcommand{\en}[1]{\left< #1 \right>}

\providecommand{\R}{\mathbb{R}}
\providecommand{\oh}{\frac{1}{2}}
\providecommand{\h}{\mathrm{h}}
\providecommand{\ah}{a_{\mathrm{h}}}

\providecommand{\vt}{\tilde v_\h} 
\providecommand{\ut}{\tilde u_\h}

\providecommand{\PfStart}[1]{\newcounter{#1}\refstepcounter{#1}} 
\providecommand{\PfStep}[2]{ \ifnum\value{#1}=1
\else\medskip\fi{\sc Step }\arabic{#1}\label{#2}\refstepcounter{#1}.} 
 
\newcommand{\BR}[1]{%
    \IfEqCase{#1}{%
        {1}{\color{red}} 
	{2}{\color{blue}}
	{3}{\color{black}}
	{4}{\color{green}}
	{5}{\color{black}}
	{6}{\color{black}}
    }[\PackageError{BC}{Undefined option to tree: #1}{}]%
}
\newcommand{\ER}{\color{black}}

\DeclareMathOperator*{\argmin}{\arg\!\min}

\newcommand{\Xk}{{X}_k}
\newcommand{\Xke}{{X}_k^{\mathrm{h}}}

\newcommand{\Yke}{{Y}_k^{\mathrm{h}}}
\newcommand{\Rd}{\mathbb{R}^d}

\newcommand{\Y}{{Y}}

\newcommand\Item[1][]{%
  \ifx\relax#1\relax  \item \else \item[#1] \fi
  \abovedisplayskip=0pt\abovedisplayshortskip=0pt~\vspace*{-.4cm}}

\newcommand{\ignore}[1]{}

\newtheorem{proposition}{Proposition}
\newtheorem{theorem}{Theorem}
\newtheorem{remark}{Remark}
\newtheorem{lemma}{Lemma}
\newtheorem{corollary}{Corollary}

\makeindex

\begin{document}

\title{Effective Multipoles in Random media}
\author{Peter Bella}
\thanks{The first author is supported by the German Science Foundation DFG in the context of the Emmy Noether junior research group BE 5922/1-1.}
\address{Mathematisches Institut, Universit\"at Leipzig, Augustusplatz 10, 04109 Leipzig, Germany} \email{bella@math.uni-leipzig.de}
\author{Arianna Giunti}
\address{Max Planck Institute for Mathematics in the Sciences, Inselstrasse 22, 04103 Leipzig, Germany} \email{giunti@mis.mpg.de}
\author{Felix Otto}
\address{Max Planck Institute for Mathematics in the Sciences, Inselstrasse 22, 04103 Leipzig, Germany} \email{otto@mis.mpg.de}

\begin{abstract}

In a homogeneous medium, the far-field generated by a localized source can be expanded in terms of multipoles;
the coefficients are determined by the moments of the localized charge distribution. We show that this structure
survives to some extent for a random medium in the sense of quantitative stochastic homogenization:
In three space dimensions, the effective dipole and quadrupole -- but not the octupole -- can be inferred without
knowing the realization of the random medium far away from the (overall neutral) source and the point of interest. 

Mathematically, this is achieved by using the two-scale expansion to higher order to construct isomorphisms between the
hetero- and homogeneous versions of spaces of harmonic functions that grow at a certain rate, or decay at a certain rate
away from the singularity (near the origin); these isomorphisms crucially respect the natural pairing between growing and decaying harmonic functions
given by the second Green's formula. This not only yields effective multipoles (the quotient of the spaces of
decaying functions) but also intrinsic moments (taken with respect to the elements of the spaces of growing functions). The construction
of these rigid isomorphisms relies on a good 
(and dimension-dependent) control on the higher-order correctors and their flux potentials.
\end{abstract}

\maketitle

\vspace{-.8cm}
\tableofcontents

\vspace{-.8cm}

\section{Mathematical context}

This paper is a contribution to quantitative stochastic homogenization of elliptic equations in divergence form.
More precisely, it provides an estimate of the homogenization error on the level of the gradients in a strong norm,
that is, using a two-scale expansion.
Recently, there has been a lot of activity in providing such error estimates of optimal scaling (in the ratio between
the correlation length and the macroscopic scale) and optimal stochastic integrability (of the random constant in
the error estimate). 
The error estimate provided in this paper is non-standard in two ways: It considers a right-hand side that is localized
(say, near the origin and on the scale of the correlation length) and it provides a pointwise error estimate for 
the gradient (again on the scale of the correlation length) that is increasingly better as one moves away from the support of the right-hand side. Loosely speaking, it can be seen as relating the quenched Green's function to the homogenized Green's function,
on the level of gradients. 
In this sense, this work takes up the
analysis started by the three authors in \cite{BellaGiuntiOttoPCMI}. Going beyond \cite{BellaGiuntiOttoPCMI}, this paper provides a {\it second}-order error analysis.
Moreover, instead of quenched Green's function, this paper is phrased in the language of {\it multipoles} (dipoles, quadrupoles, ...).
Roughly speaking, it establishes that in random heterogeneous media as described by the coefficient field $a$, there is a notion of {\it effective} multipoles (up to
a level that grows with the dimension $d$), and that these may be computed like in a homogeneous medium by evaluating (intrinsic) moments of the right-hand side.

\medskip
This result arises from analyzing the spaces $X_m$ of $a$-harmonic functions on the whole space $\mathbb{R}^d$ that grow at most at rate $m$
and the spaces $Y_k$ of $a$-harmonic functions in an exterior domain that decay at least at rate $(d-2)+k$, and the natural pairing between
these spaces (rather their quotients) that is given by the second Green's formula (an element already implicitly present in \cite{BellaGiuntiOttoPCMI}, but developed to
full strength here). More precisely, we shall construct canonical isomorphisms
between this pair of sequence of spaces on the Riemannian side on the one hand, and the Euclidean side on the other hand. Hence the (deterministic)
core of our analysis in Section \ref{abstract.r} is rather geometric and algebraic in flavor (and is worked out under very general assumptions).
The investigation of the spaces $\{X_m\}$ is classical in geometry and related to Liouville principles (of order $m$); they are finite-dimensional
under general assumptions \cite{ColdingMinicozzi,PeterLi}, they have the same dimension as in the Euclidean case for a periodic medium \cite{AvellanedaLinLiouville},
the same holds true for the stochastic case under mild assumptions \cite{BenjaminiDuminilKozma+} for the case of $m=1$, and by \cite[Corollary 4]{FischerOtto} in conjunction with
\cite[(16) in Theorem 1]{FischerOtto2} for the case of general $m$; under different but somewhat stronger statistical assumptions, a similar result was established in
\cite{AKM.CMP16}. Recently, these Liouville principles have been extended to the degenerate elliptic case \cite{BellaFehrmanOtto} as well as \BR6 the uniformly parabolic case~\cite{BellaChiariniFehrman} \ER under qualitative ergodic assumptions for $m=1$, to Bernoulli percolation \cite{ArmstrongDario} for any $m$, 
and to the parabolic case \cite{ArmstrongBordasMourrat} for any $m$.
The present paper, especially in its abstract Section \ref{abstract.r}, draws a lot from the work of Fischer with the last author \cite{FischerOtto}.

\medskip

This intimate connection between quantitative stochastic homogenization and elliptic regularity, of which Liouville principles are a qualitative
expression, was present from the beginning: Yurinskii used Nash's bounds on the heat kernel \cite{yurinski86} to get some rates, 
Naddaf and Spencer used Meyer's estimate \cite{NaddafSpencer} to get optimal rates in case of low-contrast media, Gloria and the last author used
a combination of both to get optimal rates \cite{GO1,GO2} for any (finite) contrast, see also \cite{GNO2} for a unified approach based
on the semi-group and spectral analysis. Up to then, elliptic regularity theory,
essentially on the level of the celebrated H\"older-theory by de Giorgi, Nash and Moser, has been used as an input. In \cite[Corollary 4]{MarahrensOtto} it was
first worked out by Marahrens and the last author that randomness {\it creates} large-scale regularity on the level of $C^{0,1-}$, in particular
beyond the DeGiorgi-Nash-Moser theory. On the level of periodic homogenization, such a large-scale regularity theory had been developed by Avellaneda and Lin
\cite{AL1} on the basis of qualitative arguments and a Campanato-type iteration, giving rise to a $C^{0,1}$-theory. Replacing the qualitative
argument by a quantitative argument, this approach was extended to the random case by Armstrong and Smart \cite{ArmstrongSmart}. In its essence, this Campanato approach
is independent of the DeGiorgi-Nash-Moser theory and in particular can be applied to systems. Also this paper allows for systems which
is of interest in particular because of the system of linear elasticity. Incidentally, equipped with \cite[Corollary 2]{ConlonGiuntiOtto},
also \cite{MarahrensOtto} extends to systems. The subsequent work of Gloria, Neukamm
and the third author \cite{GNO4} refined this Campanato-approach by working with the ``flux corrector'' $\sigma$ (or vector potential of the flux), 
see (\ref{i5}), which plays a crucial technical role in the present paper, and which is also known from periodic homogenization \cite[p.27]{JikovKozlovOleinik}.

\medskip

By now, there are several results on optimal error estimates in stochastic homogenization. Based on some of the above-mentioned
results, the first optimal error estimate in the $H^1$-norm (and thus on the level of the two-scale expansion) was given in
\cite[Theorem 1.1]{GNO1}. That at least the fluctuating part of the error obeys CLT-scaling in a {\it weak} norm (and thus is much smaller in
that topology for $d>2$) was first
established in \cite[Corollary 3 and (17)]{MarahrensOtto}. The leading-order fluctuations on the level of the corrector were identified by Mourrat and the last author in
\cite[Theorem 2.1]{MourratOtto} based on the Green's function estimates in \cite[Theorem 1]{MarahrensOtto}. This was extended by Gu and Mourrat, in a non-obvious way, to 
characterize the leading order of the fluctuations of the homogenization error and to show that they are Gaussian \cite{gumourrat-fluctuations},
still relying on \cite[Theorem 1]{MarahrensOtto}. Incidentally, the Gaussianity of fluctuations in stochastic homogenization was first
established on the level of the error in the representative volume element method: \cite{biskupsalviwolff} in the small-contrast case, 
\cite{glorianolen} based on Nolen's \cite{Nolen1}, and \cite{Rossignol}. A quite different approach to, among other things, Gaussianity of leading-order fluctuations
of the corrector, motivated by the approach to quantitative stochastic homogenization in \cite{ArmstrongSmart}, was carried out by Armstrong, Kuusi and Mourrat 
\cite{ArmstrongKuusiMourrat} (after being announced in \cite{gumourrat-fluctuations}). 
A closer connection between the leading-order fluctuations of the corrector (in form of a ``homogenization commutator'')
and the leading-order fluctuations of the homogenization error was drawn by Duerinckx, Gloria and the last author in \cite[Theorem 1]{DuerinckxGloriaOtto}.
The first result that the strong error estimate improves when passing to higher order in the two-scale expansion is due to Gu \cite{Gu_highorder}
and relies on constructing stationary higher-order correctors (provided the dimension is sufficiently high); using some of the calculus developed
in \cite{GNO2,MarahrensOtto}. While \cite{Gu_highorder} contains a result on correctors of any order, it is somewhat suboptimal in e.~g. dimension $d=3$,
because then, the second-order corrector generically is not stationary but still has tamed growth leading to an error estimate of the order $\frac{3}{2}-$.
This gap was closed by Fehrman, Fischer, and the first and last author in \cite[Theorem 3]{BellaFehrmanFischerOtto}, using the calculus developed
in \cite{GNO4}, and introducing a flux corrector $\Psi$ (see \eqref{Psi}) also on the level of the second-order corrector $\psi$, which plays a crucial technical role
in the present paper. It is \cite{BellaFehrmanFischerOtto} we rely on for the probabilistic ingredient to this paper.

\medskip

All these works are concerned with a macroscopic right-hand side; as mentioned, this paper deals with a localized right-hand side as first considered by the authors in \cite[Theorem 2]{BellaGiuntiOttoPCMI}. Moreover, like in that work, we not only consider a strong error estimate, but we establish 
essentially an error estimate in $C^1$. Error estimates with a localized right-hand side can be seen as (rather: post-processed to) estimates on the error between
the quenched Green's function and the annealed Green's function, on the level of the second mixed derivatives, \cite[Corollary 3]{BellaGiuntiOttoPCMI}. 
In the case of periodic media, such estimates were derived in \cite{AvellanedaLin91}; in case of small-contrast
random media, such optimal estimates have been established in \cite{ConlonSpencer}. 

\medskip

When it comes to the probabilistic ingredients, which in this paper we ``take from the shelf'', there are essentially two approaches:
There is the approach from Naddaf and Spencer \cite{NaddafSpencer} relying on an underlying product structure of the probability space,
at least in form of a Spectral Gap Estimate, later refined in \cite[Definition 1]{MarahrensOtto} to a Logarithmic Sobolev Estimate 
because of the ensuing concentration of measure phenomenon \cite[Lemma 4]{MarahrensOtto} and adapted in \cite{GNO4}
and \cite[Lemma 1 \& Proposition 1]{FischerOtto2} to thick correlation tails. This approach and the related sensitivity calculus has been
used in a large number of contributions and is also the one we (indirectly) rely on. While this approach is most
natural in case of a discrete medium, it has been extended to continuum media starting with \cite{GO14per}.
Armstrong and Smart \cite{ArmstrongSmart} have introduced another approach relying on 
a finite range assumption; this assumption is particularly suitable to quantify the qualitative approach by Dal Maso and Modica
\cite{dalmasomodica} to stochastic homogenization,
it is a variational approach based on decomposition and concatenation of representative volume elements. This approach can be extended to more
general mixing conditions \cite{ArmstrongMourrat}. While the first approach might lead to the optimal rates
in a more straightforward manner, the second approach naturally gives the optimal stochastic integrability of the random constant appearing in the
error estimates \cite{ArmstrongKuusiMourrat,GO5} --- at least on the level of integrable correlations on the one hand
and finite range on the other. In this paper, the main contribution of which is deterministic, we do not strive for optimal stochastic moments.

\section{The main result}\label{main.r}

Throughout the paper, we consider uniformly elliptic, not necessarily symmetric, 
coefficient fields $a$ in the $d$-dimensional space $\mathbb{R}^d$.
By uniformly elliptic, we understand that there exists a $\lambda>0$ such that
%
\begin{equation}\label{i01}
\forall\;x\in\mathbb{R}^d,\quad\forall\;\xi\in\mathbb{R}^d:\quad \xi\cdot a(x)\xi \geq \begin{cases}
                      \lambda|\xi|^2\\
                       |a(x)\xi|^2
                     \end{cases}
\end{equation}
Note that the second inequality implies $|a(x)\xi|\le|\xi|$, and is equivalent to the latter in case
of symmetric coefficients; we opt for the above form because it is the one preserved under homogenization.
Clearly, the upper bound \BR5 (i.e. the second inequality in~\eqref{i01}) \ER normalized to unity is no loss of generality. While we also allow for tensor fields that give rise to elliptic systems, we use scalar language and notation as in (\ref{i01}). 
We often think of $a$ as defining a Riemannian metric, and speak of the Euclidean case when
$a$ is homogeneous (i.e. independent of $x$, but not necessarily equal identity).

\medskip

\BR5 Motivated by the model of linear elasticity, it turns out that our results holds also if  the lower-bound assumption in~\eqref{i01} is replaced by a weaker integral version
\begin{equation}\label{weakell}
\forall \zeta\in C^\infty_0(\R^d): \quad  \int \nabla \zeta \cdot a \nabla \zeta \geq \lambda \int |\nabla\zeta|^2.
\end{equation}
In the scalar case this definition is equivalent to \eqref{i01}; in the case of systems (where now the smooth functions $\zeta$ are vector fields), \eqref{weakell} is weaker and implies \eqref{i01} only for rank-one matrices $\xi$. \BR6 There is a good reason to consider this more general assumption: \ER In the case of a model of linearly elastic materials only the symmetric part of the gradient is controlled, in particular lower bound in~\eqref{i01} does not hold, while by Korn's inequality its integral version~\eqref{weakell} does hold. \ER

\medskip

Our object of interest are $a$-harmonic functions $u$, that is, functions satisfying 
$-\nabla\cdot a\nabla u=0$, i.e. being harmonic with respect to the Laplace-Beltrami operator. 
More precisely, for two non-negative integers $m$ and $k$ 
we introduce the following spaces $X_m$ and $Y_k$:
\medskip
\begin{itemize}
\item We consider the space $X_m$ of $a$-harmonic functions on $\mathbb{R}^d$ that grow at most
at rate $m$, as measured in a square-averaged sense on the level of the gradients:
\begin{equation*}
\limsup_{R\uparrow\infty}R^{-m+1}\bigg(\frac{1}{R^d}\int_{|x| < R}|\nabla u|^2\bigg)^\frac{1}{2}<\infty.
\end{equation*}
These spaces are finite-dimensional under the assumption (\ref{i01}) \BR6 in the case of a single equation\ER, \cite[Theorem 0.3]{ColdingMinicozzi}, \cite[Corollary 7]{PeterLi}.
\medskip
\item We also consider the space $Y_k\BR6(r)$ of $a^*$-harmonic functions defined in the
exterior domain \BR3 $\{ |x| > r \}$, 
\ER for some $r<\infty$, that decay at least 
at rate $k+(d-2)$ in the sense of
\begin{equation*}
\limsup_{R\uparrow\infty}R^{k+(d-2)+1}\bigg(\frac{1}{R^d}\int_{|x| > R}|\nabla v|^2\bigg)^\frac{1}{2}<\infty.
\end{equation*}
Here $a^*$ denotes the (pointwise) transpose of $a$, which preserves (\ref{i01}).
Note that the index $k$ is normalized such that $Y_0$ contains the Green's function,
while $Y_1$ does not (at least in the Euclidean case). These spaces are infinite-dimensional,
as can be seen from considering the Lax-Milgram solution of $-\nabla\cdot a^*\nabla v=\nabla\cdot g$,
where $g$ runs through all (square-integrable) vector fields supported in $\{ |x| < r \}$. From a PDE point of view,
the spaces $\{Y_k\}_k$ are more pertinent than the spaces $\{X_m\}_m$.
\end{itemize}

\medskip

For an $a$-harmonic function $u$ and an $a^*$-harmonic function $v$, both defined on some
exterior domain $\{ |x| > r \}$, the vector-field $\xi:=va\nabla u-ua^*\nabla v$ is obviously
divergence-free. Hence its flux $\int_{\partial\Omega}\xi\cdot\nu$ through the boundary
of a bounded domain $\Omega\supset \{ |x| < r\}$ does not depend on $\Omega$ and is an invariant of
the pair $u$ and $v$. In the case of a non-smooth coefficient field $a$ and thus only distributionally harmonic functions, 
the definition
\begin{equation}\label{bil.def}
(u,v):=-\int\nabla\eta\cdot(va\nabla u-ua^*\nabla v),
\end{equation}
which does not depend on the compactly supported function $\eta$ provided it is equal to one on $\{ |x| < r \}$,
is more convenient. We understand $(\cdot,\cdot)$ as a bilinear form on 
$X_m\times Y_k$ for all non-negative integers $m$ and $k$. This bilinear form can easily
seen to vanish identically for $k>m$:
\begin{equation}\label{o02}
\forall\;k>m, \quad \forall\;u\in X_m, \quad \forall\;v\in Y_k : \quad (u,v)=0,
\end{equation}
see Corollary \ref{Lbil}. As we shall see at the end of the section, this bilinear form plays an important role from a PDE point of view 
because, in view of its definition, it encodes a conservation law
that allows to link local and far-field \BR3 behavior.\ER

\medskip

In the {\it Euclidean} case it is folklore that $(\cdot,\cdot)$ provides an
isomorphism between the (finite-dimensional) quotient space $Y_{k}/Y_{m+1}$
and the (algebraic) dual $(X_m/X_{k-1})^*$ of the quotient space $X_m/X_{k-1}$ for any integers $m\ge k\ge 1$:
\begin{align}\label{o10}
Y_{k}/Y_{m+1}\cong (X_m/X_{k-1})^*\quad\mbox{via}\quad(\cdot,\cdot).
\end{align}
We note that, because of (\ref{o02}), the linear map $Y_{k}\ni v\mapsto (\cdot,v)\in (X_m/X_{k-1})^*$ is always well-defined
and has kernel containing $Y_{m+1}$. It thus lifts to 
a linear map $Y_{k}/Y_{m+1}\ni v\mapsto (\cdot,v)\in (X_m/X_{k-1})^*$. Hence the non-trivial
part of the statement is that this map is onto and one-to-one;
where the latter means that the kernel of $Y_{k}\ni v\mapsto (\cdot,v)\in X_m^*$
is contained in $Y_{m+1}$. Knowing the latter, the property of being onto is equivalent to
the quotient spaces $(X_m/X_{k-1})^*$ and $Y_{k}/Y_{m+1}$ having the same dimension.

\medskip

A self-contained proof of the Euclidean statement (\ref{o10}) is provided by Lemma \ref{folklore}, 
which also shows that the space $X_m$ consists of polynomials of degree $\le m$,
and that the quotient space $Y_m/Y_{m+1}$ is spanned by $\{\partial^\alpha G\}_{\alpha}$, where  
$G$ denotes the fundamental solution and $\alpha$ runs over all multi-indices of degree $m$.
For the special case of $a={\rm id}$ (to which the scalar case can always be reduced to),
the elements of the space $Y_m/Y_{m+1}$ can be identified with the spherical harmonics of degree $m$. 
In this case, Arnol'd \cite[Lecture 11, pg.122]{ArnoldPDE} has shown that in fact the quotient space $Y_m/Y_{m+1}$ is {\it given} 
by all $\BR6m$-th order directional derivatives of $G$ (where it is not obvious that the latter is a linear space).

\medskip

Under which conditions on the metric $a$ and to which extent
do these algebraic properties of $\{X_m\}_m$, $\{Y_k\}_k$, and
$(\cdot,\cdot)$, survive in the Riemannian case? We give an answer in the context of {\it homogenization};
homogenization means that one can assimilate the given (heterogeneous) coefficient field $a$ with a 
homogeneous coefficient $\ah$ in the sense that the resolvents of the elliptic operators 
$-\nabla\cdot a\nabla$ and $-\nabla\cdot \ah\nabla$ are close on large scales. Hence the above question
may be rephrased as follows: Under which conditions we may construct isomorphisms between the
quotient spaces $X_m/X_{k-1}$ and $X_m^\h/X_{k-1}^\h$ (where the superscript $\h$ indicates
that the spaces refer to the homogeneous coefficient $\ah$) and between the quotient spaces $Y_{k}/Y_{m+1}$ 
and $Y_{k}^\h/Y_{m+1}^\h$, that at the same time convert the bilinear form $(\cdot,\cdot)$ into its Euclidean counterpart
$(\cdot,\cdot)_\h$:
\begin{align}\label{i00}
X_m/X_{k-1}\cong X_m^\h/X_{k-1}^\h,\;\;
Y_{k}/Y_{m+1}\cong Y_{k}^\h/Y_{m+1}^\h\;\mbox{compatible through}\;(\cdot,\cdot),(\cdot,\cdot)_\h.
\end{align}
We stress that because of the compatibility, this \BR3 contains \ER more information than the one that the spaces 
have the same dimension. Requiring this compatibility makes the isomorphisms (more) canonical.
Note that for fixed $m$, the property (\ref{i00}) is \BR3 stronger \ER the smaller $k$ is.
As we shall discuss after the statement 
of Theorem~\ref{main.sym}, this is not just a pleasing
academic question, but of practical significance for the effective behavior of heterogeneous media in the sense of \textit{effective multipoles}.

\medskip

While most of the intermediate results, see Section \ref{abstract.r}, apply to a general situation of $H$-convergence, 
we have periodic and in particular random homogenization in mind. In the case of {\it periodic homogenization}
(i.e., when the coefficient field $a$ is periodic, say, with respect to the cubic unit cell $[0,1)^d$), 
these isomorphisms can be constructed for the {\it entire range} $m\ge k\ge 1$ (where the statement is strongest
for $k=1$). While we do not display the proof of this result, the reader will see that this follows from the construction
of higher-order correctors \cite{AvellanedaLinLiouville} and their flux potentials, 
a generalization of the corrections of the two-scale expansion in the sense of Lemmas \ref{Rom1} and \ref{Rom2}, and the results in Section \ref{deterministic.r}.

\medskip

The situation is more delicate in the case of {\it random homogenization}. 
By the random case one understands that we are given an ensemble 
(i.e. a probability measure) $\langle\cdot\rangle$ of uniformly elliptic coefficient fields $a$ 
that is stationary (i.e. invariant under spatial translations of the fields) and such that
on distant spatial patches, the restrictions of $a$ are nearly independent (which means ergodicity).
Under this assumption (with a mild quantification of ergodicity, see \cite{FischerOtto2} for a class of examples) 
we are able to construct this (canonical) isomorphism in the {\it restricted range} $m=k\ge 1$.
This (implicitly) follows from the work of Fischer and the last author \cite{FischerOtto} on the construction
of higher-order correctors, again when upgraded with the results of Section \ref{abstract.r} of the present paper. 

\medskip

However, in this paper, we work under the strongest (but realistic) ergodicity assumption, which loosely speaking
corresponds to integrable correlation tails, and 
which we encode in the assumption of a Logarithmic Sobolev Inequality (LSI) (see~\cite[Theorem 1]{GNO4} \BR6 or~\eqref{LSI}\ER).
Under this natural assumption, the result turns out to be dependent on the dimension $d$: For $d>2$, 
the isomorphisms can be constructed in the {\it enlarged range} $m=k+1$, which is
the content of our main result Theorem \ref{main}. On the stochastic side, this relies on the fact that provided $d>2$, 
there exist {\it stationary} first-order correctors $\{\phi_i\}_{i=1,\ldots,d}$ \cite{GO1},
endowed with stationary flux potentials \cite{GNO2}, 
and second-order correctors $\{\psi_{ij}\}_{i,j=1,\ldots,d}$, which together with their flux potentials
grow at a rate strictly less than one (close to $\frac{1}{2}$ in $d=3$) \cite{BellaFehrmanFischerOtto}.
Based \BR3 on the stochastic side of the work \ER \cite{Gu_highorder},
the reader will see that this result generalizes as follows: For dimension $d>2n$ with $n$ an integer, 
the isomorphisms can be constructed in the range $m=k+n$.
In this sense, the higher the dimension $d$, the more the random case is as well-behaved as the periodic one.

\medskip

{\bf Notation.} Throughout this paper we use Einstein's summation convention over repeated indices, i.e. we write for instance $\phi_i \partial_i u$ for $\sum_{i=1}^d \phi_i \partial_i u $. We also adopt the compact notation 
$\fint_{|x|< R}$ and $\fint_{|x|> R}$ for $R >0$, which stands for $\frac{1}{|\{ |x|< R\}|} \int_{|x|< R}$ and $\frac{1}{|\{ |x|< R\}|} \int_{|x|> R}$, respectively.
For a given $a$, the coefficient field $a^*$  is defined such that for (almost every) $x\in \Rd$ the tensor $a^*(x)$ is the transposed of $a(x)$. Given a random field $F=F(a;\cdot)$, we use the notation $F^*$  to denote $F(a^*, \cdot)$. 

\medskip

We now state our main result, at first restricted to ensembles of coefficient fields which are not only stationary, but also either invariant under \textit{central symmetries}, namely such that $\langle \cdot \rangle$ is invariant under
the transformation \mbox{$a \rightarrow \{ x \mapsto a( -x) \}$}, or supported on symmetric coefficient fields (i.e. \mbox{$a=a^*$ $\en{\cdot}$-almost surely}). 
We then comment how to extend this theorem in the case of ensembles which do not satisfy any of these properties.
\begin{theorem}\label{main.sym}
Let $\langle\cdot\rangle$ be a stationary ensemble of uniformly elliptic coefficient fields
on $\mathbb{R}^d$ that satisfies a LSI, see \cite[(35) of Theorem 1]{GNO4} \BR6 or~\eqref{LSI}.\ER 
In addition, let us assume that either
\BR3
\begin{itemize}
 \item coefficient fields $a$ are symmetric, in the sense that $a = a^*$ $\en{\cdot}$-almost surely, or
 \item $\en{\cdot}$ is invariant under central symmetries.
\end{itemize}
\ER
Suppose $d>2$ and pick an exponent $\beta>1$ with $\beta<\frac{3}{2}$ for $d=3$ and $\beta<2$ for $d\ge 4$.
Then there exists a constant tensor $\ah$ satisfying~\eqref{i01} and for $\langle\cdot\rangle$-almost every realization $a$, there exist functions $\{\phi_i\}_{i=1,\ldots,d}, \{\phi_i^*\}_{i=1,\ldots,d}$
and $\{\psi_{ij}\}_{i,j=1,\ldots,d}, \{\psi_{ij}^*\}_{i,j=1,\ldots,d}$ with the following properties for any integer $m\ge 2$:

\smallskip

\begin{itemize}
\item For every $u_\h\in X_m^\h$ there exists an $u\in X_m$ such that
\begin{align}\label{o04.sym}
\lim_{|y|\uparrow\infty}|y|^{-m+1+\beta}
\bigg(\int_{|x-y|< 1}\left|\nabla\left(u-(1+\phi_i\partial_i+\psi_{ij}\partial_{ij})u_\h\right)\right|^2\bigg)^\frac{1}{2}
&=0.
\end{align}
%
Likewise, for every $u\in X_m$ there exists $u_\h\in X_m^\h$ such that (\ref{o04.sym}) holds.
Furthermore, this defines an isomorphism between $X_m^\h/X_{m-2}^\h$ and $X_m/X_{m-2}$.
\\
In case of $m=2$, $(1+\phi_i\partial_i+\psi_{ij}\partial_{ij} )u_\h$ itself is $a$-harmonic.

\item For every $v_\h\in Y_{m-1}^\h$ there exists a $v\in Y_{m-1}$,  
and for every $v\in Y_{m-1}$ there exists \BR3 a \ER $v_\h\in Y_{m-1}^\h$ such that
\begin{align}\label{o05.sym}
\lim_{|y|\uparrow\infty}|y|^{(m-1)+(d-2)+1+\beta}
\bigg(\int_{|x-y|<1}|\nabla(v-(1+\phi_i^*\partial_i+\psi_{ij}^*\partial_{ij})v_\h)|^2\bigg)^\frac{1}{2}
&=0.
\end{align}
This defines an isomorphism between $Y_{m-1}^\h/Y_{m+1}^\h$ and $Y_{m-1}/Y_{m+1}$.

\item Finally, these two isomorphisms 
respect the bilinear forms: $(u,v)=(u_\h,v_\h)_\h$.

\end{itemize}

\end{theorem}

A few technical comments are in place: The fact that 
$(1+\phi_i\partial_i+\psi_{ij}\partial_{ij})u_\h$ is $a$-harmonic for all $u_\h\in X_2^\h$
implies that $\{\phi_i\}_{i=1,\ldots,d}$ and $\{\psi_{ij}\}_{i,j=1,\ldots,d}$ are in fact 
first and second-order correctors, respectively. We note in passing that $\{\psi_{ij}\}_{i,j=1,\ldots,d}$ is only needed
in the combination of $\psi_{ij}E_{ij}$, where $E$ is a matrix with $a_{\h ij}E_{ij}=0$.
The expression $(1+\phi_i\partial_i+\psi_{ij}\partial_{ij})u_\h$ amounts to the first three terms in the
asymptotic (two-scale) expansion in (periodic or random) homogenization.
In this sense, (\ref{o04.sym}) and (\ref{o05.sym}) amount to an estimate of the homogenization error;
they state that the relative homogenization error is of the order $\beta$. The homogenization error is almost local
on the level of the gradient; assuming in addition local smoothness of the coefficient field one obtains
a pointwise result by standard regularity theory. For dimensions $d\ge 4$, the order $\beta$ is arbitrarily close to $2$, as one
would expect after correcting with first and second-order correctors. However, in $d=3$, one looses half of an order
since the second-order corrector is typically non-stationary and grows with rate $\frac{1}{2}$ (worsened by a 
logarithm).

\medskip

Before turning to the general case, which requires the introduction of further objects, we address the significance of Theorem \ref{main.sym}, focussing on the most relevant dimension of $d=3$,
and using the language of electrostatics. As it will become apparent with the statement of the main result in the case of general coefficient fields, this same remark extends also to the case of more general statistics of $a$. 
Suppose we are given a localized dipole distribution, as described by a vector field $g$ supported, say, in $\{|x| < 1 \}$.
In our medium of conductivity $a^*$, this charge distribution $\nabla\cdot g$ generates an electric potential $v$, 
which is $a^*$-harmonic outside of $\{|x| < 1 \}$ and solves
\begin{align}\label{o05.b}
-\nabla\cdot a^*\nabla v=\nabla\cdot g.
\end{align} 
%

%
It follows from Lemma \ref{LF1} ( b)$\Rightarrow$ a) ) that the Lax-Milgram solution of \eqref{o05.b} satisfies $v \in Y_1$. Therefore, by Theorem \ref{main.sym} there exists a $v_\h \in Y_1^\h$ such that
\begin{align}\label{o06}
\lim_{|y|\uparrow\infty}|y|^{3+\beta}
\bigg(\int_{|x-y|<1}|\nabla(v-(1+\phi_i^*\partial_i+\psi_{ij}^*\partial_{ij})v_\h)|^2\bigg)^\frac{1}{2}
=0,
\end{align}
which determines $v_\h$ through $v$ up to an element in $Y_3^\h$. This is more or less standard --- the interesting
question is whether we may easily characterize $v_\h$ as an element of $Y_1^\h/Y_3^\h$ in terms of the charge distribution
$\nabla\cdot g$. Note that in the language of electrostatics, $Y_1^\h/Y_3^\h$ is the information about the far field
of a dipole and a quadrupole; in the Euclidean case it can be extracted from the first and second moments
of the charge distribution. 

\medskip

The isomorphism of Theorem \ref{main.sym} ensures that the analogue stays true in the
Riemannian case: Because of the natural isomorphism between $Y_1^\h/Y_3^\h$ and $(X_2^\h/X_0^\h)^*$ provided through
$(\cdot,\cdot)_\h$,  $v_\h\in Y_1^\h/Y_3^\h$ is uniquely determined through the linear form
\begin{align*}
\ell.u_\h=(u_\h,v_\h)_\h \ \ \quad\mbox{for}\;u_\h\in X_2^\h.
\end{align*}
By Theorem \ref{main.sym}, we have
\begin{align}\label{o07}
(u_\h,v_\h)_\h=(u,v),
\end{align}
where because of $u_\h\in X_m$ with $m=2$, $u$ is related to $u_\h$ by the formula 
$u=(1+\phi_i\partial_i+\psi_{ij}\partial_{ij})u_\h$. It follows from integration by parts based on (\ref{o05.b})
that
\begin{align}\label{o08}
(u,v)=\int \nabla\bigl((1+\phi_i\partial_i+\psi_{ij}\partial_{ij})u_\h\bigr)\cdot g,
\end{align}
where the right-hand side may be interpreted as the first and second moments of the charge distribution $\nabla\cdot g$,
however not taken with respect to the Euclidean polynomials, but with respect to the $a$-harmonic ``polynomials''.
Incidentally, while $\psi_{ij}$ may be changed by an additive constant without affecting the value of (\ref{o08}),
this is not true for $\phi_i$. In fact, for (\ref{o08}) to hold, the constant is fixed by the requirement that
$\lim_{R\uparrow\infty}\BR3 \fint_{|x|< R} \phi_i=0$, which almost surely follows from the normalization $\langle \phi_i\rangle=0$, see Section \ref{stochastic.r} and Corollary~\ref{ergodicity}.
Note that this result is not covered by standard homogenization since the characteristic scale of the right-hand side 
$\nabla\cdot g$ of (\ref{o05.b}) is of order one and not large (compared to the effective correlation length of order one of
the medium).

\medskip

Let us interpret this finding in yet another way: If we are interested in predicting $\nabla v$ in the neighborhood \BR3 $\{ |x - y| < 1\}$ 
of some point $y$ with $|y|\gg 1$, in view of (\ref{o06}), we need to know $\phi_i^*$ and $\psi_{ij}^*$ in the neighborhood of $y$\ER , and we need to know $v_\h$. In order to get $v_\h$, in view of (\ref{o07}) and (\ref{o08}), we need to know   
$\phi_i$ and $\psi_{ij}$ in the neighborhood $\{ |x| < 1 \}$ of the origin. Hence the local knowledge of the
first and second-order correctors near the two distant points $y$ and $0$ is enough to understand how the
random medium transmits the information of the charge distribution $\nabla\cdot g$ in $\{ |x| < 1 \}$ to the field $-\nabla v$ near $y$. 
Since good approximations to the first and second-order correctors can locally be
obtained by the representative volume element method, Theorem~\ref{main.sym} (or more generally Theorem~\ref{main}) teaches us that given a charge distribution
$\nabla\cdot g$ localized near the origin, 
the field $-\nabla v$ \BR3 near $y$ 
for some distant point $y$ \ER may be inferred
(to order $\beta$)
without knowing (the details of) the medium $a^*$ further away from (and in particular between) the points $0$ and $y$.
In view of our discussion before Theorem \ref{main.sym}, this is expected to be true to arbitrary rate $\beta<\infty$
in the periodic case, while in the random case, the rate is limited to $\beta<\frac{d}{2}$, which in three dimensions
is nevertheless enough to predict the quadrupole next to the dipole.

\medskip

If we do not impose a symmetry property on the ensemble, 
then the same statement of Theorem \ref{main.sym} holds, provided we modify the maps which induce the isomorphisms between the quotient spaces.
More precisely, in the estimate~\eqref{o04.sym} the (two-scale) expansions for $u_\h$, i.e. the term $(1 + \phi_i\partial_i + \psi_{ij}\partial_{ij})u_\h$, has an additional term given by the expansion 
$(1+ \phi_i\partial_i)\ut$. Similarly, \BR6 in~\eqref{o05.sym} we expand \ER the term $(1+ \phi_i^*\partial_i + \psi_{ij}^*\partial_{ij})v_\h$ by $(1+ \phi_i^* \partial_i)\vt$. Here, $\ut$ and $\vt$ are homogeneous
solutions of a constant-coefficient equation and are uniquely determined by $u_\h$ and $v_\h$ \BR6 (see~\eqref{u.tilde}, \eqref{v.tilde} and the discussion afterwards)\ER. In other words, in order for an analogue of Theorem \ref{main.sym} to hold also in the general case, we first construct two maps
$X_m^\h \ni u_\h \mapsto \ut$, $m\geq 3$, and $Y_{m-1}^\h \ni v_\h \mapsto \vt$, $m \geq 2$, with which we modify the two-scale expansions in \eqref{o04.sym} and \eqref{o05.sym}. 

\medskip

To give a definition of these two maps, we first need to recall that since $\ah$ is constant, we may identify the quotient spaces  $X_m^\h/ \BR6 X_{m-1}^\h$ and  $Y_{m-1}^\h/ Y_{m}^\h$  with the space
of $\ah$-harmonic homogeneous polynomials of degree $m$ and the space of linear combinations of $\{ \partial^\alpha G \}_{\BR6 |\alpha|=m-1}$, respectively (for a self-contained proof, see Lemma \ref{folklore} in Section \ref{deterministic.r}). If we denote by $u_\h'$ and $v_\h'$ the
projections of $u_\h$ and $v_\h$  onto these two spaces, respectively, we require that $\ut$ and $\vt$ solve the equations
\begin{align}
-\nabla \cdot \ah \nabla\ut &= \nabla \cdot \bigl( \partial_{ij}u_\h'C_{ij}^{\mathrm{sym}} \bigr) \quad &&\text{ in $\mathbb{R}^d$,}\label{u.tilde}\\
-\nabla \cdot \ah^* \nabla\vt &= \nabla \cdot \bigl(\partial_{ij}v_\h' C_{ij}^{*,\mathrm{sym}} \bigr) \quad  &&\text{ in $\mathbb{R}^d \backslash \{ 0 \}$,}\label{v.tilde}
\end{align}
where $C^{\mathrm{sym}} =\{ C_{ijk}^{\mathrm{sym}} \}_{i,j,k=1,\ldots,d}$ is the symmetrization of a constant tensor $C$ in all three indices (see below for definition of $C$ and $C^*$). 
The appearance of this tensor and the need for the additional corrections $(1+ \phi_i\partial_i)\ut$ and $(1 + \phi_i \partial_i^*)\vt$, with $\ut$ and $\vt$ solving equations 
\eqref{u.tilde} and \eqref{v.tilde}, are not surprising: In homogenization they naturally arise when computing the two-scale expansion for the heterogeneous solution via second or higher-order correctors. 
In the case of periodic homogenization, where correctors of any order exist and are bounded, this is a well-established result (see, for instance, \cite[(2.13)]{AllBrVan}).

\medskip

\BR6 In order to be of lower order with respect to $u_\h$ and $v_\h$, we require $\ut$ to be a homogeneous polynomial of degree $m-1$ and $\vt$ to be a homogeneous function of degree $-(d-2 + m)$\ER. It is clear that the equations above and this last condition are not enough to uniquely define $\ut$ and $\vt$ (in fact, we may add any homogeneous element of $X_{m-1}^\h$ and $Y_{m}^\h$, respectively). 
To obtain uniqueness, we require conditions which at this stage may look obscure, but which are enough to uniquely define $\ut$ and $\vt$ (see Lemmas \ref{Rom1} and \ref{Rom2}) and which are needed for preservation of the bilinear forms $( \cdot, \cdot )$ and $(\cdot, \cdot )_\h$ under the isomorphisms: For any radius $R>0$, we require $\ut$ and $\vt$ to be such that for every $v' \in Y_{m-1}^\h/Y_{m}^\h$ 
\begin{align}
\int_{|x| = R \ER}\nu\cdot \big(v'(\ah\nabla\ut+\partial_{ij}u_\h'C_{ij}^{\mathrm{sym}})-\ut \ah^*\nabla v'\big)= 0,
\label{u.tilde.condition}
\end{align}
and also that for every  $u'\in X_{m}^\h/X_{m-1}^\h$
\begin{equation}
\int_{|x|=R}\nu\cdot \big(\vt \ah\nabla u'-u'(\ah^*\nabla\vt +\partial_{ij}v_\h'C_{ij}^{*,\mathrm{sym}})\big)
= \int_{|x|=R}\nu_k \partial_i v_\h' \partial_j u' C_{ijk}^{\mathrm{sym}}.
\label{v.tilde.condition}
\end{equation}
We remark that it is completely arbitrary whether $\ut$ or $\vt$ takes care of the extra term $\int_{|x|=R}\nu_k \partial_i v_\h' \partial_j u' C_{ijk}^{\mathrm{sym}}$.

\medskip
Existence and uniqueness of $\ut$ and $\vt$, as defined by~\eqref{u.tilde}, \eqref{u.tilde.condition}, and~\eqref{v.tilde}, \eqref{v.tilde.condition}, are established in Lemma~\ref{Rom1} and Lemma~\ref{Rom2}, respectively.  

\begin{remark}\label{symmetrization}
In our framework of stochastic homogenization, the tensor $C=\{C_{ijk}\}_{i,j,k=1,\ldots,d}$  can be written as an ensemble average involving the first-order corrector, see \eqref{definition.C}. We argue in Lemma \ref{tensor.C} 
that $C$ (or at least its symmetric part $C^{\mathrm{sym}}$) vanishes for an ensemble that is invariant under central reflections or is supported on symmetric coefficient fields, as considered in Theorem~\ref{main.sym}. 
In this case, it follows from the uniqueness
statements in Lemmas \ref{Rom1} and \ref{Rom2} that $\vt$ and $\ut$ always vanish, so that Theorem \ref{main} below
reduces to Theorem~\ref{main.sym}.


%
%
\end{remark}

We finally give the statement of the main theorem in the general case:

\begin{theorem}\label{main}
Let $\langle\cdot\rangle$ be a stationary ensemble of uniformly elliptic coefficient fields
on $\mathbb{R}^d$ that satisfies a LSI, see \cite[(35) of Theorem 1]{GNO4} or~\eqref{LSI}. 
Suppose $d>2$ and pick an exponent $\beta>1$ with $\beta<\frac{3}{2}$ for $d=3$ and $\beta<2$ for $d\ge 4$. 
Then there exists \BR5 a tensor $\ah$ satisfying~\eqref{i01} \ER and a tensor $C=\{C_{ijk}\}_{i,j,k=1,\ldots,d}$ and, for $\langle\cdot\rangle$-almost every configuration $a$,
functions $\{\phi_i\}_{i=1,\ldots,d}$, $\{\phi_i^*\}_{i=1,\ldots,d}$, $\{\psi_{ij}\}_{i,j=1,\ldots,d}$, 
and $\{\psi_{ij}^*\}_{i,j=1,\ldots,d}$ with the following property for any integer $m\ge 2$. If we consider
the two linear maps 
\begin{equation*}
X_m^\h/X_{m-1}^\h\ni u_\h\mapsto \ut, \quad Y_{m-1}^\h/Y_{m}^\h\ni v_\h\mapsto \vt
\end{equation*}
defined through Lemmas \ref{Rom1} and \ref{Rom2}, then we have:


%
%
%

\begin{itemize}
\item For every $u_\h \in X_m^\h$ there exists a $u\in X_m$, and for every $u\in X_m$ there exists a $u_\h\in X_m^\h$, such that

\begin{equation}\label{o04}
\lim_{|y|\uparrow\infty}|y|^{-m+1+\beta}
\Big(\int_{|x-y|<1}\hspace{-0.4cm}|\nabla\big(u- ( (1+\phi_i\partial_i)(u_\h + \ut) + \psi_{ij}\partial_{ij}u_\h )\big)|^2\Big)^\frac{1}{2}
=0.
\end{equation}

This defines an isomorphism between $X_m^\h/X_{m-2}^\h$ and $X_m/X_{m-2}$.

\smallskip


\item For every $v_\h\in Y_{m-1}^\h$ there exists a $v\in Y_{m-1}$, and for every $v\in Y_{m-1}$ a $v_\h\in Y_{m-1}^\h$, such that
\begin{equation}\label{o05}
\lim_{|y|\uparrow\infty}|y|^{(d-1) + (m-1)+\beta}
\Big(\int_{|x-y|<1}\hspace{-0.4cm}|\nabla\big(v-((1+\phi_i^*\partial_i)(v_\h + \vt)+\psi_{ij}^*\partial_{ij}v_\h ) \big)|^2\Big)^\frac{1}{2}
=0.
\end{equation}
This defines an isomorphism between $Y_{m-1}^\h/Y_{m+1}^\h$ and $Y_{m-1}/Y_{m+1}$.

\item Finally, these two isomorphisms are such that they respect the bilinear forms: $(u,v)=(u_\h,v_\h)_\h$.

\end{itemize}
\end{theorem}

\section{First and second-order correctors in stochastic homogenization}\label{stochastic.r}

This section is devoted to collect the results on the correctors in stochastic homogenization on which we rely in the proof of the main results of this paper. The next section provides the deterministic results needed for Theorem \ref{main} and relies on
the properties of the correctors which are taken as an input. Here, we report that these properties are generic, i.e. hold for $\langle \cdot \rangle$-almost every coefficient field $a$, provided that the ensemble is stationary and satisfies a LSI. 
As for the main statement, we assume that $d > 2$.

\medskip

The purpose of the (first-order) correctors is to correct
affine functions to $a$-harmonic functions. 
In terms of the standard Euclidean basis, the first-order
correctors $\{\phi_i\}_{i=1,\ldots,d}$ are thus characterized by
\begin{align}\label{i2}
-\nabla\cdot a\nabla(x_i+\phi_i)=0,
\end{align}
with the understanding that this and the following differential equations hold in the whole
space $\mathbb{R}^d$. The fundamental result in qualitative stochastic homogenization,
i.\ e.\ for stationary ensembles $\langle\cdot\rangle$ that are just ergodic, states that
there exists a unique curl-free random vector field $\nabla\phi_i$, $i=1,\ldots,d$, 
that is stationary, of bounded second moments and vanishing expectation, and such that
(\ref{i2}) holds almost-surely in $a$. Here
stationarity of a random field $g$, i.\ e.\ of a function(al) $g=g(a,x)$, 
means shift invariance in the sense of
$g(a(\cdot+y),x)=g(a,x+y)$. This result is due to Kozlov and Papanicolaou \& Varadhan
\cite{papvar,Kozlov}, see \cite[Chapter 7, Section 7.2]{JikovKozlovOleinik} for a textbook. 
As a consequence, the potential $\phi_i$ has sublinear 
growth so that indeed the correction does not dominate the affine function $x_i$. 
It follows from this, for instance by the method of oscillating test functions,
that the operator $-\nabla\cdot a\nabla$ $H$-converges to
$-\nabla\cdot \ah\nabla$ with the constant and deterministic coefficient $\ah$ given by
\begin{align}\label{ir12}
\ah e_i=\langle a(e_i+\nabla e_i)\rangle,
\end{align}
see \cite{TartarBook2009} for an introduction to this notion of convergence and the
type of arguments.

\medskip

In view of (\ref{i2}) and (\ref{ir12}) the flux $a(e_i+\nabla\phi_i)$ is divergence-free
and of expectation $\ah e_i$. It is thus natural to consider, next to the correction
$\phi_i$ of the (scalar) potential of the closed $1$-form $e_i+\nabla\phi_i$,
also 
the correction of the (vector) potential of the closed $(d-1)$-form $a(e_i+\nabla\phi_i)$.
This potential is a $(d-2)$-form and thus represented by a skew symmetric tensor 
$\sigma_i=\{\sigma_{ijk}\}_{j,k=1,\ldots,d}$. Hence we are lead to consider
\begin{align}\label{i5}
\sigma_i\;\mbox{skew},\quad\nabla\cdot\sigma_i=a(e_i+\nabla\phi_i)-\ah e_i,
\end{align}
with the understanding that for a tensor field $\sigma$, 
$(\nabla\cdot\sigma)_j:=\partial_k\sigma_{jk}$. We note that because of the
skew-symmetry we have $\nabla\cdot\nabla\cdot\sigma_i=0$ so that (\ref{i5}) contains
(\ref{i2}). Even in case of periodic homogenization, (\ref{i5}) leaves $\sigma_i$
under-determined up to a $(d-3)$-form, so that a gauge has to be chosen for
the construction of $\sigma_i$. Under the assumptions of qualitative
stochastic homogenization, this can be done in such a way that there exists
a unique random curl-free tensor-field $\nabla \sigma_{ijk}$, $i,j,k=1,\ldots,d$,
that is stationary, of bounded second moments and vanishing expectation, and such that
(\ref{i5}) holds almost surely \cite[Lemma 1]{GNO4}. As a consequence,
$\sigma_i$ grows sublinearly almost surely. Hence the status of the flux corrector
$\{\sigma_i\}_{i=1,\ldots,d}$ is similar to the one of $\{\phi_i\}_{i=1,\ldots,d}$.

\medskip

In homogenization, given an $\ah$-harmonic function $u_\h$,
it is common to consider the two-scale expansion
$(1+\phi_i\partial_i)u_\h$, which is a first-order approximation to
an $a$-harmonic function. The main merit of the flux corrector $\{\sigma_i\}_{i=1,\ldots,d}$
is that it allows for a divergence-form representation of the residuum:
\begin{align}\label{ir11}
-\nabla\cdot a\nabla(1+\phi_i\partial_i)u_\h
=-\nabla\cdot((\phi_ia-\sigma_i)\nabla\partial_iu_\h).
\end{align}
This is an easy calculation, see the arguments leading to (\ref{two.scale.1}) in Proposition \ref{Rom4}
for a second-order version of this calculation. We note that (\ref{ir11})
turns into (\ref{i2}) for $u_\h\in X_1^\h$, that is, for affine $u_\h$. 
This motivates the following notion of {\it second-order} correctors. 
Given $u_\h\in X_2^\h$, that is, a quadratic $\ah$-harmonic function $u_\h$,
we seek a function $\psi_{u_\h}$ that corrects $(1+\phi_i\partial_i)u_\h$
so that $(1+\phi_i\partial_i)u_\h+\psi_{u_\h}$ is $a$-harmonic. In view
of (\ref{ir11}), this second-order corrector must satisfy
\begin{align*}
-\nabla\cdot a\nabla\psi_{u_\h}
=\nabla\cdot((\phi_ia-\sigma_i)\nabla\partial_iu_\h).
\end{align*}
Like for the first-order correctors, it is convenient to work with a basis
$\{\psi_{ij}\}_{i,j=1,\ldots,d}$ characterized by the equation
\begin{align}\label{ir13}
-\nabla\cdot a\nabla\psi_{ij}
=\nabla\cdot((\phi_ia-\sigma_i)e_j)
\end{align}
so that $\psi_{u_\h}=\psi_{ij}\partial_{ij}u_\h$. However, we
only need $\{\psi_{ij}\}_{i,j=1,\ldots,d}$ in form of linear combinations
$E_{ij}\psi_{ij}$ with coefficients $\{E_{ij}\}_{i,j=1,\ldots,d}$
that are symmetric and trace-free in the sense of  $a_{\h ij}E_{ij}=0$
(so that $u_\h(x)=\frac{1}{2}E_{ij}x_ix_j$ is $\ah$-harmonic). We note that
the first and second-order correctors $(\phi,\psi):=\{\phi_i,\psi_{ij}\}_{i,j=1,\ldots,d}$,
are characterized by providing a map 
\begin{align*}
X_2^\h\ni u_\h\mapsto (1+\phi_i\partial_i+\psi_{ij}\partial_{ij})u_\h\in X_2,
\end{align*}
provided of course they grow sublinearly and subquadratically, respectively.

\medskip

The solution theory for $\psi_{ij}$, even on the level of the gradient $\nabla\psi_{ij}$,
is more subtle than for (\ref{i2}). Under the assumption that the augmented
corrector $(\phi,\sigma):=\{\phi_i,\sigma_{ijk}\}_{i,j,k=1,\ldots,d}$ 
grows by a logarithm less than linearly, there exists
a unique subquadratic solution of (\ref{ir13}) \cite{FischerOtto}. This
slightly tamer growth of $(\phi,\sigma)$ holds under a mild quantification
of ergodicity, see for instance \cite{FischerOtto2}. However, we need more
than that; in particular, we need that $\nabla\psi_{ij}$ can be chosen to be 
stationary and of vanishing expectation, so that $\psi_{ij}$ has sublinear
and not just subquadratic growth. 
In view of (\ref{ir13}) this can only be expected if $(\phi,\sigma)$ and not
just their gradients can be chosen to be stationary, which under the best ergodicity
assumptions can only be expected in $d>2$. That this stationarity can be achieved
under natural ergodicity assumptions was shown in \cite{GO1} (for $\phi$
and in \cite{GNO2} for $\sigma$). For $d>4$, the function $\psi_{ij}$ itself
(and thus a fortiori $\nabla\psi_{ij}$) can be chosen to be stationary so that it has
even subalgebraic growth \cite{Gu_highorder};
again, this can only be expected to hold for $d>4$. In order to optimally
treat the most relevant case of $d=3$, we need to resort to the results of
\cite{BellaFehrmanFischerOtto}, which are collected in Proposition \ref{stoch.results}.
Loosely speaking, (\ref{T.2}) states that $\psi_{ij}$ grows not much worse than the square root
in $d=3$ and subalgebraically for $d>3$.

\medskip

The existence of second-order correctors suggests to upgrade the two-scale
expansion from first to second order as $(1+\phi_i\partial_i+\psi_{ij}\partial_{ij})u_\h$ 
for a given $\ah$-harmonic $u_\h$. In order to capture that this function is closer to being 
$a$-harmonic than the first-order version, it is convenient to introduce flux-correctors
on the second-order level like we did on the first-order level, see~\eqref{i5}.
In view of (\ref{ir13}) (and following the notation of \cite{BellaFehrmanFischerOtto})
for $i,j=1,\ldots,d$ we introduce $\Psi_{ij}:=\{\Psi_{ijkn}\}_{k,n=1,\ldots,d}$ via
\begin{align}\label{Psi}
\Psi_{ij}\;\mbox{skew},\quad \nabla\cdot\Psi_{ij} =a\nabla\psi_{ij}+(\phi_ia-\sigma_i)e_j
-C_{ij},
\end{align}
where the vector $C_{ij}$ plays a similar role to $\ah$ in \eqref{i5}, so that in view of
(\ref{ir12}), we define
\begin{align}\label{ir14}
C_{ij}:=\langle a\nabla\psi_{ij}+(\phi_ia-\sigma_i)e_j\rangle.
\end{align}
The latter makes only sense when $\phi_i$, $\sigma_j$, and $\nabla\psi_{ij}$
are stationary, which as discussed above is the case under our assumptions.
Equipped with $\Psi_{ij}$ we obtain the second-order version of (\ref{ir11});
under the additional assumption that the tensor $C$ vanishes, it assumes the form
\begin{align*}
-\nabla\cdot a\nabla(1+\phi_i\partial_i+\psi_{ij}\partial_{ij})u_\h
=-\nabla\cdot((\psi_{ij}a-\Psi_{ij})\nabla\partial_{ij}u_\h),
\end{align*}
which is an easy calculation, leading to~\eqref{two.scale.1} in Proposition 8.
We note that as for (\ref{i5}) and (\ref{i2}), (\ref{Psi}) contains (\ref{ir13}).

\medskip

After these motivations we now collect the properties of $(\phi, \sigma)$ and 
$(\psi, \Psi)$ known to hold under assumptions of stationarity and LSI 
on the ensemble $\langle \cdot \rangle$. 

\begin{proposition}{\bf (Existence and sublinearity of the correctors, 
\cite{GNO4, FischerOtto, BellaFehrmanFischerOtto}) \, }\label{stoch.results} 
Let $d>2$ and $\beta\in(1,\frac{3}{2})$ in case of $d=3$ and $\beta\in(1,2)$ in case of $d>3$. 
Let $\langle \cdot \rangle$ be a stationary ensemble satisfying a $LSI$. 
Then there exist random fields $\phi=\{\phi_i\}_{i=1,\ldots,d}$, $\sigma=\{\sigma_{ijk}\}_{i,j,k=1,\ldots,d}$, 
$\psi=\{\psi_{ij}\}_{i,j=1,\ldots,d}$, $\Psi=\{\Psi_{ijkn}\}_{i,j,k,n=1,\ldots,d}$, and $r_*>0$, 
with the following properties:
\begin{itemize}
\item The random fields $\phi$, $\sigma$, $\nabla\psi$, $\nabla \Psi$, and $r_*$ are stationary; 
$\nabla\phi$, $\phi$, \BR6 $\nabla \sigma$, \ER $\sigma$, $\nabla\psi$, and $\nabla \Psi$ have finite second moments,
so that in particular $\ah$ and $C$ are well-defined through (\ref{ir12}) and (\ref{ir14}). 
We impose the normalization
\begin{align}\label{ir14bis}
\langle\phi\rangle=0,\quad\langle\sigma\rangle=0.
\end{align}
\item For $\langle \cdot \rangle$-almost every $a$, $(\phi,\sigma)$ satisfies (\ref{i5}) 
and $(\psi,\Psi)$ satisfies (\ref{Psi}).
\item All algebraic moments of $\int_{|x|<1}|(\phi,\sigma)|^2$ and $r_*$ are finite.
\item For $\langle \cdot \rangle$-almost every $a$, we have $(2-\beta)$-growth of 
$(\psi,\Psi)$ starting from radius $r_*(y)$ at any point $y$, that is,
\begin{align}
\frac{1}{R^2} \biggl(\fint_{|x-y|<R}\biggl|(\psi,\Psi)-\fint_{|x-y|< R}(\psi,\Psi)\biggr|^2\biggr)^\frac{1}{2}
&\le\biggl(\frac{r_*(y)}{R}\biggr)^{\beta} \quad\mbox{for all}\;R\ge r_*(y), \label{T.2}
\end{align}
where the stationary extension of $r_*$ is defined by $r_*(y) := r_*(a(\cdot+y))$. 
\end{itemize}
The same properties hold for the objects $\phi^*$, $\sigma^*$, $\psi^*$, $\Psi^*$, $\ah^*$, and $C^*$ 
coming from the transposed coefficient field $a^*$. 
\end{proposition}

The proof of existence of stationary gradients $\nabla \phi$ and $\nabla \sigma$ for stationary and ergodic ensembles can be found for example in~\cite[Lemma 1]{GNO4}, while the stationarity of $\phi$ and $\sigma$ as well as the finiteness of all moments for $\int_{|x|<1} |(\phi,\sigma)|^2$ (assuming $\en{\cdot}$ satisfies LSI) 
are the content of~\cite[Theorem 3]{GNO4}. Provided the first-order correctors $\phi$ and $\sigma$ are stationary, 
the existence of stationary $\nabla \psi$ and $\nabla \Psi$ follows from an argument along the lines of the proof of~\cite[Lemma 1]{GNO4}. For completeness, we sketch the argument in the appendix (see~Appendix~\ref{construction_psiPsi}). 

\medskip
Finally, estimates similar to~\eqref{T.2} on the growth of the second-order correctors $\psi$ and $\Psi$ are given in~\cite[Theorem 4]{BellaFehrmanFischerOtto}.
The difference between our assumption~\eqref{T.2} and statement of~\cite[Theorem 4]{BellaFehrmanFischerOtto} is twofold. First, as noted below, the estimate in~\cite[Theorem 4]{BellaFehrmanFischerOtto} is given for one radius $R \ge 1$, but can be easily extended to a set of (say dyadic) radii at the expense of additional logarithmic factor. In turn, such factor can be bounded by $R^{\epsilon}$ for any $\epsilon > 0$, hence giving an estimate with a  slightly smaller exponent $\beta$. Nevertheless, since the values of $\beta$ are assumed to be strictly smaller than $3/2$ and $2$ in the case $d=3$ and $d\ge4$, respectively, we still obtain for these values of $\beta$ the estimate
\begin{equation}\label{T.2BFFO}
\biggl(\fint_{|x-y|<R}\biggl|(\psi,\Psi)-\fint_{|x-y|< R}(\psi,\Psi)\biggr|^2\biggr)^\frac{1}{2}
\le C(a,y,\beta) R^{\beta} \quad\mbox{for all}\;R\ge 1,
\end{equation}
where $C(a,y,\beta)$ has all algebraic moments (in fact even stretched exponential). Second, while randomness in~\eqref{T.2} is expressed through the smallest radius $r_*$ from which the estimate starts to hold, in~\cite[Theorem 4]{BellaFehrmanFischerOtto} the randomness enters through the random factor $C(a,y,\beta)$ on the right-hand side. To obtain~\eqref{T.2}, choosing $r_*(y) := C(a,y,\beta)^{1/\beta}$ and
dividing~\eqref{T.2BFFO} by $R^2$ yields \eqref{T.2} for $R \ge r_*(y)$. Since all algebraic moments of $C$ are bounded, the same holds also for $r_*(y)$. 

\medskip
As already mentioned, the stationary ensemble $\en{\cdot}$ is assumed to satisfy the Logarithmic Sobolev Inequality (see, e.g. \cite[(35) of Theorem 1]{GNO4} or~\cite[Definition 1]{BellaFehrmanFischerOtto}), which for any random variable $\xi=\xi(a)$ requires
\begin{equation}
\label{LSI}
\en{ \xi^2 \log \xi^2 }-\en{ \xi^2 } \en{\log \xi^2}
\le 
\sum_{z \in \mathbb{Z}^d} \left(\int_{z+[0,1)^d} \left|\frac{\partial \xi}{\partial a}(x)\right| \,dx \right)^2,
\end{equation}
where $\frac{\partial \xi}{\partial a}$ denotes the Fr\'echet derivative of $\xi$ with respect to the coefficient field $a=a(x)$. \BR6 In the present paper we consider the classical Logarithmic Sobolev inequality (which loosely speaking corresponds to integrable correlations tails for the ensemble), in which case the partition of $\R^d$ appearing on the right-hand side~\eqref{LSI} 
consists of unit cubes. To cover more general ensembles, following~\cite{BellaFehrmanFischerOtto,GNO4} one can replace the unit cubes in~\eqref{LSI} with a partition $\{ D \}$ of $\R^d$ 
satisfying for some $\tilde \beta \in [0,1)$
\begin{align*}
(\operatorname{dist}(D,0)+1)^{\tilde \beta} \leq \operatorname{diam}(D)\leq C(d) (\operatorname{dist}(D,0)+1)^{\tilde \beta}.
\end{align*}
 The parameter $\tilde \beta$ is related to the decay rate of correlations of the coefficient fields, and the classical LSI corresponds to the choice $\tilde \beta = 0$. As can be read from~\cite[Theorem 3]{GNO4}, the first-order correctors $(\phi,\sigma)$ are stationary (hence grow slower than any positive power of $|x|$) for values $0 \le \tilde \beta < 1-\frac{2}{d}$. The second-order correctors $(\psi,\Psi)$ are stationary provided $0 \le \tilde \beta < 1 - \frac 4 d$~\cite[Theorem 4]{BellaFehrmanFischerOtto}, and grow like $|x|^{2-d(1-\tilde\beta)/2}$ for $1 - \frac 4 d \leq \tilde \beta < 1$. In particular, one can see the limitations in the most physically relevant case $d=3$: even assuming that $\en{\cdot}$ has integrable correlations, which corresponds to the best value $\tilde \beta = 0$, the second-order correctors will grow like square root (which then corresponds to $\beta < 3/2$).

\medskip 
While the LSI assumption in the form of~\eqref{LSI} is satisfied by a large family of models for random coefficients, for example by Gaussian ensembles~\cite[Lemma 4]{GNO4}, models for random coefficients constructed from the Poisson Point Process (see~\cite[example after Theorem 6]{BellaOtto}) do not satisfy~\eqref{LSI} due to the use of the Fr\'echet derivatives on the right-hand side. Nevertheless, the Poisson Point Process would satisfy~\eqref{LSI} provided the Fr\'echet derivatives are replaced with an oscillation, while the results of~\cite{GNO4,BellaFehrmanFischerOtto} would still be true assuming such modified \eqref{LSI} (with modifications in the corresponding proofs).

\ER


\medskip

The following almost uniform estimates (we think of $\alpha$ as being close to $1$ and $\varepsilon$ close to $0$) 
are an easy consequence of stationarity and the algebraic moment bounds \BR6 of $(\phi,\sigma)$ and $r_*$.\ER 

\begin{corollary}\label{r.star.log}
Let $0 < \alpha < 1$ be given. After a redefinition of $r_*$, retaining (\ref{T.2}), 
we may assume in addition that for $\langle\cdot\rangle$-almost every $a$ and every $x \in \R^d$ we have
\begin{align}
\frac1R \biggl(\fint_{|y-x|<R}|(\phi,\sigma)|^2 \dy \biggr)^\frac{1}{2}
\le\bigg(\frac{r_*(y)}{R}\bigg)^{\alpha} \quad\mbox{for all}\;R\ge r_*(y).\label{T.1}
\end{align}
Moreover, for any $\varepsilon > 0$ we have $\en{\cdot}$-almost surely 
\begin{equation}\label{rstar:bound}
\sup_{y \in \mathbb{Z}^d}\frac{r_*(y)}{(1+|y|)^\varepsilon}<\infty.
\end{equation}
\end{corollary}

The following statement is 
an immediate consequence of the ergodicity of $\langle \cdot \rangle$, 
normalization (\ref{ir14bis}) of the first-order correctors and the identity for $C$
\begin{align}\label{definition.C}
C_{ijk}= \en{ e_j \cdot \sigma_i \nabla\phi_k^* - e_j \cdot \sigma_k^* \nabla \phi_i }
\end{align}
(see Lemma~\ref{tensor.C} for the proof of~\eqref{definition.C}).

\begin{corollary}\label{ergodicity}
For $\langle \cdot \rangle$-almost every $a$ and for $i,j,k=1,\ldots, d$ we have the
distributional convergences as the rescaling parameter $R$ tends to infinity
\begin{align*}
\phi_i(R \, \cdot \, ) &\rightharpoonup 0,\\
\sigma_{ijk}(R \, \cdot \, ) &\rightharpoonup 0,\\
\bigl( e_j \cdot \sigma_i \nabla\phi_k^* - e_j \cdot \sigma_k^* \nabla \phi_i \bigr)( R \, \cdot \,) &\rightharpoonup C_{ijk},
\end{align*}
with the analogous convergences being true also in the transposed case. 
\end{corollary}

\ER

\BR3
Finally, we state the following (higher-order) regularity statement for $a$-harmonic functions:

\begin{proposition}{\bf ($C^{k,1}$-regularity)}\label{Ck.1} 
Let the coefficient field $a$ be such that $(\phi, \sigma)$ satisfying~\eqref{i5} exist and satisfy for some $r_* < +\infty$ and some $0 < \alpha < 1$ 
\begin{equation}\label{Ckeq}
 \frac1R \biggl(\fint_{|x|<R}|(\phi,\sigma)|^2\biggr)^\frac{1}{2}
\le\biggl(\frac{r_*}{R}\biggr)^{\alpha}\quad\mbox{for all}\;R\ge r_*.
\end{equation}
\BR3
Then 
the following holds:
\begin{itemize}
\item Let $k\geq 1$. For every $u_\h \in X_k^\h$ there exists a corrector $\psi_{u_\h}^{(k)}$ such that $( 1 + \phi_i\partial_i )u_\h + \psi_{u_\h}^{(k)} \in X_k$ 
\BR6 and for a constant $C=C(d,\lambda,k,\alpha) < \infty$
\begin{align}\label{higher.correctors}
\sup_{R \geq r_*}\biggl( \frac {R}{ r_*}\biggr)^{1-k}\biggl( \fint_{|x|<R}|\nabla((1+\phi_i\partial_i)u_\h + \psi_{u_\h}^{(k)})|^2 \biggr)^\oh \leq C 
\biggl( \fint_{|x|<r_*}|\nabla ((1+\phi_i\partial_i)u_\h + \psi_{u_\h}^{(k)})|^2 \biggr)^\oh.
\end{align}
\ER
Here, we use the understanding that $\psi^{(1)}_{u_\h} \equiv 0$ and that $\psi^{(2)}_{u_\h}$ coincides with $\psi_u$ obtained via \eqref{ir13}.
\item For $R \geq r_*$, let $u$ be an $a$-harmonic function in $\{|x| < R \}$. Then for every \BR5 $k \geq 0$ \ER there exists a constant $C= C(d, \lambda, k, \alpha) < +\infty$ such that
\begin{equation}\label{eqCk1}
\inf_{u_\h \in X_k^\h}\biggl(\fint_{|x| < r}\hspace{-0.3cm}|\nabla\bigl(u - ( (1 + \phi_i \partial_i) u_\h + \psi_{u_\h}^{(k)}) \bigr)|^2 \biggr)^\oh \leq C\biggl(\frac{r}{R} \biggr)^{\BR5 (k+1)-1}\hspace{-0.1cm}\biggl( \fint_{|x|<R}\hspace{-0.2cm}|\nabla u|^2 \biggr)^\oh
\end{equation} 
for all $r_* \leq r \leq R$, \BR6 with the understanding that $\psi_{u_\h}^{(0)}\equiv 0$. \ER
\item It holds that \BR6
\begin{equation}\label{Fischer1}
X_k= \bigl\{ (1+ \phi_i \partial_i)u_\h + \psi_{u_\h}^{(k)} \ | \ u_\h \in X_k^\h \bigr\}\quad \textrm{ and } \quad {\rm dim}X_k={\rm dim}X_k^\h.
\end{equation}
\end{itemize}
\end{proposition}

\medskip
Proposition~\ref{Ck.1} is a direct consequence of \cite[Theorem 1 \& Lemma 17]{FischerOtto}, which hold under weaker assumption of $(\phi,\sigma)$ having sublinear \BR6 growth quantified by just a logarithm~\cite[assumption (4)]{FischerOtto}\ER. Indeed, \cite[(10) in Theorem 1]{FischerOtto} provides an estimate on $\psi_{u_\h}^{(k)}$, which combined with the estimate of the first-order corrector $\phi$ \eqref{Ckeq} and the homogeneity of $u_\h$ yields~\eqref{higher.correctors}, while the regularity estimate~\eqref{eqCk1} is contained in~\cite[Lemma 17]{FischerOtto}. The last statement~\eqref{Fischer1} then follows directly from~\eqref{eqCk1}. 

\medskip

In the next lemma we discuss some properties of the tensor $C$ as defined in~\eqref{ir14} by
$$
C_{ijk} = \langle e_k\cdot \bigl( a\nabla \psi_{ij} + (\phi_ia-\sigma_i)e_j \bigr) \rangle
$$
and argue that it vanishes under the assumptions on the ensemble $\langle \cdot \rangle$ of Theorem \ref{main.sym}.
This, as explained in Remark~\ref{symmetrization}, implies that in this case Theorem \ref{main} actually reduces to Theorem \ref{main.sym}. We postpone the proof of this result to the Appendix~\ref{subsect_lemma1}.

\medskip

\begin{lemma}\label{tensor.C}
Let $\langle \cdot \rangle$ be a stationary ensemble satisfying an LSI.
\BR5 Then the coefficient $\ah \in \R^{d \times d}$ defined in~\eqref{ir12} satisfies
\begin{equation}\nonumber
  \forall\;\xi\in\mathbb{R}^d\quad \xi\cdot \ah\xi \geq \begin{cases}
                      \lambda|\xi|^2\\
                       |\ah\xi|^2,
                       \end{cases}
\end{equation}
\ER 
and the tensor $C \in \mathbb{R}^{d \times d\times d}$ defined above satisfies
\begin{align}\label{C.bdd}
|C| \lesssim 1,
\end{align}
and  may be rewritten as in \eqref{definition.C}. Furthermore
\begin{itemize}
 \item[i)] the symmetrizations of $C$ and $C^*$ in all the indices satisfy
 $$
 C^{*, sym}= - C^{sym}.
 $$
 This in particular implies that if the ensemble $\langle \cdot \rangle$ is supported only on symmetric coefficient fields,
 then $C^{sym}= C^{*,sym} = 0$.
 \item[ii)] If the ensemble is invariant under central reflections, then $C = 0$.
\end{itemize}
\end{lemma}
\begin{remark}\label{remark.tensor.C}
The tensor $C$ has a natural interpretation: As we show in the appendix, $\en{\cdot}$-almost surely, we have
\begin{equation}\label{li01}
e_j\cdot\Big(\big((x_i+\phi_i)a^* (e_k+\nabla\phi_k^*)-(x_k+\phi_k^*)a (e_i+\nabla\phi_i)\big)
            -\big( x_i        \ah^*e_k                - x_k          \ah e_i              \big)\Big)
\rightharpoonup C_{ijk},
\end{equation}
where the weak convergence is in the sense of Corollary~\ref{ergodicity}, that means that (large-scale) averages of the left-hand side converge to $C$.
This characterization is natural in our context, since the first term on the left-hand side in (\ref{li01}),
namely the divergence-free vector field $(x_i+\phi_i)a^* (e_k+\nabla\phi_k^*)$ 
$-(x_k+\phi_k^*)a (e_i+\nabla\phi_i)$, would come up if it would make sense to define an
invariant for the functions $x_i+\phi_i$ and $x_k+\phi_k^*$, which are $a$-harmonic and $a^*$-harmonic, respectively.
Likewise, the second term on the left-hand side of (\ref{li01}) is just the $\ah$ counterpart of the
first term. 
\end{remark}

\section{Deterministic results}
%

Throughout this section we assume that for a fixed coefficient field $a$ the correctors $(\phi,\sigma), (\phi^*,\sigma^*)$ and $(\psi,\Psi), (\psi^*,\Psi^*)$ exist and satisfy the properties enumerated in Proposition~\ref{stoch.results} and Proposition~\ref{Ck.1}.
Without loss of generality, we assume that 
\begin{equation}\label{psi.avgzero}
\fint_{|x|<r_*}(\psi,\Psi) = \fint_{|x|<r_*}(\psi^*,\Psi^*) =0,
\end{equation}
\BR5 where we write $r_* = r_*(0)$ for conciseness.\ER

\medskip

It is convenient to introduce an abbreviation for the second-order two-scale expansion in \eqref{o04}, \eqref{o05} of Theorem \ref{main}: For $u_\h \in X_m^\h$, $m\ge 2$, and $v_\h \in Y_k^\h$, $k\ge 1$, we write
\begin{equation}\label{m15}
 \begin{aligned}
 Eu_\h:&=(1+\phi_i\partial_i+\psi_{ij} \partial_{ij})u_\h+(1+\phi_i\partial_i)\ut,\\
Ev_\h:&=(1+\phi_i^*\partial_i+\psi_{ij}^*\partial_{ij})v_\h+(1+\phi_i^*\partial_i)\vt, 
 \end{aligned}
\end{equation}
where $\ut$ is associated to $u_\h$ according to \eqref{u.tilde}, \eqref{u.tilde.condition} and $\vt$ to $v_\h$ according to \eqref{v.tilde}, \eqref{v.tilde.condition} (see Lemma \ref{Rom1} and Lemma \ref{Rom2} for their construction). 

\medskip

We also introduce the following abbreviations.
For any (possibly non-integer) $m,k\ge 0$ we define the following semi-norms
\begin{align}
\|u\|_{m}&:=\sup_{R\ge r_*}\bigg(\frac{r_*}{R}\bigg)^{m-1}\bigg(
\fint_{|x| < R}|\nabla u|^2\bigg)^\frac{1}{2},\label{m01}\\
\|v\|_{k}&:=\sup_{R\ge r_*}\bigg(\frac{R}{r_*}\bigg)^{k+(d-2)+1}\bigg(
\fint_{|x|>R}|\nabla v|^2\bigg)^\frac{1}{2},\label{m02}
\end{align}
%
for growing functions $u$ and decaying functions $v$,
where it will always be clear from the context whether we mean the first or the second expression.
We note that $\|u\|_{m}$ becomes stronger with decreasing $m$ while $\|v\|_k$ becomes stronger with
increasing $k$.
We note that for integer $m\ge 1$, (\ref{m01}) is a semi-norm on $X_m$ and for integer $k\ge 0$, (\ref{m02})
is a semi-norm on $Y_k$. For any $r\ge r_*$ we introduce the spaces $Y_k(r) \subset Y_k$ constituted by the elements of $Y_k$ which are $a^*$-harmonic in $\{ |x| > r \}$ and need the following modification of $\|\cdot\|_{k}$
\begin{align}\label{m02.bis}
\|v\|_{k,r}&
:=\sup_{R\ge r}\bigg(\frac{R}{r}\bigg)^{k+(d-2)+1}\bigg(
\fint_{|x| > R}|\nabla v|^2\bigg)^\frac{1}{2}.
\end{align}

\medskip

\begin{theorem}\label{Rom3}
Let assumption \eqref{T.2} be satisfied for some $1 < \beta < 2$, 
and let assumption~\eqref{T.1} be satisfied with $\max(1/2,\beta-1) < \alpha < 1$, both (only) at $y=0$.  
Then we have for any integer $m\ge 2$: 

\smallskip

Relating $u_\h\in X_m^\h$ and $u\in X_m$ by
\begin{align}\label{m03}
\lim_{R\uparrow\infty}\frac{1}{R^{(m-1)-1}}\bigg(
\fint_{|x|<R}|\nabla(u-Eu_\h)|^2\bigg)^\frac{1}{2}=0
\end{align}
defines an isomorphism between the normed linear spaces $X_m^\h/X_{m-2}^\h$ and $X_m/X_{m-2}$.
%
%
It yields two linear maps
\begin{align*}
&L_X^\h: X_m^\h \rightarrow X_m, \ \ \ \ u_\h \mapsto u \text{\ \ \ with\ \ } \|u-Eu_\h\|_{m-\beta}\leq C \|u_\h\|_m,\\
&L_X: X_m \rightarrow X_m^\h, \ \ \ \ u \mapsto u_\h \text{\ \ \ with\ \ } \|u-Eu_\h\|_{m-\beta}\leq C\|u\|_m,
\end{align*}
with the constant $C$ depending on the dimension $d$, the ellipticity ratio $\lambda$, the order $m$ and the exponents $\alpha$ and $\beta$.

\smallskip

Relating $v_\h\in Y_{m-1}^\h$ and $v\in Y_{m-1}$ by
\begin{align}\label{m04}
\lim_{R\uparrow\infty}R^{(d-1)+ (m-1) +1}\bigg(
\fint_{|x|>R}|\nabla(v-Ev_\h)|^2\bigg)^\frac{1}{2}=0
\end{align}
defines an isomorphism between the normed linear spaces $Y_{m-1}^\h/Y_{m+1}^\h$ and $Y_{m-1}/Y_{m+1}$.
%
%
For any $r\ge r_*$, it yields two linear maps
\begin{align*}
&L_Y^\h: Y_{m-1}^\h(r) \rightarrow Y_{m-1}(r), \ \ \ \ v_\h\mapsto v \text{\ \ \ with\ \ } \|v-Ev_\h\|_{m-1+\beta,r} \leq C \|v_\h\|_{m-1,r},\\
&L_Y: Y_{m-1}^\h(r) \rightarrow Y_{m-1}(r), \ \ \ \ v\mapsto v_\h \text{\ \ \ with\ \ } \|v-Ev_\h\|_{m-1+\beta,r} \leq C\|v\|_{m-1,r},
\end{align*}
with the constant $C$ \BR5 as above. \ER

\smallskip

Finally, we have that the two isomorphisms preserve the bilinear forms: $\,(u,v)=(u_\h,v_\h)_\h$.
\end{theorem}


A couple of semantic comments are in place: Given the natural semi-norms $\|\cdot\|_m$ and $\|\cdot\|_k$,
defined in (\ref{m01}) and (\ref{m02}), on the spaces $X_m$ and $Y_k$ (as well as $X_m^\h$ and $Y_k^\h$), respectively,
we endow the quotient spaces with the induced norm. For the quotient space $X_m/X_{m-2}$ we consider the norm
$\inf_{w\in X_{m-2}}\|u-w\|_{m}$, which in fact is a norm and not just a semi-norm.
Saying that the relation (\ref{m03}) between elements
of $X_m^\h$ and elements of $X_m$ defines an isomorphism between the spaces $X_m^\h/X_{m-2}^\h$ and $X_m/X_{m-2}$,
amounts to the following fact: For every $u_\h\in X_m^\h$ there exists a $u\in X_m$, unique up to an element in $X_{m-2}$,
such that (\ref{m03}) holds and, likewise, for every $u\in X_m$, there exists a $u_\h\in X_m^\h$, 
unique up to an element in $X_{m-2}^\h$, such that (\ref{m03}) holds. The theorem claims that this bijection is
linear (which is obvious since by the triangle inequality the relation (\ref{m03}) is compatible with the linear structure) 
and that it is bounded with bounded inverse. 
>From the point of view of quantitative homogenization, among the
four quantitative estimates the most pertinent is that given $\BR6 v\in Y_{k}$, $k\ge 1$, (i.e. a decaying $a$-harmonic
function), there exists $\BR6 v_\h \in Y_{k}^\h$ (i.e. a constant-coefficient harmonic function of the same order of decay),
so that \BR6 the difference between $v$ and the correction $Ev_\h$ of $v_\h$ \ER is small with a relative decay rate of $\beta$.

\medskip


\begin{theorem}\label{T3}
\BR5 Let assumption \eqref{T.2} be satisfied for some $1 < \beta < 2$, and let assumption~\eqref{T.1} be satisfied with $\max(1/2,\beta-1) < \alpha < 1$, both at $y$ and at $y=0$. \ER 
Provided $|y| \ge \BR5 4r_*(0)$, for any integer $m \ge 2$ the linear maps $L_X^\h$, $L_X$ of Theorem \ref{Rom3} satisfy 
\begin{align}
\biggl(\fint_{|x-y| < r_*(y)}|\nabla( u - Eu_\h)|^2\biggr)^\frac{1}{2}&\le C \biggl(\frac{|y|}{\BR5 r_*(0) }\biggr)^{m - 1 - \beta} || u_\h||_{m},\label{growth.loc}\\
\biggl(\fint_{|x-y| < r_*(y)}|\nabla( u - Eu_\h)|^2\biggr)^\frac{1}{2}&\le C \biggl(\frac{|y|}{\BR5 r_*(0) }\biggr)^{m - 1- \beta} ||u||_m,\label{growth.loc2}
\end{align}
with the constant $C$ depending on the dimension $d$, the ellipticity ratio $\lambda$, the order $m$ and the exponents $\alpha$ and $\beta$.

\smallskip

Similarly, for $r\geq r_*(0)$ if 
$y$ satisfies $|y| \geq 2( r_*(y) + r)$, then the linear maps $L_Y^\h, L_Y$ are such that 
\begin{align}
\biggl(\fint_{|x-y| < r_*(y)}|\nabla(v - Ev_\h)|^2\biggr)^\frac{1}{2}&\le C \biggl(\frac{\BR5 r}{ |y| }\biggr)^{(d -1) + m-1 + \beta} || v_\h||_{m-1,r},\label{dec.loc}\\
\biggl(\fint_{|x-y| < r_*(y)}|\nabla( v - E v_\h)|^2\biggr)^\frac{1}{2}&\le C \biggl(\frac{\BR5 r}{ |y| }\biggr)^{(d -1) + m-1 + \beta}|| v||_{m-1,r},\label{dec.loc2}
\end{align}
with the constant $C$ as above. 
\end{theorem}

\section{Abstract results on the spaces \texorpdfstring{$X_m$}{} and \texorpdfstring{$Y_k$}{}}\label{abstract.r}
In this section we give the auxiliary results which are the building blocks for Theorem \ref{Rom3}. For the sake of simplicity, we assume that $a$ is symmetric: The results of this section immediately extend to non-symmetric $a$'s with no modification besides the appearance of $a^*$, $\ah^*$ in the equations related to the elements in the family of spaces $\{ Y_k \}_k, \{ Y_k^\h \}_k$.

\medskip

\BR3 We will show (see Proposition~\ref{Pdual}) that \ER 
the bilinear form $( \cdot ,\cdot )$ defined by \eqref{bil.def} of Section \ref{main.r}
provides an isomorphism between the two scales of spaces $\{X_m\}_m$
and $\{Y_k\}_{k}$. More precisely, it provides a canonical isomorphism
between the quotient space $Y_k/Y_{m+1}$ and the dual $(X_m/X_{k-1})^*$ of the quotient
space $X_m/X_{k-1}$ for all $m\ge k\ge 1$. Even more precisely,
this canonical isomorphism is defined by associating to a $v\in Y_k$
the linear form $ X_m\ni u \mapsto ( u, v )$. By (\ref{o02}) this form vanishes on
$X_{k-1}$ and thus can be (canonically) identified with an element of $(X_m/X_{k-1})^*$.
Again by (\ref{o02}), this form vanishes on $Y_{m+1}$ so that the (linear) map
$Y_k\ni v\mapsto ( \cdot,v )$ lifts to a map on $Y_k/Y_{m+1}$. It is this
map we claim is an isomorphism.

\medskip

Combining the isomorphism between $Y_k/Y_{k+1}$ and $(X_k/X_{k-1})^*$ with 
a canonical isomorphism between $\Xk / X_{k-1}$ and $\Xke / X_{k-1}^{\mathrm{h}}$, which follows from higher-order Liouville principles obtained in \cite[Lemma 19]{FischerOtto}, it follows that for any $k \ge 1$
\begin{align}\label{approx.decay}
Y_k/Y_{k+1}\cong \Yke / Y_{k+1}^{\mathrm{h}}.
\end{align}

\subsection{General setting and assumptions of this section}

Before giving the statements of this section, we stress that they
can be formulated and proven in greater generality (which we will do without changing
the notation).
The Euclidean space $\mathbb{R}^d$ may be replaced by any differentiable manifold endowed
with a measure and a metric that behave
like $\mathbb{R}^d$ on scales larger than some radius $r_*$ in the sense of
the following two properties:
\begin{itemize}
\item Volume control for large balls centered at zero
\begin{align}\label{vol}
\BR5 \frac{1}{C_0} R^d \le |\{ |x| < R \}|\leq C_0 R^d\quad\mbox{for all}\;R\ge r_*,
\end{align}
for some constant $C_0 < +\infty$ and a fixed exponent $d$ (which does not have to be an integer).
\item Poincar\'e inequality with mean-value zero on large dyadic annuli centered at zero: For a constant $C_0 < +\infty$
\begin{align}\label{poinc}
\inf_{c\in\mathbb{R}\ER }\biggl(\int_{R < |x| < 2R}|u-c|^2\biggr)^\frac{1}{2}&\leq \BR5 C_0 \ER R\biggl(\int_{R < |x| < 2R}|\nabla u|^2\biggr)^\frac{1}{2}\\
&\mbox{for}\;R\ge r_*
\quad\mbox{and functions }\;u.\nonumber
\end{align}
\end{itemize}
Here and in the sequel, all balls are centered at the origin and refer to the Riemannian metric under consideration.
In addition to this metric we consider another (tensor)-field $a$, which gives rise to $a$-harmonic functions (solutions of $-\nabla \cdot a \nabla u = 0$) on the manifold. 

\medskip

The main structural assumption \BR6 on the coefficient field $a$ and the underlying Riemannian manifold \ER is the following: For a strictly increasing sequence of real numbers
\begin{align}\label{F1bis}
0, 1<\cdots<k<\cdots
\end{align}
there exists a nested sequence of finite-dimensional subspaces \BR5 of $a$-harmonic functions \ER
\begin{align*}
X_0:=\{const\}\subset X_1\subset\cdots X_k\subset \cdots \subset \BR5 \{ a\textrm{-harmonic functions on  } \R^d \}
\end{align*}

such that the following two properties are satisfied:
\begin{itemize}
\item \BR5 For every $k\ge 0$\ER, there exists a finite constant $C_1$ depending on $k$, the ellipticity contrast $\lambda$, the exponent $d$ in \eqref{vol} and the constants in \eqref{vol}, \eqref{poinc}, such that for any two radii $R\ge r\ge r_*$ and any $a$-harmonic function $u$ in $\{ |x| < R \}$ we have
\begin{align}\label{F1}
\inf_{v\in X_k}\biggl(\fint_{|x|< r}|\nabla(u-v)|^2\biggr)^\frac{1}{2}\leq C_1 
\biggl(\frac{r}{R}\biggr)^{(k+1)-1}\biggl(\fint_{|x| < R}|\nabla u|^2\biggr)^\frac{1}{2},
\end{align}
where $k+1$ stands for the next largest number in the sequence (\ref{F1bis}).
\item For any $k\ge 1$, there exists a finite constant $C_2$ depending on $k$, the ellipticity contrast $\lambda$, the exponent $d$ in \eqref{vol} and the constants in \eqref{vol}, \eqref{poinc} such that for any $u\in X_k$ and two radii $R\ge r\ge r_*$ we have
\begin{align}\label{F2}
\biggl(\fint_{|x| < R}|\nabla u|^2\biggr)^\frac{1}{2}\leq C_2 
\biggl(\frac{R}{r}\biggr)^{\BR6 -1+k}\biggl(\fint_{|x|< r}|\nabla u|^2\biggr)^\frac{1}{2}.
\end{align}
\end{itemize}

\BR6
It follows from~\eqref{Fischer2} in~Lemma~\ref{Lup} below that these spaces $X_k$ coincide with the ones defined in the introduction. 
\ER

\medskip

\BR6
The core assumptions~\eqref{F1} and \eqref{F2} are very natural: \ER For example, as can be read from Proposition~\ref{Ck.1}, 
in the case of the Euclidean space $\R^d$ (equipped with standard Euclidean metric) and for a large class of coefficient fields $a$ on $\R^d$, as considered in Proposition~\ref{Ck.1}, \eqref{F1} follows from~\eqref{eqCk1} and~\eqref{Fischer1}, while~\eqref{F2} follows from~\eqref{Fischer1} and~\eqref{higher.correctors}. 

\ER


\subsection{The abstract results}

In all the following results $A \lesssim B$ stands for $A\le C B$ with a constant $C$ that only depends on
the ellipticity contrast $\lambda$, the exponent $d$ and the constant appearing in the volume bound (\ref{vol}), 
the constant appearing in the Poincar\'e inequality (\ref{poinc}), 
and the hypotheses (\ref{F1}) and (\ref{F2}) for the range of
$k$'s in (\ref{F1bis}) under consideration. When applicable, it also depends on the exponent $\beta$, more precisely, its distance
to the values in (\ref{F1bis}). 
 
\begin{proposition}\label{Pg}
Consider the exponents $m\ge k\ge 1$ from \eqref{F1bis} and $\beta\in(k-1,k)$ (where $k-1$ denotes the predecessor of $k$ in
(\ref{F1bis})) and a radius $r\ge r_*$.
We are given a function $u$ and a vector field $g$ such that
\begin{align}\label{Pg2}
-\nabla\cdot a\nabla u=\nabla\cdot g
\end{align}
satisfying the growth conditions
\begin{align}\label{Pg3}
\biggl(\fint_{|x| < R}|g|^2\biggl)^\frac{1}{2}\le\biggl(\frac{R}{r}\biggl)^{\beta-1},\quad
\biggl(\fint_{|x| < R}|\nabla u|^2\biggl)^\frac{1}{2}\le\biggl(\frac{R}{r}\biggl)^{m-1}\quad\mbox{for all}\;R\ge r.
\end{align}
Then there exist $w\in X_m$ such that for all $R\ge r$
\begin{align}\label{Pg4}
\biggl(\fint_{|x| < R}|\nabla(u-w)|^2\biggl)^\frac{1}{2}\lesssim\biggl(\frac{R}{r}\biggl)^{\beta-1},\quad
\biggl(\fint_{|x| < R}|\nabla w|^2\biggl)^\frac{1}{2}\lesssim\biggl(\frac{R}{r}\biggl)^{m-1}.
\end{align}
Moreover, if there is another $w'\in X_m$ with these properties we have $w-w'\in X_{k-1}$ and
\begin{align}\label{Pg5}
\biggl(\fint_{|x| < R}|\nabla(w-w')|^2\biggl)^\frac{1}{2}\lesssim\biggl(\frac{R}{r}\biggl)^{(k-1)-1}\quad\mbox{for all}\;R\ge r.
\end{align}
\end{proposition}

\begin{proposition}\label{Pd}
Consider exponents $m\ge k\ge 1$ and $\beta\in(m,m+1)$ 
(where $m+1$ denotes the successor of $m$ in the sequence (\ref{F1bis})) and a radius $r\ge r_*$. 
We are given a function $u$ and a vector field $g$ such that
\begin{align*}
-\nabla\cdot a\nabla u=\nabla\cdot g\quad\textrm{ in }\;\{ |x| > r \}
\end{align*}
satisfying the decay conditions
\begin{align}\label{Pd2}
\biggl(\fint_{|x| > R}|g|^2\biggl)^\frac{1}{2}\le\biggl(\frac{r}{R}\biggl)^{\beta+d-1},\quad
\biggl(\fint_{|x| > R}|\nabla u|^2\biggl)^\frac{1}{2}\le\biggl(\frac{r}{R}\biggl)^{k+d-1}\quad\mbox{for all}\;R\ge r.
\end{align}
Then there exists $w\in Y_k(r_*)$ such that for all $R\ge 2r$
\begin{equation}\label{Pd18}
\begin{aligned}
\biggl(\fint_{|x| > R}|\nabla(u-w)|^2\biggl)^\frac{1}{2}\lesssim \Bigl( \frac{r}{r_*} \Bigr)^{m-k}\left(\frac{r}{R}\right)^{\beta+d-1},\\
\biggl(\fint_{|x| > R}|\nabla w|^2\biggl)^\frac{1}{2}\lesssim \Bigl( \frac{r}{r_*} \Bigr)^{m-k} \left(\frac{r}{R}\right)^{k+d-1}.
\end{aligned}
\end{equation}

Moreover, if there is another $w'\in Y_k(r_*)$ with these properties we have $w-w'$ $\in Y_{m+1}(r_*)$ and
\begin{align*}
\biggl(\fint_{|x| > R}|\nabla(w-w')|^2\biggl)^\frac{1}{2}\lesssim\Bigl( \frac{r}{r_*} \Bigr)^{m-k}\Bigl(\frac{r}{R}\Bigr)^{(m+1)+d-1}\quad\mbox{for all}\;R\ge 2r.
\end{align*}
\end{proposition}

Whereas Proposition \ref{Pg} associates with $(u,g)$ an element of $X_m/X_{k-1}$,
Proposition \ref{Pd} associates an element of the dual space $Y_k(r_*)/Y_{m+1}(r_*)$.

\medskip
Another fundamental result of this section is the following statement, which expresses and quantifies the isomorphism provided by the bilinear map $(\cdot,\cdot)$:
\begin{proposition}\label{Pdual}
i) The bilinear map $(\cdot,\cdot)$ provides for $1\le k\le m$ an isomorphism
\begin{align*}
Y_k/Y_{m+1}\cong(X_m/X_{k-1})^*,
\end{align*}
where $k-1$ denotes the predecessor of $k$ and $m+1$ the successor of $m$ in the sequence (\ref{F1bis}).

\smallskip 

ii) This isomorphism is quantitative in the following sense: For every $k\le n\le m$ we may select a complement
$Z_n(r_*)$ of $Y_{n+1}(r_*)$ in $Y_n(r_*)$ (hence the direct sum $Z_k(r_*)\oplus \cdots \oplus Z_m(r_*)$ is a complement of
$Y_{m+1}(r_*)$ in $Y_k(r_*)$ and thus isomorphic to $Y_k(r_*)/Y_{m+1}(r_*)$) in such a way that for every linear
form $\ell$ on $X_m$ that vanishes on $X_{k-1}$ (note that $(X_m/X_{k-1})^*$ is canonically isomorphic
to the space of these $\ell$'s) there exists a unique $w\in Z_k(r_*)\oplus \cdots \oplus Z_m(r_*)$ such that
$\ell = (\cdot,w)$ on $X_m$ and that for all $r\ge r_*$
\begin{align}\label{Pd6}
\bigg(\fint_{|x|  > r}|\nabla w|^2 \biggr)^\oh
\lesssim \Bigl(\frac{r}{r_*} \Bigr)^{m-k} \sup_{u\in X_m}\frac{\frac{1}{r^d}\ell.u }{\bigl( \fint_{|x|<r}|\nabla u|^2\bigr)^\oh }.
\end{align}

\smallskip

iii) Moreover, the complements $Z_n(r_*)$, $n=k,\ldots,m$, are natural in the sense that for every $r\ge r_*$,
we also have the direct sum $Y_k(r)=Z_k(r_*)\oplus \cdots \oplus Z_m(r_*) \oplus Y_{m+1}(r)$ and that the 
corresponding projections are continuous in the sense that if $w=w_k+\cdots+w_{m}+\tilde w$ is a decomposition then
\begin{align}\label{Pd21}
\bigg(\int_{|x| > r}|\nabla w_k|^2\bigg)^\frac{1}{2}+\cdots+\bigg(\int_{|x| > r}|\nabla w_m|^2\bigg)^\frac{1}{2}
\lesssim C(k,m) \bigg(\int_{|x| > r}|\nabla w|^2\bigg)^\frac{1}{2}.
\end{align}
Here, the constant $C(k,m) < +\infty$ depends only on the number of indices in~\eqref{F1bis} between $k$ and $m$.

\smallskip

iv) Finally, elements $v\in Z_n(r_*)$, $n=k,\ldots,m$, display homogeneous behavior in the sense of
\begin{equation}\label{Pd20}
\bigg(\fint_{|x| > R}|\nabla v|^2\bigg)^\frac{1}{2}\sim\Big(\frac{r}{R}\Big)^{\BR5 n + d - 1}\bigg(\fint_{|x| > r}|\nabla v|^2\bigg)^\frac{1}{2}
\quad\mbox{for all}\;R\ge r\ge r_*.
\end{equation}
\end{proposition}

\BR5 We remark that in our construction, the complement $Z_n(r_*)$ of $Y_{n+1}(r_*)$ in $Y_n(r_*)$ for $n=k,\ldots,m$ depends on $m$ (and not just on $n$ as the notation suggests). \ER

\medskip

\BR5
The following proposition is a dual version of Proposition~\ref{Pdual}: \ER

\begin{proposition}\label{Pbidual}
Let $m \ge k\ge 1$. 
i) The bilinear form $(\cdot,\cdot)$ provides an isomorphism
\begin{align*}
X_m/X_{k-1}\cong(Y_k/Y_{m+1})^*.
\end{align*}

\smallskip

ii) This isomorphism can be made quantitative by constructing complements $W_n$ of $X_{n-1}$ in $X_n$,
$n=k,\ldots,m$, in such a way that for $u\in W_k\oplus \cdots \oplus W_m$ and $r\ge r_*$ we have
\begin{align*}
\bigg(\fint_{|x|<r}|\nabla u|^2\bigg)^\frac{1}{2}
\lesssim\Bigl( \frac{r}{r_*} \Bigr)^{m-k}\sup_{v\in Z_k(r_*)\oplus\cdots \oplus Z_m(r_*)}\frac{\frac{1}{r^d}(u,v)}{\big(\fint_{|x|>r}|\nabla v|^2\big)^\frac{1}{2}}.
\end{align*}

\smallskip

iii) These complements $\{W_n\}_n$ are compatible with the complements $\{Z_m(r_*)\}_m$ in the sense of
\begin{align*}
(W_n,Z_m(r_*))=0\quad\mbox{for}\;n\not=m.
\end{align*}

\smallskip

iv) These complements are natural in the sense that the projections $X_m\ni u$ $\mapsto(u_m,\ldots,u_k)$
$\in W_m\times\cdots\times W_k$ are continuous:
\begin{align*}
\bigg(\int_{|x|<r}|\nabla u_m|^2\bigg)^\frac{1}{2}+\cdots+\bigg(\int_{|x|<r}|\nabla u_k|^2\bigg)^\frac{1}{2}
\lesssim C(k,m)\bigg(\int_{|x|<r}|\nabla u|^2\bigg)^\frac{1}{2}\quad\mbox{for all}\;r\ge r_*,
\end{align*}
with $C(k,m) < +\infty$ depending only on the number of indices in~\eqref{F1bis} between $k$ and $m$.

\smallskip

v) The elements $u\in W_n$ have homogeneous behavior in the sense of
\begin{equation}\label{Wnhomog}
\bigg(\fint_{|x|<R}|\nabla u|^2\bigg)^\frac{1}{2}
\sim\Big(\frac{R}{r}\Big)^{n -1 }\bigg(\fint_{|x|<r}|\nabla u|^2\bigg)^\frac{1}{2}
\quad\mbox{for all}\;R\ge r\ge r_*.
\end{equation}

\end{proposition}

\medskip

A consequence of the two previous propositions is the following:

\begin{corollary}\label{Cor5.b}
For all $m\ge 2$ 
the induced projection $X_m\ni u\mapsto u'\in W_m \oplus W_{m-1}$ has the following boundedness property
\begin{align}\label{m07}
\|u'\|_m + \|u-u'\|_{m-2}\lesssim \|u\|_m
\end{align}
\BR5(see~\eqref{m01} for definition of $\|\cdot\|_{m}$). \ER 
%

\smallskip

For all $k\ge 1$ and $r \geq r_*$ 
the induced projection $Y_k(r)\ni v\mapsto v'\in Z_k(r_*)\oplus Z_{k+1}(r_*)$ has the following boundedness property
\begin{equation}\label{eqC301}
\|v'\|_{k,r} + \|v-v'\|_{k+2,r}\lesssim \|v\|_{k,r}
\end{equation}
\BR5(see~\eqref{m02.bis} for definition of $\|\cdot\|_{k,r}$). \ER 
%
\end{corollary}

With the next statements we outline the main ingredients needed to prove Propositions \ref{Pg} and \ref{Pd} and Propositions \ref{Pdual} and \ref{Pbidual}.
We start with Lemma \ref{Lg}, on which Proposition~\ref{Pg} mainly relies, and Lemma \ref{Ld}, which together with  Proposition \ref{Pdual} is the main ingredient for \BR6Proposition~\ref{Pd}: \ER

\begin{lemma}\label{Lg}
Consider $k\ge 1$ and $\beta\in(k,k+1)$ (where $k + 1$ denotes the successor of $k$ in the sequence
(\ref{F1bis})) and the radius $r\ge r_*$. We are
given a vector field $g$ such that
\begin{align}\label{Lg1}
\bigg(\fint_{|x| < R}|g|^2\bigg)^\frac{1}{2}\le\bigg(\frac{R}{r}\bigg)^{\beta-1}\quad\mbox{for all}\;R\ge r.
\end{align}
Then there exists a function $u$ with
\begin{align}\label{Lg6}
-\nabla\cdot a\nabla u=\nabla\cdot g
\end{align}
and such that
\begin{align}\label{Pg5bis}
\bigg(\fint_{|x| < R}|\nabla u|^2\bigg)^\frac{1}{2}\lesssim\bigg(\frac{R}{r}\bigg)^{\beta-1}\quad\mbox{for all}\;R\ge r.
\end{align}
\end{lemma}


\begin{lemma}\label{Ld}
Consider $k\ge 1$ and $\beta\in(k-1,k)$ and the radius $r\ge r_*$. We are
given a vector field $g$ such that
\begin{align}\label{Ld1}
\bigg(\fint_{|x|< r}|g|^2\bigg)^\frac{1}{2}\le 1,\quad
\bigg(\fint_{|x| > R}|g|^2\bigg)^\frac{1}{2}\le\bigg(\frac{r}{R}\bigg)^{\beta+d-1}\quad\mbox{for all}\;R\ge r
\end{align}
and the moment condition
\begin{align}\label{Ld4}
\int \nabla\tilde u\cdot g=0\quad\textrm{for all }\;\tilde u\in X_k.
\end{align}
Then there exists a function $u$ with
\begin{align}\label{Ld2}
-\nabla\cdot a\nabla u=\nabla\cdot g
\end{align}
and such that
\begin{align}\label{Ld12}
\bigg(\fint_{|x|< r}|\nabla u|^2\bigg)^\frac{1}{2}\lesssim 1,\quad
\bigg(\fint_{|x| > R}|\nabla u|^2\bigg)^\frac{1}{2}\lesssim\bigg(\frac{r}{R}\bigg)^{\beta+d-1}\quad\mbox{for all}\;R\ge r.
\end{align}
\end{lemma}

\medskip

Both Lemmas \ref{Lg} and \ref{Ld} rely on the upgrade of (\ref{F1}) contained in the next lemma, which contains also a Liouville-type statement for the spaces $X_k$. This lemma upgrades condition \eqref{F1} by showing that a minimizer $v\in X_k$ on the left-hand side of \eqref{F1} may be chosen to be \textit{independent} from the scale $r \ge r_*$.

\begin{lemma}\label{Lup}
For any \BR5 $k\ge 1$ \ER and radius $R\ge r_*$ we have: If $u$ is an $a$-harmonic function in $\{ |x| < R \}$ then there
exists $\tilde u\in X_k$, \BR6 which is independent of $R$ \ER and such that
\begin{align}\label{eqLup01}
\biggl(\fint_{|x|< r}|\nabla(u-\tilde u)|^2\biggr)^\frac{1}{2}&\lesssim\biggl(\frac{r}{R}\biggr)^{(k+1)-1}\biggl(\fint_{|x| < R}|\nabla u|^2\biggr)^\frac{1}{2}
&&\mbox{for all}\;r\in[r_*,R],
\\ 
\label{cw05}
\BR5 \biggl(\fint_{|x|<r}|\nabla \tilde u|^2\biggl)^\frac{1}{2} \BR6 \lesssim \ER \biggl(\fint_{|x|<r}|\nabla u|^2\biggl)^\frac{1}{2}
&\lesssim\biggl(\fint_{|x|<R}|\nabla u|^2\biggl)^\frac{1}{2}
&&\mbox{for all}\;r\in[r_*,R].
\end{align}
In addition, for $ m \geq k \geq 1$ and 
\BR5 any $a$-harmonic function $u$ in the whole space
\begin{align}\label{Fischer2}
 \lim_{R \uparrow +\infty} \frac{1}{R^{-1+k}}\biggl( \fint_{|x| < R} | \nabla u|^2 \biggr)^\oh = 0 \ \ \ \Leftrightarrow u \in X_{k-1},
\end{align}
\BR5 where $k-1$ is the predecessor of $k$ in~\eqref{F1bis}. 
\end{lemma}

\medskip

An ingredient for Proposition \ref{Pd} as well as Propositions \ref{Pdual} and \ref{Pbidual}
is given by the following
lemma, which is a dualization of our hypothesis (\ref{F1}).

\begin{lemma}\label{LF1}
Let $k\ge 1$, $r\ge r_*$ and $v$ be $a$-harmonic in $\{ |x| > r \}$ with $(\int_{|x| > r}|\nabla v|^2)^\frac{1}{2}$ $<\infty$. 
Then the following statements are equivalent:
\begin{itemize}
\item [a)] $v\in Y_k(r)$
\item [b)] $( \tilde u,v )=0$ for all $\tilde u\in X_{k-1}$,
where $k-1$ is the predecessor of $k$ in (\ref{F1bis})
\Item [c)]
\begin{equation}\label{Lext22}
\bigg(\fint_{|x| > R}|\nabla v|^2\bigg)^\frac{1}{2}\lesssim\bigg(\frac{r}{R}\bigg)^{k+d-1}\bigg(\fint_{|x| > r}|\nabla v|^2\bigg)^\frac{1}{2}
\quad\mbox{for all}\;R\ge 2r.
\end{equation}
\end{itemize}
\end{lemma}

\medskip

Using the previous statements we obtain:

\begin{corollary}\label{Lbil}
For $0\le k<m$ we have
\begin{align*}
( X_k,Y_m) =0.
\end{align*}
In addition, we also have for \BR6 $r \ge r_*$ and $v \in \BR6 \Y_1(r)$ \ER that
\begin{align}\label{Liouv.decay}
\lim_{R \uparrow +\infty} R^{d-1 + k}\biggl( \fint_{|x|> R} |\nabla v|^2 \biggr)^\oh =0 \Leftrightarrow v \in Y_{k+1}\BR6(r).
\end{align}

\end{corollary}

\medskip

Before we proceed with the proofs of the previous statements, we also give two last results: First, an elementary extension lemma, which relies on (\ref{poinc}), and which is a frequent tool, for
instance, in the proofs of Lemma \ref{Lbil} and of Proposition \ref{Pd}. Second, we present another lemma, which relies on the upgrade of the assumptions \eqref{F1} and \eqref{F2}  given by Lemma \ref{Lup}, and which will be needed in the proof of the main theorem (see also Proposition \ref{localization}).

\begin{lemma}\label{Lext}
Let a function $u$ and a vector field $g$ be such that
\begin{align}\label{Lext1}
-\nabla\cdot a\nabla u=\nabla\cdot g\quad\mbox{in}\;\{ |x| > r \}
\end{align}
and
%
\begin{align}
\mbox{the total flux of} &\mbox{ the divergence-free vector field}\;\;a\nabla u+g\;\;\mbox{vanishes}, i.e
.\notag\\
&\int_{|x| = r} \nu \cdot ( a\nabla u + g) = 0, \label{Lext27}\\
\mbox{OR}\quad&\sup_{R\ge r}\int_{R < |x| < 2R}|a\nabla u+g|<\infty\label{Lext2}
\end{align}
for some radius $r\ge r_*$. Then there exists a function $\bar u$ and a vector field $\bar g$ be such that
\begin{align}\label{Pd12}
-\nabla\cdot a\nabla \bar u=\nabla\cdot \bar g 
\end{align}
and which are extensions in the sense of
\begin{align}\label{Pd3}
(\nabla\bar u,\bar g)=(\nabla u,g)\quad\mbox{in}\;\{ |x| > 2r\}.
\end{align}
This extension is bounded in the sense of
\begin{align}\label{Pd1}
\bigg(\int_{ |x| < 2r}|\nabla\bar u|^2+|\bar g|^2\bigg)^\frac{1}{2}\lesssim
\bigg(\int_{r < |x| < 2r}|\nabla u|^2+|g|^2\bigg)^\frac{1}{2}.
\end{align}
\end{lemma}

\begin{lemma}\label{propositionlocalize}
Let $R \geq r_*$ and let us consider a function $w$ and a vector field $h$ solving
\begin{align}\label{ld.1b}
- \nabla \cdot a \nabla w = \nabla \cdot h \quad \text{ in $\{ |x| < R \}$},
\end{align}
and such that for some $\alpha > 0$
\begin{align}\label{properties}
\biggl( \fint_{|x|< R}|\nabla w|^2 \biggr)^\oh \leq 1, \quad  \biggl( \fint_{|x| < r} |h|^2 \biggr)^\oh \leq \biggl( \frac{r}{R} \biggr)^\alpha \quad \textrm{for } r_* \leq r \leq R.
\end{align}
Then we have 
\begin{align}\label{ld.4loc}
\biggl( \fint_{|x| < r_*} |\nabla w|^2 \biggr)^{\frac 1 2} \le C(\alpha).
\end{align}
\end{lemma}

\subsection{Proofs of Proposition~\ref{Pg} and~\ref{Pd}}\hspace*{\fill} 

\begin{proof}[Proof of Proposition~\ref{Pg}]

By Lemma \ref{Lg}, which we may apply since $g$ satisfies the appropriate decay estimate, see the \BR3 first assumption in \ER (\ref{Pg3}),
there exists a function $\tilde u$ with
\begin{align}\label{Pg1}
-\nabla\cdot a\nabla\tilde u=\nabla\cdot g
\end{align}
and such that
\begin{align*}
\bigg(\fint_{|x| < R}|\nabla\tilde u|^2\bigg)^\frac{1}{2}\lesssim\biggl(\frac{R}{r}\biggl)^{\beta-1}\quad\mbox{for all}\;R\ge r.
\end{align*}
We now consider $w:=u-\tilde u$, then (\ref{Pg5bis}) turns into the first estimate of (\ref{Pg4}).
It remains to argue that $w\in X_m$: By (\ref{Pg2}) and (\ref{Pg1}) $w$ is $a$-harmonic.
>From the second item in our assumption (\ref{Pg3}) and by (\ref{Pg5bis}), using $\beta\le k\le m$, we obtain
by the triangle inequality that (\ref{Pg4}) holds. The $a$-harmonicity of $w$ and estimate (\ref{Pg4}) imply
$w\in X_m$ by definition of the latter.

\medskip

We now turn to the uniqueness and (\ref{Pg5}). From the first item in (\ref{Pg4}) we obtain
\begin{align}\label{Pg6}
\bigg(\fint_{|x| < R}|\nabla(w-w')|^2\bigg)^\frac{1}{2}\lesssim\bigg(\frac{R}{r}\bigg)^{\beta-1}\quad\mbox{for all}\;R\ge r.
\end{align}
\BR5 Since $\beta < k$, by~\eqref{Fischer2} in Lemma~\ref{Lup} this implies $w-w' \in X_{k-1}$. \ER 
Then (\ref{Pg5}) follows from (\ref{Pg6}) with $R=r$ by
our hypothesis (\ref{F2}) (with $k-1$ playing the role of $k$).
\end{proof}


\medskip
\begin{proof}[Proof of Proposition~\ref{Pd}]

We note that by (\ref{Pd2}) and $\beta\ge m\ge k$ we have
\begin{align}\label{Lext11}
\bigg(\fint_{|x| > R}|a\nabla u+g|^2\bigg)^\frac{1}{2}\lesssim\bigg(\frac{r}{R}\bigg)^{k+d-1}\quad\mbox{for all}\;R\ge r.
\end{align}
By Jensen's inequality and (\ref{vol}) this yields
\begin{align*}
\int_{R < |x| < 2R}|a\nabla u+g|\lesssim R^d\bigg(\frac{r}{R}\bigg)^{k+d-1}\quad\mbox{for all}\;R\ge r,
\end{align*}
which by $k\ge 1$ in particular implies 
\begin{align}\label{Lext12}
\sup_{R\ge r}\int_{R < |x| < 2R}|a\nabla u+g|<\infty,
\end{align}
the second assumption of Lemma \ref{Lext}, see (\ref{Lext2}).
We thus may apply Lemma \ref{Lext}, which provides an extension $(\bar u,\bar g)$. 
By (\ref{Pd3}) and estimate (\ref{Pd1}),
our assumption (\ref{Pd2}) implies
\begin{align}\label{Pd16}
\bigg(\fint_{ |x| < \BR6 2r}|\bar g|^2\bigg)^\frac{1}{2}\lesssim 1,\quad
\bigg(\fint_{ |x| < \BR6 2r}|\nabla \bar u|^2\bigg)^\frac{1}{2}\lesssim 1
\end{align}
next to
\begin{align}\label{Pd15}
\bigg(\fint_{|x| > R}|\bar g|^2\bigg)^\frac{1}{2}\le\bigg(\frac{r}{R}\bigg)^{\beta+d-1},\quad
\bigg(\fint_{|x| > R}|\nabla \bar u|^2\bigg)^\frac{1}{2}\le\bigg(\frac{r}{R}\bigg)^{k+d-1}\quad\mbox{for all}\;R\ge \BR6 2r.
\end{align}

\medskip

We now consider the form $\ell$ on $X_m$ defined through
\begin{align}\label{Pd7}
\ell.\tilde u=\int\nabla \tilde u \cdot\bar g\quad\mbox{for}\;\tilde u \in X_m.
\end{align}
We claim that $\ell$ is well-defined and bounded by unity:
\begin{align}\label{Pd4}
\sup_{\tilde u \in X_m}\frac{\frac{1}{r^d}\ell.u}{(\fint_{ |x| < r}|\nabla\tilde u |^2)^\frac{1}{2}}\lesssim 1
\end{align}
and that it vanishes on $X_{k-1}$:
\begin{align}\label{Pd5}
\ell.\tilde u =0\quad\mbox{for}\;\tilde u \in X_{k-1}.
\end{align}
Indeed, for (\ref{Pd4}) we decompose the integral on the right-hand side of \BR6 (\ref{Pd7}) \ER into dyadic annuli 
and use the Cauchy-Schwarz inequality
\begin{equation*}
\int|\nabla\tilde u \cdot \bar g|
\le\sum_{R\ge \BR6 2r}\bigg(\int_{R < |x| < 2R}|\nabla\tilde u |^2\bigg)^\frac{1}{2}
\bigg(\int_{R < |x| < 2R}|\bar g|^2\bigg)^\frac{1}{2}+\bigg(\int_{ |x| < \BR6 2r}|\nabla\tilde u|^2\bigg)^\frac{1}{2}
\bigg(\int_{ |x| < \BR6 2r}|\bar g|^2\bigg)^\frac{1}{2},
\end{equation*}
where the sum runs over all dyadic multiples of $r$. We rewrite this as
\begin{align*}
\frac{1}{r^d}\int|\nabla\tilde u\cdot \bar g| \le\sum_{R\ge \BR6 2r}\bigg(\frac{R}{r}\bigg)^d\bigg(\fint_{|x| < 2R}|\nabla\tilde u|^2\bigg)^\frac{1}{2}
\bigg(\fint_{|x| > R}|\bar g|^2\bigg)^\frac{1}{2}
+\bigg(\fint_{ |x| < \BR6 2r}|\nabla\tilde u|^2\bigg)^\frac{1}{2}\bigg(\fint_{ |x| < \BR6 2r}|\bar g|^2\bigg)^\frac{1}{2},
\end{align*}
into which we insert (\ref{Pd16}) and (\ref{Pd15}):
\begin{align*}
\frac{1}{r^d}\int|\nabla\tilde u\cdot \bar g|\le
\sum_{R\ge \BR6 2r}\bigg(\frac{r}{R}\bigg)^{\beta-1}\bigg(\fint_{ |x| < 2R}|\nabla\tilde u|^2\bigg)^\frac{1}{2}
+\bigg(\fint_{ |x| < \BR6 2r}|\nabla\tilde u|^2\bigg)^\frac{1}{2}.
\end{align*}
We now appeal to our hypothesis (\ref{F2}) with $k$ replaced by $m$ \BR6 and use it on both terms on the right-hand side to get \ER
\begin{align}\label{Pd2e}
\frac{1}{r^d}\int|\nabla\tilde u\cdot \bar g|\BR6 \lesssim \ER 
\bigg(\sum_{R\ge r}\bigg(\frac{r}{R}\bigg)^{\beta-m}+1\bigg)
\bigg(\fint_{ |x| < r}|\nabla\tilde u|^2\bigg)^\frac{1}{2},
\end{align}
which implies (\ref{Pd4}) because of $\beta>m$.

\medskip

We now turn to (\ref{Pd5}). We note that for any constants $\tilde c$ and $\bar c$
we obtain from (\ref{Pd12}) and the $a$-harmonicity of $\tilde u\in X_{k-1}$
\begin{align*}
\nabla\cdot\big((\tilde u-\tilde c)(a\nabla\bar u+\bar g)-(\bar u-\bar c)a\nabla\tilde u\big)=\nabla\tilde u\cdot\bar g.
\end{align*}
Hence if $\eta$ denotes a cut-off function for $\{ |x| < R \}$ in $ \{ |x| < 2R \}$, with $|\nabla\eta|\lesssim\frac{1}{R}$,
\begin{align}\label{Pd24}
\int\eta\nabla\tilde u\cdot\bar g
=-\int\nabla\eta\cdot\big((\tilde u-\tilde c)(a\nabla\bar u+\bar g)-(\bar u-\bar c)a\nabla\tilde u\big).
\end{align}
The right-hand side is easily estimated by 
\begin{align*}
\lefteqn{\frac{1}{R}\int_{R < |x| < 2R}|(\tilde u-\tilde c)(a\nabla\bar u+\bar g)-(\bar u-\bar c)a\nabla\tilde u\big|}\nonumber\\
&\lesssim
\frac{1}{R}\bigg[\bigg(\int_{R < |x| < 2R}(\tilde u-\tilde c)^2\bigg)^\frac{1}{2}\bigg(
\bigg(\int_{|x| > R}|\nabla\bar u|^2\bigg)^\frac{1}{2}+\bigg(\int_{|x| > R}|\bar g|^2\bigg)^\frac{1}{2}\bigg)\nonumber\\
&\qquad +\bigg(\int_{R < |x| < 2R}(\bar u-\bar c)^2\bigg)^\frac{1}{2}\bigg(\int_{R < |x| < 2R}|\nabla\tilde u|^2\bigg)^\frac{1}{2}\bigg],
\end{align*}
which by our hypothesis (\ref{poinc}) implies (for a suitable choice of $\tilde c$ and $\bar c$)
\begin{align*}
\frac{1}{R}\int_{R < |x| < 2R}|(\tilde u-\tilde c)(a\nabla\bar u+\bar g)-(\bar u-\bar c)a\nabla\tilde u\big|
\lesssim\bigg(\bigg(\int_{|x| > R}|\nabla\bar u|^2\bigg)^\frac{1}{2}+\bigg(\int_{|x| > R}|\bar g|^2\bigg)^\frac{1}{2}\bigg)
\bigg(\int_{ |x| < 2R}|\nabla\tilde u|^2\bigg)^\frac{1}{2}.
\end{align*}
We insert (\ref{Pd15}) (using $\beta\ge m\ge k$) for the first factor and our assumption (\ref{F2}) (with $k-1$ playing the role of $k$), \BR6 both together with~\eqref{vol}, to get for $R \ge 2r$\ER
\begin{align}\label{Pd25}
\frac{1}{R}\int_{R < |x| < 2R}|(\tilde u-\tilde c)(a\nabla\tilde u+\tilde g)-(\bar u-\bar c)a\nabla\tilde u\big|
\lesssim 
\BR6 r^d R^{-1} \ER 
\bigg(\fint_{ |x| < r}|\nabla\tilde u|^2\bigg)^\frac{1}{2}.
\end{align}
>From identity (\ref{Pd24}), the absolute convergence of the left-hand side integral, see (\ref{Pd2e}), and the
estimate (\ref{Pd25}) of the right-hand side integral, we obtain 
\begin{align*}
\int\nabla\tilde u\cdot\bar g=0\quad\mbox{for}\;\tilde u\in X_{k-1}
\end{align*}
in the limit $R\uparrow\infty$. In view of (\ref{Pd7}) this implies (\ref{Pd5}).

\medskip

In view of (\ref{Pd5}) and the definition (\ref{Pd7}), by Proposition \ref{Pdual}, there exists $w\in Y_k(r_*)$ such
that
\begin{align}\label{Pd14}
( \tilde u,w )=\int\nabla\tilde u\cdot \bar g\quad\mbox{for}\;\tilde u\in X_{m}.
\end{align}
By (\ref{Pd6})  and (\ref{Pd4}) we have
\begin{align}\label{Pd8}
\bigg(\fint_{|x| > r}|\nabla w|^2\bigg)^\frac{1}{2}\lesssim \biggl( \frac{r}{r_*} \biggr)^{m-k}.
\end{align}
By Lemma \ref{LF1} applied with $w\in Y_k(r_*)\subset Y_k(r)$ playing the role of
$v$ this implies the second statement in (\ref{Pd18}).
 
\medskip

We apply Lemma \ref{Lext} once more, this time to $(u,g)$ replaced by $(w,0)$;
assumption (\ref{Lext2}) of that lemma is satisfied by the second
statement in (\ref{Pd18}), using the same argument that lead from (\ref{Lext11}) to (\ref{Lext12}). 
Lemma \ref{Lext} yields a couple
$(\bar w,\bar h)$ of a function and a vector field such that 
\begin{align}\label{Pd10}
-\nabla\cdot a\nabla \bar w=\nabla\cdot \bar h 
\end{align}
and which are extensions in the sense of
\begin{align}\label{Pd9}
(\nabla\bar w,\bar h)=(\nabla w,0)\quad\mbox{in}\;\{ |x| > 2r\}.
\end{align}
Furthermore, we have by (\ref{Pd1}) and (\ref{Pd8})
\begin{align}\label{Pd16bis}
\bigg(\int_{ |x| < 2r}|\nabla\bar w|^2+|\bar h|^2\bigg)^\frac{1}{2}\lesssim \biggl( \frac{r}{r_*} \biggr)^{m-k}.
\end{align}
Note that (\ref{Pd9}) and (\ref{Pd10}) imply for any $a$-harmonic $\tilde u$
\begin{align*}
\tilde u a\nabla w-wa\nabla\tilde u&=\tilde u(a\nabla\bar w+\bar h)-\bar wa\nabla\tilde u&&\quad\mbox{in}\;\{ |x| > 2r\},\\
\nabla\cdot(\tilde u(a\nabla\bar w+\bar h)-\bar wa\nabla\tilde u)&=\nabla\tilde u\cdot\bar h,
\end{align*}
respectively, so that by definition of $ ( \, \cdot \, , \, \cdot \, )$:
\begin{align}\label{Pd13}
( \tilde u,w )=\int\nabla\tilde u\cdot \bar h.
\end{align}

\medskip

We now consider the function $\bar u-\bar w$. By (\ref{Pd12}) and (\ref{Pd10}) we have the equation
\begin{align}\label{Pd17}
-\nabla\cdot a\nabla(\bar u-\bar w)=\nabla\cdot(\bar g-\bar h).
\end{align}
We obtain from (\ref{Pd13}) and (\ref{Pd14}) for the right-hand side
\begin{align*}
\int\nabla\tilde u\cdot(\bar g-\bar h)=0\quad\mbox{for all}\;\tilde u\in X_m.
\end{align*}
By (\ref{Pd16}), (\ref{Pd15}) (for $\bar g$) and (\ref{Pd9}),(\ref{Pd16bis}) (for $\bar h$) we have \BR5 by the triangle inequality \ER 
\begin{align}\nonumber
\bigg(\fint_{ |x| < r}|\bar g-\bar h|^2\bigg)^\frac{1}{2}\lesssim \biggl( \frac{r}{r_*} \biggr)^{m-k},\quad
\bigg(\fint_{|x| > R}|\bar g-\bar h|^2\bigg)^\frac{1}{2}\le \biggl( \frac{r}{r_*} \biggr)^{m-k}\biggl(\frac{r}{R}\biggr)^{\beta+d-1}.
\end{align}
Hence we may apply Lemma \ref{Ld} with $\bar g-\bar h$ playing the role
of $g$ and with $m$ playing the role of $k$. Since by (\ref{Pd16}), (\ref{Pd15}) and (\ref{Pd9}),(\ref{Pd16bis}), this time respectively for $\bar u$ and  $\bar w$, the vector field $\nabla(\bar u-\bar w)$ is square integrable,
we learn from (\ref{Pd17}) and the uniqueness for (\ref{Ld2}) that $\nabla(\bar u-\bar w)$ plays
the role of $\nabla u$ in Lemma \ref{Ld}. Hence (\ref{Ld12}) turns into
\begin{align}\nonumber
\bigg(\fint_{|x| > R}|\nabla(\bar u-\bar w)|^2\bigg)^\frac{1}{2}\lesssim \biggl( \frac{r}{r_*} \biggr)^{m-k}\biggl(\frac{r}{R}\biggr)^{\beta+d-1}\quad\mbox{for all}\;R\ge r,
\end{align}
which by (\ref{Pd3}) and (\ref{Pd9}) implies the first estimate in (\ref{Pd18}).

\medskip

We finally turn to the uniqueness statement. From the first estimate in (\ref{Pd18}) we obtain by
the triangle inequality
\begin{align}\label{Pd22}
\bigg(\fint_{|x| < R}|\nabla(w-w')|^2\bigg)^\frac{1}{2}\lesssim \biggl( \frac{r}{r_*} \biggr)^{m-k}\biggl(\frac{r}{R}\biggr)^{\beta+d-1}.
\end{align}
By Proposition \ref{Pdual}, there exist $w_n\in Y_n(r_*)$, $n=k,\ldots,m$, and a remainder $\tilde w\in Y_{m+1}(r)$
such that
\begin{align*}
w-w'=w_k+\cdots+ w_m+\tilde w.
\end{align*}
We now argue that $\nabla w_n=0$ in $\{ |x| > r_*\}$: Indeed by the homogeneity (\ref{Pd20}) of elements of $Z_n(r_*)$ and the
boundedness (\ref{Pd21}) (with $R\ge r$ playing the role of $r$) of the projection we have for all $R\ge r$
\begin{align*}
\bigg(\fint_{|x|>r_*}|\nabla w_n|^2\bigg)^\frac{1}{2}&\stackrel{(\ref{Pd20})}{\lesssim}
\bigg(\frac{R}{r_*}\bigg)^{n+d-1}\bigg(\fint_{|x|>R}|\nabla w_n|^2\bigg)^\frac{1}{2}\nonumber\\
&\stackrel{(\ref{Pd21})}{\lesssim} \bigg(\frac{R}{r_*}\bigg)^{n+d-1}\bigg(\fint_{|x|>R}|\nabla(w-w')|^2\bigg)^\frac{1}{2}\nonumber\\
&\stackrel{(\ref{Pd22})}{\lesssim}
\bigg( \frac{r}{r_*}\bigg)^{m-k}\bigg(\frac{R}{r_*}\bigg)^{n+d-1}\bigg(\frac{r}{R}\bigg)^{\beta+d-1}.
\end{align*}
Since $n\le m<\beta$ we obtain in the limit $R\uparrow\infty$ that $\nabla w_n =0$ in $\{ |x| > r_* \}$.
\end{proof}

\subsection{Proofs of Propositions \ref{Pdual} and \ref{Pbidual} and of Corollary~\ref{Cor5.b}}\hspace*{\fill} 

\begin{proof}[Proof of Proposition~\ref{Pdual}]

\PfStart{C:Pdual}

\PfStep{C:Pdual}{C:Pdual1}
Estimate of $ ( \, \cdot ,  \cdot \, )$; we claim that for an $a$-harmonic $v$ defined
on $\{ |x| > r \}$ with $\big(\int_{|x| > r}|\nabla v|^2\big)^\frac{1}{2}<\infty$ and an $a$-harmonic $u$ defined on $\mathbb{R}^d$ 
we have
\begin{align}\label{f48}
|( u,v )|\lesssim\bigg(\int_{ |x| < 2R}|\nabla u|^2\bigg)^\frac{1}{2}\bigg(\int_{|x| > R}|\nabla v|^2\bigg)^\frac{1}{2}
\quad\mbox{for all}\;R\ge r.
\end{align}
Let $\eta$ be a (say, radial) function with
\begin{align}\label{f49}
\eta=0\;\mbox{on} \ \ \;\{ |x| < R \},\quad\eta=1\;\mbox{on}\;\{ |x| > 2R \},\quad
|\nabla\eta|\lesssim\frac{1}{R}
\end{align}
so that by definition~\eqref{bil.def} of $( u,v )$ we have
\begin{align}\label{f50}
( u,v )=\BR5 -\int\nabla\eta\cdot(va\nabla u-ua\nabla v).
\end{align}
The $a$-harmonicity of $v$ on $\{ |x| > r \}$ together with $(\int_{|x| > r}|\nabla v|^2)^\frac{1}{2}<\infty$ implies 
$\int\nabla\eta\cdot a\nabla v=0$;
likewise the $a$-harmonicity of $u$ on $\mathbb{R}^d$ yields $\int\nabla\eta\cdot a\nabla u=0$. Hence the 
formula (\ref{f50}) may be modified to
\begin{equation}\nonumber
( u,v )=\BR5 -\int\nabla\eta\cdot\big((v-\tilde c)a\nabla u - (u-c)a\nabla v\big)
\end{equation}
for some arbitrary constants $c,\tilde c$. By the Cauchy-Schwarz inequality and (\ref{f49}) we obtain
\begin{align*}
|( u,v )|&\lesssim \frac{1}{R}\bigg[\bigg(\int_{R < |x| < 2R}(u-c)^2\bigg)^\frac{1}{2}
\bigg(\int_{R < |x| < 2R}|\nabla v|^2\bigg)^\frac{1}{2}\nonumber\\
&\qquad+\bigg(\int_{R < |x| < 2R}(v-\tilde c)^2\bigg)^\frac{1}{2}
\bigg(\int_{R < |x| < 2R}|\nabla u|^2\bigg)^\frac{1}{2}\bigg].
\end{align*}
For an appropriate choice of $c$ and $\tilde c$, by Poincar\'e's inequality \eqref{poinc} this yields~\eqref{f48}.

\PfStep{C:Pdual}{C:Pdual2} Proof of i);
we consider the map $(Y_k/Y_{m+1}) \ni v \mapsto ( \cdot, v) \in (X_m/X_{k-1})^*$. We first remark that if $v \in Y_k$, we have by Corollary \ref{Lbil} that $( u, v)= 0$ whenever $u \in X_{k-1}$. Moreover, if 
for some $v, v' \in Y_k$ it holds $( u , v)= (u , v')$ for every $u \in X_m$, then Lemma \ref{LF1} yields that $v - v' \in Y_{m+1}$. Therefore, the map is not only well-defined but also one-to-one.
We prove that it is also onto: To do so we construct a partial inverse $(X_{m}/ X_{k-1})^* \rightarrow Y_k(r_*)$, \BR5 and thus give ourselves an $\ell \in (X_m/X_{k-1})^*$. \ER Since the product $\int_{|x|< r_*} \nabla u \cdot \nabla \tilde u $ is an inner product on the space $X_m/X_{0}$ \BR5(see~\eqref{F2})\ER, Riesz's representation theorem implies
that there exists an element $U \in X_m/X_{0}$ such that $\ell. u = \int_{|x|< r_*} \nabla u \cdot \nabla U$ and
\begin{align}\label{Pdual3}
\biggl(\int_{|x|< r_*} |\nabla U|^2 \biggr)^\oh \leq \sup_{ u\in X_m}\frac{\ell.u}{(\int_{|x|< r_*}|\nabla u|^2)^\frac{1}{2}}.
\end{align}
We 
claim that the Lax-Milgram solution of 
\begin{align}\label{Pdual1}
 - \nabla \cdot a \nabla v = \BR5 - \ER \nabla \cdot h \qquad \BR5 \textrm{ where } h:=  I(\{|x|< r_*\}),
\end{align}
is in $Y_k(r_*)$ and such that $\ell = ( \cdot, v)$.
Since by definition of $h$, the function $v$ is $a$-harmonic in $\{ |x| > r_* \}$, satisfies the energy estimate
\begin{align*}
 \int |\nabla v |^2 \lesssim \int |h|^2 = \int_{|x|< r_*}|\nabla U|^2 < +\infty,
\end{align*}
\BR5 and $(1,v)=0$, we have by Lemma~\ref{LF1} b) $\Rightarrow$ a) \ER 
that $ v \in Y_1(r_*)$. 
By the definition \eqref{bil.def} of the bilinear form $( \cdot, \cdot)$, the divergence theorem and equation \eqref{Pdual1}
we have that \BR5 for all $u \in X_m$\ER
\begin{align*}
 (u ,v) = \BR5 - \ER \int \nabla \eta \cdot ( v a \nabla u - u a \nabla v) =  \int \nabla(\eta u) \cdot h = \int \nabla u \cdot h = \ell. u .
\end{align*}
In addition, since $\ell \in  (X_m/X_{k-1})^*$, it follows that $(u, v)= 0$ whenever $u \in X_{k-1}$. This implies by Lemma \ref{LF1} \BR5 b) $\Rightarrow$ a)\ER, that $v \in Y_{k}(r_*)$. We thus constructed a partial inverse $(X_{m}/ X_{k-1})^* \rightarrow Y_k(r_*)$ 
and \BR5 this way \ER established that the map $(Y_k/Y_{m+1}) \rightarrow (X_m/X_{k-1})^*$ is also onto. The proof of i) is thus concluded.


\PfStep{C:Pdual}{C:Pdual3}
Proof of ii): construction of the spaces $Z_n(r_*)$, $n= k, ..., m$; without loss of generality, we give the argument for $Z_k(r_*)$. 
We consider the linear map  $M : (X_{k}/ X_{k-1})^* \rightarrow Y_k(r_*)$ constructed in Step~\ref{C:Pdual2}, with $m=k$: We define $Z_k(r_*)$ as the image of $(X_{k}/ X_{k-1})^*$ under this map. 
\BR5
While we only need to show that $Z_k(r_*)$ is a complement of $Y_{k+1}(r_*)$ in $Y_k(r_*)$, for the later purposes we show that $Z_k(r_*)$ is complement of $Y_{k+1}(r)$ in $Y_k(r)$ for any $r \ge r_*$. 
\ER
Since for every $r \geq r_*$ it holds $Y_k(r_*) \subset Y_k(r)$, 
we clearly have that $Z_k(r_*)\subset Y_k(r)$ for every $r\geq r_*$ and we need to argue that $Z_k(r_*)$ is a complement
of $Y_{k+1}(r)$ in $Y_k(r)$. On the one hand, we have $Z_k(r_*)\cap Y_{k+1}(r)=\{0\}$, which can
be seen as follows: We are given a $v\in Z_k(r_*)\cap Y_{k+1}(r)$; because of $v\in Z_k(r_*)$, by definition, there exists a pre-image $\ell\in (X_k / X_{k-1})^*$ of $v$ under the map $M$ such that $\ell = (\cdot, v)$; because of $v\in Y_{k+1}(r)$,
we infer from Lemma \ref{LF1} that $\ell$ vanishes on $X_k$; hence $\ell=0$ and also its image $v$ under the linear map $M$
has to vanish.
On the other hand, we have $Y_k(r)=Z_k(r_*)+Y_{k+1}(r)$ which can be seen as follows: Given $\tilde v\in Y_k(r)$, we consider the form $\ell \in \BR5 (X_k/X_{k-1})^*$ given by $X_k\ni u\mapsto ( u,\tilde v)$; \BR5 by the above construction of $M$, there exists 
$v\in Y_k(r)$ \ER such that $\ell= (\cdot , v)$. Hence we have $( u,\tilde v-v)=0$
for all $u\in X_k$, which by Lemma \ref{LF1} implies $\tilde v-v\in Y_{k+1}(r)$ and thus
$\tilde v \in Z_k(r_*)+Y_{k+1}(r)$.

\medskip

\BR5
We are now given $l \in (X_m / X_{k-1})^*$. Since $(X_m / X_{k-1})^* = (X_m / X_{m-1})^* \oplus \cdots \oplus (X_k / X_{k-1})^*$, there exists a unique $(m-k+1)$-tuple $(w_m,\ldots,w_k) \in Z_m(r_*) \times \cdots \times Z_k(r_*)$ such that $l = (\cdot,w)$ with $w = w_m + \cdots + w_k$.
\ER
 It thus remains to show that $w$ satisfies \eqref{Pd6}. By linearity and
the construction of each $Z_n(r_*)$, we have that $w$ satisfies \eqref{Pdual1} 
for some $U\in X_m$ such that \BR5 for all $ u \in X_m$\ER 
\begin{align*}
\ell. u  = ( u , w) \overset{\eqref{bil.def}}{=} -\int \nabla \eta \cdot \bigl( w a\nabla u  -  u  a \nabla w )\stackrel{\eqref{Pdual1}}{=}\int_{|x|<r_*}\nabla U\cdot \nabla  u,
\end{align*}
\BR5
where $\eta$ here comes from definition~\eqref{bil.def} (i.e. $\eta = 1$ in $\{|x|<r_*\}$), in particular it is quite different from $\eta$ defined in~\eqref{f48}.
\ER 
Therefore, the energy estimate for $w$
\begin{equation*}
\int |\nabla w|^2 \lesssim \int |h|^2 = \int_{|x|< r_*}|\nabla U|^2,
\end{equation*}
\BR5 the above identity, and~\eqref{Pdual3} \ER yield that 
\begin{align*}
\biggl( \fint_{|x| > r_*} |\nabla w |^2\biggr)^\oh {\lesssim} \sup_{  u \in X_m} \frac{\frac{1}{r_*^d} \ell. u }{\bigl(\fint_{|x|< r_*}|\nabla  u|^2\bigr)^\frac{1}{2}}.
\end{align*}
Moreover, by Lemma \ref{LF1} c) and \eqref{F2} we also get for every $r \geq r_*$ that
\begin{align}
\biggl(\fint_{|x| > r}|\nabla w|^2 \biggr)^\oh &\lesssim \biggl(\frac{r_*}{r}\biggr)^{d-1+ k}\sup_{  u \in X_m} \frac{\frac{1}{r_*^d} \ell. u }{\bigl(\frac{r_*}{r}\bigr)^{m-1}\bigl(\fint_{|x|< r}|\nabla  u|^2\bigr)^\frac{1}{2}}\notag\\
&\simeq \biggl(\frac{r}{r_*} \biggr)^{m-k}\sup_{  u \in X_m} \frac{\frac{1}{r^d} \ell. u }{\bigl(\fint_{|x|< r}|\nabla  u|^2\bigr)^\frac{1}{2}},\nonumber
\end{align}
i.e. inequality \eqref{Pd6}.


\PfStep{C:Pdual}{C:Pdual4}
Proof of iii); for \BR5 $m=k$ estimate \ER \eqref{Pd6} of ii) turns into
\begin{equation}\nonumber
\biggl( \int_{|x| > r} |\nabla w |^2 \biggr)^\oh \lesssim   \sup_{ u\in X_k}\frac{(u, w)}
{(\int_{|x| < r}|\nabla u|^2)^\frac{1}{2}}\BR5 \quad \textrm{for all } r \ge r_* \textrm{ and } w \in Z_k(r_*), 
\end{equation}
\BR5
which by~\eqref{F2} may be upgraded to
\begin{equation}\label{Pdual4}
\biggl( \int_{|x| > r} |\nabla w |^2 \biggr)^\oh \lesssim   \sup_{ u\in X_k}\frac{(u, w)}
{(\int_{|x| < 2r}|\nabla u|^2)^\frac{1}{2}} \quad \textrm{for all } r \ge r_* \textrm{ and } w \in Z_k(r_*), 
\end{equation} 
\ER
We begin with \eqref{Pd21} for the decomposition $w= w_k + \tilde w \in Z_k(r_*) + Y_{k+1}(r)$. By estimate~\eqref{Pdual4} it holds for every $r \geq r_*$
\begin{align*}
\biggl( \int_{|x| > r} |\nabla w_k |^2 \biggr)^\oh \lesssim  \sup_{ u\in X_k}\frac{(u, w_k)}{(\int_{|x| < 2r}|\nabla u|^2)^\frac{1}{2}}.
\end{align*}
Since by linearity and Lemma \ref{LF1} \BR5 a) $\Rightarrow$ b)\ER, we have that $( u, w) = (u, w_k)$ for every $u \in X_k$, it follows that 
\begin{align*}
\biggl( \int_{|x| > r} |\nabla w_k |^2 \biggr)^\oh \lesssim  \sup_{ u\in X_k}\frac{(u, w)}{(\int_{|x| < \BR5 2r}|\nabla u|^2)^\frac{1}{2}} \stackrel{\eqref{f48}}{\lesssim} \biggl( \int_{|x| > \BR5 r/2} |\nabla w|^2 \biggr)^\oh,
\end{align*}
i.e. estimate \eqref{Pd21} with $k=m$. We generalize \eqref{Pd21} by iteration: If  $w- w_k = w_{k+1} + \tilde w \in Z_{k+ 1}(r_*) + Y_{k+2}(r)$ we repeat the previous argument to infer
\begin{align*}
 \biggl( \int_{|x| > r} |\nabla w_{k+1} |^2 \biggr)^\oh \lesssim \biggl( \int_{|x| > r} |\nabla (w - w_k)|^2 \biggr)^\oh.
\end{align*}
The triangle inequality on the right-hand side and \eqref{Pd21} for $w_k$ concludes the proof.


\PfStep{C:Pdual}{C:Pdual5}
Proof of iv); let  $v \in Z_n(r_*)$. The $\lesssim$ part for \eqref{Pd20} follows immediately from Lemma \ref{LF1} c). Moreover, estimate \eqref{Pdual4} and assumption 
\eqref{F2} yield 
\begin{align*}
\biggl( \int_{|x| > r} |\nabla v |^2 \biggr)^\oh &\lesssim  \biggl(\frac R r \biggr)^{\frac d 2 + m -1}\hspace{-0.6cm}\sup_{ u\in X_n}\frac{(u, v )}{(\int_{|x| < 2R}|\nabla u|^2)^\frac{1}{2}}
\stackrel{\eqref{f48}}{\lesssim}\biggl(\frac R r \biggr)^{\frac d 2 + m -1} \hspace{-0.2cm}\biggl( \int_{|x| > R} |\nabla v|^2 \biggr)^\oh,
\end{align*}
i.e. the $\gtrsim$ in \eqref{Pd20}.
\end{proof}


\begin{proof}[Proof of Proposition~\ref{Pbidual}]

We start with the following observation that expresses a fundamental non-degeneracy of
our bilinear form: For any $a$-harmonic function $u$ we have
\begin{align}\label{cw01}
\biggl(\int_{|x|<r}|\nabla u|^2\biggl)^\frac{1}{2}\lesssim\sup_{v\in Y_1(r)}\frac{(u,v)}{(\int_{|x|>r}|\nabla v|^2)^\frac{1}{2}}
\quad\mbox{for all}\;r\ge r_*.
\end{align}
To this purpose, we fix $r\ge r_*$ and give a duality argument, which means we give ourselves
a square-integrable vector field $h$ supported on $\{|x|<r\}$ and denote by $v$ the Lax-Milgram solution
of $-\nabla\cdot a\nabla v=- \nabla\cdot h$; in particular we have by the energy estimate $\int|\nabla v|^2\lesssim\int|h|^2$.
Since clearly $v$ is $a$-harmonic in $\{|x|>r\}$ and it holds $(1,v)=0$, we learn from Lemma \ref{LF1}
$b)\Rightarrow a)$ that $v\in Y_1(r)$. From definition (\ref{bil.def}) we see that $(u,v)=\int\nabla u\cdot h$. 
Hence the right-hand side of (\ref{cw01}) dominates $\sup_{h:{\rm supp}h\subset\{|x|<r\}}\frac{\int\nabla u\cdot h}{(\int|h|^2)^\frac{1}{2}}$,
which is identical to the left-hand side.

\medskip

Equipped with (\ref{cw01}) and (\ref{Pd6}) of Proposition \ref{Pdual} we now may conclude
\begin{align}\label{cw02}
(\cdot,\cdot)\;\mbox{is non-degenerate on}\;X_m/X_0\,\times\, Y_1(r)/Y_{m+1}(r)\quad\mbox{for any}\;r\ge r_*,
\end{align} 
by which, in view of Corollary \ref{Lbil}, we understand the two properties: 1) that for any $u\in X_m$ with
$\forall\;v\in Y_1(r)\;(u,v)=0$ we have $u\in X_0$, and 2) that for any $v\in Y_1(r)$ with 
$\forall\;u\in X_m\;(u,v)=0$ we have $v\in Y_{m+1}(r)$. Property 1) is a direct qualitative consequence
of (\ref{cw01}). In view of $Y_1(r)=Z_1(r_*)+\cdots+ Z_m(r_*)+ Y_{m+1}(r)$, see Proposition \ref{Pdual} {\rm iii)},
property 2) is a qualitative consequence of (\ref{Pd6}) for $k=1$, in form of (\ref{Pdual4}). 
By elementary linear algebra, we note that (\ref{cw02}) allows to pass from part {\rm i)} of Proposition
\ref{Pdual} to part {\rm i)} of this proposition.

\medskip

Encouraged by property (\ref{cw02}), we now define for $n=1,\ldots,m$
\begin{equation}\label{cw03}
W_n:=\bigg\{u\in X_n|\fint_{|x|<r_*}u=0,\;\forall\;v\in Z_1(r_*)+\cdots+Z_{n-1}(r_*)\;\;:(u,v)=0\bigg\}
\end{equation}
with the understanding that for $n=1$, the last condition is void. Hence $W_n/X_0$ is the
$(\cdot,\cdot)$-polar complement of $Z_1(r_*)+\cdots+Z_{k-1}(r_*)$ in $X_n/X_0$. It is built into this definition
that {\rm iii)} holds for $n>m$; {\rm iii)} for $n<m$ follows from Corollary \ref{Lbil}.
We now argue that (\ref{cw02}) implies the qualitative part of {\rm ii)}, namely
\begin{align*}
X_k=X_{k-1}\oplus W_k,
\end{align*}
which amounts to 1) $\{0\}=X_{k-1}\cap W_k$ and 2) $X_k=X_{k-1}+W_k$. Property 1) follows from the non-degeneracy (\ref{cw02}) 
of $(\cdot,\cdot)$ on $X_{k-1}/X_0\times (Y_1(r_*)/Y_{k}(r_*)\cong Z_1(r_*)+\cdots+Z_{k-1}(r_*))$: Because of the latter,
an element $u\in X_{k-1}\cap W_k$ has to be in $X_0$, so that by the first condition in (\ref{cw03}) it has to vanish. 
Property 2) follows once the dimensions of $X_k$ and $X_{k-1}+W_k$ are equal, which by property 1) reduces to arguing that
$\dim X_k$ $= \dim X_{k-1}+ \dim W_k$, which we rewrite as $\dim X_k/X_0 = \dim X_{k-1}/X_0+ \dim W_k$.
Again, this follows from the non-degeneracy (\ref{cw02}), which ensures that for the polar set $W_k$ we have the
dimension formula $ \dim W_k = \dim X_k/X_0- \dim (Z_1(r_*)+\cdots+Z_{k-1}(r_*))$, and from Proposition \ref{Pdual},
which ensures that $\dim (Z_1(r_*)+\cdots+Z_{k-1}(r_*)) = \dim Y_1(r_*)/Y_k(r_*)$ $=\dim X_{m-1}/X_0$.

\medskip

We now turn to the quantitative part of {\rm ii)} and give ourselves $u\in W_k+\cdots+W_m$ and $r\ge r_*$. Since
in particular $u\in X_m$, we have by our assumption (\ref{F2}) that 
\begin{align*}
\bigg(\fint_{|x|<r}|\nabla u|^2\bigg)^\frac{1}{2}\lesssim\bigg(\frac{r}{r_*}\bigg)^{m-1}\bigg(\fint_{|x|<r_*}|\nabla u|^2\bigg)^\frac{1}{2}.
\end{align*}
We now dualize with help of (\ref{cw01}) (with $r=r_*$)
\begin{align*}
\bigg(\fint_{|x|<r}|\nabla u|^2\bigg)^\frac{1}{2}\lesssim
\bigg(\frac{r}{r_*}\bigg)^{m-1}\sup_{\tilde v\in Y_1(r_*)}\frac{\frac{1}{r_*^d}(u,\tilde v)}{(\fint_{|x|>r_*}|\nabla \tilde v|^2)^\frac{1}{2}}.
\end{align*}
We then split $Y_1(r_*)\ni \tilde v$ $=v_-+v+v_+\in(Z_1+\cdots+ Z_{k-1})(r_*)\oplus(Z_k+\cdots+Z_m)(r_*)\oplus Y_{m+1}(r_*)$,
cf.\ Proposition \ref{Pdual}, use that by part {\rm iii)} and Corollary \ref{Lbil}
we have $(u,\tilde v)=(u,v)$, and that by (\ref{Pd21}) it holds
$(\int_{|x|\ge r_*}|\nabla v|^2)^\frac{1}{2}$ $\lesssim (\int_{|x|\ge r_*}|\nabla\tilde v|^2)^\frac{1}{2}$, to arrive at
\begin{align*}
\bigg(\fint_{|x|<r}|\nabla u|^2\bigg)^\frac{1}{2}\lesssim
\bigg(\frac{r}{r_*}\bigg)^{m-1}\sup_{v\in (Z_k+\cdots+Z_m)(r_*)}\frac{\frac{1}{r_*^d}(u,v)}{(\fint_{|x|>r_*}|\nabla v|^2)^\frac{1}{2}}.
\end{align*}
Since in particular $v\in Y_k(r_*)$, we may use Lemma \ref{LF1} $a)\Rightarrow c)$ to obtain
\begin{equation}\label{cw04}
\bigg(\fint_{|x|<r}|\nabla u|^2\bigg)^\frac{1}{2}
\lesssim
\bigg(\frac{r}{r_*}\bigg)^{m-1}\bigg(\frac{r_*}{r}\bigg)^{k+d-1}\sup_{v\in (Z_k+\cdots+Z_m)(r_*)}
\frac{\frac{1}{r_*^d}(u,v)}{(\fint_{|x|>r}|\nabla v|^2)^\frac{1}{2}},
\end{equation}
which is the desired estimate.

\medskip

Equipped with part {\rm ii)}, part {\rm iv)} is easy: By (\ref{F2}) it is enough to establish the estimate for $r\ge 2r_*$.
We fix $n=k,\ldots,m$; by (\ref{F2}) again we have
$(\fint_{|x|<r}|\nabla u_n|^2)^\frac{1}{2}$ $\lesssim (\fint_{|x|<\frac{r}{2}}|\nabla u_n|^2)^\frac{1}{2}$.
We now appeal to (\ref{cw04}) with $k=m=n$
and applied to $u=u_n$ and $r$ replaced by $\frac{r}{2}$. By part {\rm iii)} we have for the numerator $(u_n,v)=(u,v)$, on which we apply
estimate (\ref{f48}) in form of $|\frac{1}{r^d}(u,v)|$ $\lesssim(\fint_{|x|<r}|\nabla u|^2)^\frac{1}{2}$
$(\fint_{|x|>\frac{r}{2}}|\nabla v|^2)^\frac{1}{2}$. This yields part {\rm iv)}.

\medskip

The $\lesssim$-part of~\eqref{Wnhomog} in v) follows trivially by condition \eqref{F2}. We argue for the $\gtrsim$-part: By iii) with $m=k=n$, we have
\begin{align*}
\biggl(\int_{|x| < r}|\nabla u|^2 \biggr)^\oh \lesssim \sup_{v\in Z_n (r_*)}\frac{(u,v)}{\bigl( \int_{|x|>r }|\nabla v|^2\bigr)^\oh},
\end{align*}
and, by \eqref{Pd20} of Proposition \ref{Pdual}, also
\begin{align*}
\biggl(\int_{|x| < r}|\nabla u|^2 \biggr)^\oh \lesssim \biggl( \frac{r}{R}\biggr)^{\frac d 2 -1 + n} \sup_{v\in Z_n (r_*)}\frac{(u,v)}{\bigl( \int_{|x|>R/2}|\nabla v|^2\bigr)^\oh}.
\end{align*}
We now apply inequality \eqref{f48} to the right-hand side above and establish~\eqref{Wnhomog}. 
\end{proof}

%


\begin{proof}[Proof of Corollary~\ref{Cor5.b}]


\BR5
We start with \eqref{m07}: If $u \in X_m$ and we denote by $w_m$ its projection onto $W_m$, we have by \eqref{F2}, together with iv) of Proposition \ref{Pbidual}, that
\begin{equation}\label{Rom5.b2}
\| u -w_m\|_{m-1} \overset{\eqref{F2}}{\lesssim} \biggl( \fint_{|x| < r_* }|\nabla (u - w_m)|^2\biggr)^\oh \overset{\textrm{iv)}}{\lesssim} \biggl( \fint_{|x| < r^*}|\nabla u|^2\biggr)^\oh \le \|u\|_{m}.
\end{equation}
Let us now define by $w_{m-1}$ the projection of $u - w_m$ onto $W_{m-1}$ (so that $w_m+w_{m-1}$ is the projection of $u$ onto $W_m\oplus W_{m-1}$). The same reasoning as above yields also that
\begin{align*}
\| (u - w_m) - w_{m-1} \|_{m-2} \lesssim \|u-w_m\|_{m-1}\stackrel{\eqref{Rom5.b2}}{\lesssim}\|u\|_m.
\end{align*}
In addition, assumption \eqref{F2}, the triangle inequality and Proposition \ref{Pbidual} iv) immediately imply
\begin{equation*}
\| w_m + w_{m-1}\|_m \overset{\eqref{F2}}{\lesssim} \biggl(\fint_{|x|<r_*}|\nabla (w_m  + w_{m-1})|^2\biggr)^\oh \overset{\textrm{iv)}}{\lesssim} \| u\|_{m}.
\end{equation*}
The combination of the previous two inequalities yields~\eqref{m07}.

\medskip

The proof of~\eqref{eqC301} is similar, 
but relying on iii) and iv) of Proposition \ref{Pdual} and on Lemma \ref{LF1} c) instead of assumption \eqref{F2}. 

\end{proof}

\ER

\subsection{Proofs of Lemmas \ref{Lg}, \ref{Ld}, and~\ref{Lup}}\hspace*{\fill} 

\begin{proof}[Proof of Proposition~\ref{Lg}]

For any dyadic multiple $\rho$ of $r$ we define
\BR3
\begin{align}\label{Lg4}
g_\rho:=I(\{ \rho < |x| < 2\rho\})g\quad\mbox{for}\;\rho>r\quad\mbox{and}\quad g_r:=I(\{ |x| < r \})g,
\end{align}
\ER 
where $I(B)$ denotes the characteristic function of the set $B$.
Assumption (\ref{Lg1}) then translates into
\begin{align}\label{Lg2}
\bigg(\frac{1}{\rho^d}\int|g_\rho|^2\bigg)^\frac{1}{2}\lesssim\bigg(\frac{\rho}{r}\bigg)^{\beta-1}.
\end{align}
Let $u_\rho$ denote the Lax-Milgram solution of
\begin{align}\label{Lg5}
-\nabla\cdot a\nabla u_\rho=\nabla\cdot g_\rho;
\end{align}
by the energy estimate we obtain from (\ref{Lg2}) that
\begin{align}\label{Lg10}
\bigg(\frac{1}{\rho^d}\int|\nabla u_\rho|^2\bigg)^\frac{1}{2}\lesssim\bigg(\frac{\rho}{r}\bigg)^{\beta-1}.
\end{align}

Since by (\ref{Lg4}) and (\ref{Lg5}), $u_\rho$ is $a$-harmonic in $\{|x|<\rho\}$ (for $\rho>r$), by Lemma \ref{Lup}, with $k-1$ playing the role of the exponent $k$,
there exists $\tilde u_{\rho}\in X_{k-1}$ 
\begin{align*}
\bigg(\fint_{|x| < R}|\nabla(u_\rho-\tilde u_{  \rho})|^2\bigg)^\frac{1}{2}&\lesssim\bigg(\frac{R}{\rho}\bigg)^{k-1}\bigg(\fint_{|x|<\rho}|\nabla u_{  \rho}|^2\bigg)^\frac{1}{2}
\quad\mbox{for all}\;R\in[r_*,\rho],
\end{align*}
which with help of (\ref{Lg10}) we upgrade to
\begin{align}
\bigg(\fint_{|x| < R}|\nabla(u_\rho-\tilde u_{  \rho})|^2\bigg)^\frac{1}{2}&\lesssim\bigg(\frac{R}{\rho}\bigg)^{k-1}\bigg(\frac{\rho}{r}\bigg)^{\beta-1}
\quad\mbox{for all}\;R\in[r,\rho].
\label{Lg13}
\end{align}

\medskip
Appealing to~\eqref{cw05} from Lemma~\ref{Lup} we have 

\begin{equation}\label{Lg14}
\bigg(\fint_{|x| < R}|\nabla \tilde u_{  \rho}|^2\bigg)^\frac{1}{2}\lesssim 
\bigg(\fint_{|x|<\rho}|\nabla u_\rho|^2\bigg)^\frac{1}{2} \quad\mbox{for all}\;R\in[r_*,\rho].
\end{equation}
%

\medskip

Up to additive constants we define $u$ via
\begin{align}\label{Lg7}
\nabla u:=\sum_{\rho> r}\nabla (u_\rho-\tilde u_{  \rho})+\nabla u_r.
\end{align}
Provided this series converges, the desired equation (\ref{Lg6}) follows from (\ref{Lg4}) in form
of $g=\sum_{\rho\ge r} g_\rho$ and (\ref{Lg5}) together with the $a$-harmonicity of $\tilde u_{\rho}$.
The absolute convergence of (\ref{Lg7}) and the desired estimate (\ref{Pg5bis}) both follow from
\begin{align*}
\sum_{\rho> r}\bigg(\fint_{|x| < R}|\nabla (u_\rho-\tilde u_{  \rho})|^2\bigg)^\frac{1}{2}+\bigg(\fint_{|x| < R}|\nabla u_r|^2\bigg)^\frac{1}{2}
\lesssim \bigg(\frac{R}{r}\bigg)^{\beta-1}\quad\mbox{for}\;R\ge r,
\end{align*}
which for fixed $R\ge r$ (which w.l.o.g. we assume to be a dyadic multiple of $r$) we split into
\begin{align}
\sum_{\rho> R}\bigg(\fint_{|x| < R}|\nabla (u_\rho-\tilde u_{  \rho})|^2\bigg)^\frac{1}{2}&\lesssim \bigg(\frac{R}{r}\bigg)^{\beta-1},\label{Lg12}\\
\sum_{R\ge \rho> r}\bigg(\bigg(\fint_{|x| < R}|\nabla u_\rho|^2\bigg)^\frac{1}{2}+\bigg(\fint_{|x| < R}|\nabla\tilde u_{  \rho}|^2\bigg)^\frac{1}{2}\bigg)&
\lesssim \bigg(\frac{R}{r}\bigg)^{\beta-1},\label{Lg11}\\
\bigg(\fint_{|x| < R}|\nabla u_r|^2\bigg)^\frac{1}{2}&\lesssim \bigg(\frac{R}{r}\bigg)^{\beta-1}.\label{Lg9}
\end{align}
Estimate (\ref{Lg9}) is an immediate consequence of (\ref{Lg10}) for $\rho=r$ which implies
$\big(\fint_{|x| < R}|\nabla u_r|^2\big)^\frac{1}{2}$ $\lesssim(\frac{r}{R})^\frac{d}{2}$ $\le(\frac{R}{r})^{\beta-1}$.
The estimate (\ref{Lg12}) of the far-field contribution is a consequence of 
(\ref{Lg13}) thanks to $\beta<k$. By our hypothesis (\ref{F2}) (with $k$ replaced by $k-1$) and estimate
(\ref{Lg14}) we have for the second term in the intermediate-field contribution (\ref{Lg11})
\begin{align*}
\bigg(\fint_{|x| < R}|\nabla\tilde u_{  \rho}|^2\bigg)^\frac{1}{2}\stackrel{(\ref{F2})}{\lesssim}
\bigg(\frac{R}{\rho}\bigg)^{(k-1)-1}\bigg(\fint_{|x|<\rho}|\nabla\tilde u_{  \rho}|^2\bigg)^\frac{1}{2}
\stackrel{(\ref{Lg14})}{\lesssim} \bigg(\frac{R}{\rho}\bigg)^{(k-1)-1}\bigg(\frac{1}{\rho^d}\int|\nabla u_\rho|^2\bigg)^\frac{1}{2}.
\end{align*}
We trivially have for the first contribution 
$\big(\fint_{|x| < R}|\nabla u_\rho|^2\big)^\frac{1}{2}$ $\le(\frac{\rho}{R})^\frac{d}{2}(\frac{1}{\rho^d}\int|\nabla u_\rho|^2\big)^\frac{1}{2}$
so that it is of higher order with respect to the second contribution (in the considered range $\rho\le R$). Hence 
by (\ref{Lg10}), both contributions are $\lesssim (\frac{R}{\rho})^{(k-1)-1}(\frac{\rho}{r})^{\beta-1}$.
Thanks to $\beta>(k-1)$, (\ref{Lg11}) follows.

\end{proof}



\begin{proof}[Proof of Lemma~\ref{Ld}]

By (\ref{Ld1}), $g$ is in particular square integrable. Hence by Lax-Milgram, there exists a square-integrable
solution $u$ of (\ref{Ld2}) and we have
\begin{align}\label{Ld13}
\bigg(\int|\nabla u|^2\bigg)^\frac{1}{2}\lesssim\bigg(\int|g|^2\bigg)^\frac{1}{2}.
\end{align}
Now let $R\ge r$ be given. In preparation for a duality argument,
let $h$ be a square-integrable vector field supported in $\{ |x| > R \}$ and
$v$ the Lax-Milgram solution of $-\nabla\cdot a\nabla v=\nabla\cdot h$. We thus have
\begin{align}\label{Ld3}
\bigg(\int|\nabla v|^2\bigg)^\frac{1}{2}\lesssim\bigg(\int|h|^2\bigg)^\frac{1}{2}
\end{align}
and 
\begin{align}\label{Ld5}
\int h\cdot\nabla u=\int g\cdot\nabla v.
\end{align}
By the support assumption on $h$, $v$ is $a$-harmonic in $\{ |x| < R \}$ so that by Lemma \ref{Lup}
and using (\ref{Ld3}),
there exists $w\in X_k$ such that
\begin{align}
\bigg(\fint_{ |x| < r}|\nabla(v-w)|^2\bigg)^\frac{1}{2}&\lesssim\bigg(\frac{r}{R}\bigg)^{(k+1)-1}\bigg(\frac{1}{R^d}\int|h|^2\bigg)^\frac{1}{2}
\quad\mbox{for all}\;r\in[r_*,R],\label{Ld11}\\
\bigg(\fint_{ |x| < r}|\nabla w|^2\bigg)^\frac{1}{2}&\lesssim \bigg(\frac{1}{R^d}\int|h|^2\bigg)^\frac{1}{2}
\quad\mbox{for all}\;r\in[r_*,R].\label{Ld9}
\end{align}
By assumption (\ref{Ld4}), identity (\ref{Ld5}) turn into 
\begin{align}\label{Ld6}
\int h\cdot\nabla u=\int g\cdot\nabla(v-w).
\end{align}

\medskip

In order to carry out the duality argument, we now estimate the right-hand side of (\ref{Ld6}) 
in terms of $(\int|h|^2)^\frac{1}{2}$. We do this by splitting the integral into dyadic annuli and
using the Cauchy-Schwarz inequality (where for this purpose, we assume w.l.o.g. that $R$ is a dyadic multiple of $r$)
\begin{align*}
\bigg|\int g\cdot\nabla(v-w)\bigg|&\le\sum_{\rho\ge r}
\bigg(\int_{\rho < |x| < 2\rho}|g|^2\bigg)^\frac{1}{2}\bigg(\int_{\rho < |x| < 2\rho}|\nabla(v-w)|^2\bigg)^\frac{1}{2}\nonumber\\
&\quad+\bigg(\int_{ |x| < r}|g|^2\bigg)^\frac{1}{2}\bigg(\int_{ |x| < r}|\nabla(v-w)|^2\bigg)^\frac{1}{2},
\end{align*}
where the sum runs over all radii $\rho$ that are a dyadic multiple of $r$. By (\ref{vol})
this implies
\begin{align*}
\bigg|\frac{1}{R^d}\int g\cdot\nabla(v-w)\bigg|&\lesssim\sum_{\rho\ge r}\bigg(\frac{\rho}{R}\bigg)^d
\bigg(\fint_{|x|>\rho}|g|^2\bigg)^\frac{1}{2}\bigg(\fint_{|x| < 2\rho}|\nabla(v-w)|^2\bigg)^\frac{1}{2}\nonumber\\
&\quad+\bigg(\frac{r}{R}\bigg)^d\bigg(\fint_{ |x| < r}|g|^2\bigg)^\frac{1}{2}\bigg(\fint_{ |x| < r}|\nabla(v-w)|^2\bigg)^\frac{1}{2},
\end{align*}
which allows to insert our assumption (\ref{Ld1}):
\begin{align*}
\bigg|\frac{1}{R^d}\int g\cdot\nabla(v-w)\bigg|&\lesssim\sum_{\rho\ge r}\bigg(\frac{\rho}{R}\bigg)^d
\bigg(\frac{r}{\rho}\bigg)^{\beta+d-1}\bigg(\fint_{|x|<2\rho}|\nabla(v-w)|^2\bigg)^\frac{1}{2}\nonumber\\
&\quad+\bigg(\frac{r}{R}\bigg)^d\bigg(\fint_{ |x| < r}|\nabla(v-w)|^2\bigg)^\frac{1}{2}.
\end{align*}
We distinguish the far-field, the intermediate and the near-field part:
\begin{align}
\lefteqn{\bigg|\frac{1}{R^d}\int g\cdot\nabla(v-w)\bigg|}\nonumber\\
&\lesssim\sum_{\rho\ge R}\bigg(\frac{\rho}{R}\bigg)^d
\bigg(\frac{r}{\rho}\bigg)^{\beta+d-1}\bigg[\bigg(\fint_{|x|<2\rho}|\nabla v|^2\bigg)^\frac{1}{2}
+\bigg(\fint_{|x|<2\rho}|\nabla w|^2\bigg)^\frac{1}{2}\bigg]\label{Ld8}\\
&\quad +\sum_{r\le\rho< R}\bigg(\frac{\rho}{R}\bigg)^d
\bigg(\frac{r}{\rho}\bigg)^{\beta+d-1}\bigg(\fint_{|x|<2\rho}|\nabla(v-w)|^2\bigg)^\frac{1}{2}\label{Ld7}\\
&\quad+\bigg(\frac{r}{R}\bigg)^d\bigg(\fint_{ |x| < r}|\nabla(v-w)|^2\bigg)^\frac{1}{2}.\label{Ld10}
\end{align}
On the second term in (\ref{Ld8}) we use (\ref{F2}) in form of
$\big(\fint_{|x|<2\rho}|\nabla w|^2\big)^\frac{1}{2}$ $\lesssim (\frac{\rho}{R})^{k-1}\big (\fint_{|x|<R}|\nabla w|^2\big)^\frac{1}{2}$
and then (\ref{Ld9}) to estimate it by
$(\frac{\rho}{R})^{k-1}\big(\frac{1}{R^d}\int|h|^2 \big)^\frac{1}{2}$; on the first term in (\ref{Ld8}) we use (\ref{Ld3}) in form of
$\big(\fint_{|x|<2\rho}|\nabla v|^2\big )^\frac{1}{2}$ $\lesssim\big(\frac{1}{\rho^d}\int|h|^2\big )^\frac{1}{2}$,
since this can be rewritten as $\big(\frac{R}{\rho})^{d/2}(\frac{1}{R^d}\int|h|^2\big )^\frac{1}{2}$ we see
that this contribution is of higher order with respect to the first one. On the terms in (\ref{Ld7}) and (\ref{Ld10})
we use (\ref{Ld11}). \BR5 Combining these \ER we see
\begin{align}
\bigg|\frac{1}{R^d}\int g\cdot\nabla(v-w)\bigg|
&\lesssim\bigg[ \sum_{\rho\ge R}\bigg(\frac{\rho}{R}\bigg)^d
\bigg(\frac{r}{\rho}\bigg)^{\beta+d-1}\bigg(\frac{\rho}{R}\bigg)^{k-1}\nonumber\\
&\quad+\sum_{r\le\rho< R}\bigg(\frac{\rho}{R}\bigg)^d
\bigg(\frac{r}{\rho}\bigg)^{\beta+d-1}\bigg(\frac{\rho}{R}\bigg)^{(k+1)-1}\nonumber\\
&\quad+\bigg(\frac{r}{R}\bigg)^d\bigg(\frac{r}{R}\bigg)^{(k+1)-1}\bigg]\bigg(\frac{1}{R^d}\int|h|^2\bigg)^\frac{1}{2}.\nonumber
\end{align}
Since $k<\beta<k+1$ we obtain for these geometric series
\begin{align}
\bigg|\frac{1}{R^d}\int g\cdot\nabla(v-w)\bigg|\lesssim 
\bigg(\frac{r}{R}\bigg)^{\beta+d-1}\bigg(\frac{1}{R^d}\int|h|^2\bigg)^\frac{1}{2}.\nonumber
\end{align}
By formula (\ref{Ld6}) and the fact that $h$ was an arbitrary function supported in $\{ |x| > R \}$
we obtain the second part of (\ref{Ld12}). The first part of (\ref{Ld12}) follows immediately
from assumption (\ref{Ld1}) in form of $(\frac{1}{R^d}\int|g|^2)\lesssim 1$ and (\ref{Ld13}).

\end{proof}


\begin{proof}[Proof of Lemma~\ref{Lup}]

For any radius $r_* \leq r \leq R$, we define 
\begin{align*}
u^r:= \argmin_{u'\in X_k} \biggl(\fint_{|x| < r} |\nabla( u - u') |^2 \biggr)^\oh.
\end{align*}
We start by showing that for $r_* \leq r \leq R$ it holds
\begin{align}\label{Lup1}
\biggl( \fint_{|x| < R} |\nabla( u^r - u^R)|^2 \biggr)^\oh \lesssim \inf_{u' \in X_k}\biggl(\fint_{|x| < R}|\nabla (u - u') |^2 \biggr)^\oh.
\end{align}
By the triangle inequality and a dyadic decomposition it is enough to show that, whenever $r \leq r_1 \leq 2r$, we have
\begin{equation}\label{Lup2}
\biggl( \fint_{|x| < R}|\nabla( u^r - u^{r_1})|^2 \biggr)^\oh \lesssim \biggl( \frac{r}{R} \biggr)^{(k+1)-1} \inf_{u' \in X_k} \biggl(\fint_{|x| < R}|\nabla( u - u')|^2 \biggr)^\oh,
\end{equation}
\BR5 where we recall that $k+1$ is the successor of $k$ in~\eqref{F1bis}. \ER 
Indeed, assumption \eqref{F2} on the space $X_k$ and the triangle inequality yield
\begin{align*}
\biggl(  &\fint_{|x| < R}|\nabla( u^r - u^{r_1})|^2 \biggr)^\oh\lesssim \biggl( \frac R r \biggr)^{k-1} \biggl( \fint_{|x| < r}|\nabla( u^r - u^{r_1})|^2 \biggr)^\oh\\
&\lesssim  \biggl( \frac R r \biggr)^{k-1} \biggl( \fint_{|x| < r_1}|\nabla( u - u^R - u^{r_1})|^2 \biggr)^\oh +  \biggl( \frac R r \biggr)^{k-1} \biggl( \fint_{|x| < r}|\nabla( u - u^R - u^r)|^2 \biggr)^\oh.
\end{align*}
By definition of the elements $u^r, u^{r_1}$ this yields
\begin{align*}
\biggl( \fint_{|x| < R}|\nabla( u^r - u^{r_1})|^2 \biggr)^\oh\lesssim & \biggl( \frac R r \biggr)^{k-1} \inf_{u' \in X_k} \biggl(\fint_{|x| < r_1}|\nabla (u - u^R - u')|^2 \biggr)^\oh\\
& +  \biggl( \frac R r \biggr)^{k-1} \inf_{u' \in X_k} \biggl(\fint_{|x| < r}|\nabla (u - u^R - u')|^2 \biggr)^\oh.
\end{align*}
We now apply assumption \eqref{F1} and $r_1 \leq 2r$ to conclude 
\begin{align*}
 \biggl( \fint_{|x| < R}|\nabla( u^r - u^{r_1})|^2 \biggr)^\oh\lesssim \biggl( \frac{r}{R}\biggl)^{(k+1)-1}\biggl(\fint_{|x| < R}|\nabla (u - u^R)|^2 \biggr)^\oh,
\end{align*}
i.e. inequality \eqref{Lup2} by definition of $u^R$. We thus established estimate \eqref{Lup1}.

\medskip

Let us now set $\tilde u  := u^{r_*}$. We prove inequality~\eqref{eqLup01} of the lemma.
By the triangle inequality we get for every $r \leq R$
\begin{align*}
\biggl( \fint_{|x|< r} |\nabla (u - \tilde u )|^2 \biggr)^\oh &\leq \biggl( \fint_{|x|< r} |\nabla (u - u^r)|^2 \biggr)^\oh + \biggl( \fint_{|x|< r} |\nabla (u^r - \tilde u )|^2 \biggr)^\oh,
\end{align*}
so that by \eqref{Lup1} with $r=r_*$ and $R=r$ we get
\begin{align}
\biggl( \fint_{|x|< r} |\nabla (u - \tilde u )|^2 \biggr)^\oh &\stackrel{\eqref{Lup1}}{\lesssim}\biggl( \fint_{|x|< r} |\nabla (u - u^r)|^2 \biggr)^\oh + \inf_{u' \in X_k}\biggl(\fint_{|x| < r}|\nabla (u - u') |^2 \biggr)^\oh\notag\\
&\stackrel{\eqref{F1}}{\lesssim}\biggl( \frac r R \biggr)^k \biggl( \fint_{|x|< R} |\nabla u |^2 \biggr)^\oh.\nonumber
\end{align}

\medskip

\BR5

We now turn to \eqref{cw05}: We note that, since $u$ is $a$-harmonic in $\{|x| < R \}$, the second inequality immediately follows from our assumption~\eqref{F1} with $k=0$. The first estimate is a consequence of the triangle inequality
\begin{align*}
\biggl(\fint_{|x|<r}|\nabla \tilde u|^2\biggl)^\frac{1}{2} \lesssim  \biggl(\fint_{|x|<r}|\nabla( u - \tilde u)|^2\biggl)^\frac{1}{2} + \biggl(\fint_{|x|<r}|\nabla u|^2\biggl)^\frac{1}{2},
\end{align*} 
and an application of \eqref{eqLup01} on the first term on the right-hand side with $R=r$ (which is allowed to apply, since the choice of $\tilde u$ is independent from the radius $R$).

\ER


\medskip

We conclude by establishing the Liouville property \eqref{Fischer2} for the spaces $X_m$, 
thereby generalizing the result of \cite{FischerOtto} to our abstract setting. The implication $\Leftarrow$ immediately follows from assumption \eqref{F2}.
We turn to $\Rightarrow$: Since $u$ is $a$-harmonic on the whole space, by the first inequality~\eqref{eqLup01} of this lemma we may find an element $\tilde u  \in X_{k-1}$ such that for every $ R \ge r \ge r_* $ it holds
\begin{align*}
\biggl( \fint_{|x| < r} |\nabla ( u -\tilde u )|^2 \biggr)^\oh \lesssim \BR5 \biggl( \frac{r}{R}  \biggr)^{k-1} \ER \biggl( \fint_{|x| < R}|\nabla u|^2 \biggr)^\oh.
\end{align*}
Therefore, \BR5 by letting $R \uparrow +\infty$ our assumption in~\eqref{Fischer2} \ER 
implies
\begin{align*}
\biggl( \fint_{|x| < r} |\nabla ( u -\tilde u )|^2 \biggr)^\oh = 0 \quad \textrm{for all } r \ge r_*,
\end{align*}
and thus $u=\tilde u \in X_{k-1}$. 

\end{proof}

\subsection{Proof of Lemma \ref{LF1} and of Corollary \ref{Lbil} }\hspace*{\fill} 

\begin{proof}[Proof of Lemma~\ref{LF1}]


The implication $c)\Rightarrow a)$ is obvious by definition of $Y_k(r)$.
The implication $a)\Rightarrow b)$ is a consequence of estimate (\ref{f48}) applied with $\tilde u$
playing the role of $u$; in this estimate, we let $R\uparrow\infty$ and use the defining properties
of $X_{k-1}$ and $Y_k(r)$ to infer $(\tilde u,v)=0$. 

\medskip

It thus remains to address $b)\Rightarrow c)$. We start by applying Lemma \ref{Lext} to $(v,0)$
playing the role of $(u,g)$. We may do so, since by
b) and $k\ge 1$ we have in particular that $( 1,v )=0$, which means that the total flux of the divergence-free vector
field $a\nabla v$ in $\{ |x| > r \}$ vanishes, which amounts to the first variant (\ref{Lext27})
of the second hypothesis of Lemma \ref{Lext}.
Hence there exists $(\bar v, \bar g)$ such that (\ref{Pd12}) and (\ref{Pd3}) hold, as well as (\ref{Pd1}) in form of
\begin{align}\label{Lext20}
\biggl(\frac{1}{r^d}\int|\nabla\bar v|^2+|\bar g|^2\bigg)^\frac{1}{2}
\lesssim\bigg(\fint_{|x| > r}|\nabla v|^2\bigg)^\frac{1}{2}.
\end{align}
We now argue that our assumption b) implies
\begin{align}\label{Lext18}
\int\nabla\tilde u\cdot \bar g=0\quad\mbox{for all}\;\tilde u\in X_{k-1}.
\end{align}
Indeed, by (\ref{bil.def}) and (\ref{Pd3}) we have 
$(\tilde u,v)$ $=\int\nabla\eta\cdot(\tilde ua\nabla \bar v-\bar va\nabla\tilde u)$
for every compactly supported $\eta$ with $\eta=1$ on $\{|x|<2r\}$. Using (\ref{Pd12}) and the $a$-harmonicity of 
$\tilde u$ we have $\int\nabla\eta\cdot(\tilde ua\nabla \bar v-\bar va\nabla \tilde u)$ 
$=\int\eta\nabla\tilde u\cdot \bar g$,
which by (\ref{Pd3}) is identical to $\int\nabla\tilde u\cdot \bar g$. Therefore, $(\tilde u,v)=0$
implies (\ref{Lext18}).

\medskip

We now resort to a duality argument to establish (\ref{Lext22}). 
To this purpose we give ourselves an arbitrary $R\ge 2r$ and an arbitrary square-integrable 
vector field $h$ supported in $\{ |x| > R \}$, and let $\tilde v$ be the
Lax-Milgram solution of 
\begin{align}\label{Lext17}
-\nabla\cdot a\nabla \tilde v=\nabla\cdot h.
\end{align}
We record the energy inequality in form of
\begin{align}\label{Lext16}
\biggl(\frac{1}{R^d}\int|\nabla \tilde v|^2\bigg)^\frac{1}{2}\lesssim\biggl(\frac{1}{R^d}\int|h|^2\bigg)^\frac{1}{2}.
\end{align}
Since by the support assumption on $h$, $\tilde v$ is $a$-harmonic in $\{ |x| < R \}$, we may apply our hypothesis
(\ref{F1}), which yields a $\tilde u\in X_{k-1}$ with
\begin{align}\label{Lext21}
\bigg(\fint_{ |x| < 2r}|\nabla(\tilde v-\tilde u)|^2\bigg)^\frac{1}{2}\lesssim\biggl(\frac{r}{R}\biggl)^{k-1}\bigg(\frac{1}{R^d}\int|h|^2\bigg)^\frac{1}{2},
\end{align}
where we used (\ref{Lext16}). Since both $\bar u$ and $\tilde v$ are Lax-Milgram solutions 
of (\ref{Pd12}) and (\ref{Lext17}), respectively, we have the formula
\begin{align*}
\int h\cdot\nabla\bar v=\int \bar g\cdot\nabla \tilde v,
\end{align*}
which with help of (\ref{Lext18}) 
we upgrade to
$\int h\cdot\nabla\bar v$ $=\int \bar g\cdot\nabla(\tilde v-\tilde u)$.
Since $h$ is supported in $\{ |x| > R \}\subset \{ |x| > 2r\}$ we obtain by (\ref{Pd3})
\begin{align}\label{Lext19}
\int h\cdot\nabla v=\int \bar g\cdot\nabla(\tilde v-\tilde u).
\end{align}

\medskip

>From formula (\ref{Lext19}) we obtain by the Cauchy-Schwarz inequality and the fact that $\bar g$ is supported
in $\{ |x| < 2r\}$, see (\ref{Pd3}),
\begin{align*}
\frac{1}{R^d}\biggl|\int h\cdot\nabla v\biggl|
\le\biggl(\frac{r}{R}\biggl)^d\biggl(\frac{1}{r^d}\int|\bar g|^2\bigg)^\frac{1}{2}\bigg(\fint_{ |x| < 2r}|\nabla(\tilde v-\tilde u)|^2\bigg)^\frac{1}{2},
\end{align*}
into which we insert (\ref{Lext20}) and (\ref{Lext21})
\begin{align*}
\frac{1}{R^d}\biggl|\int h\cdot\nabla v\biggl|
\le\bigg(\fint_{|x| > r}|\nabla v|^2\bigg)^\frac{1}{2}\biggl(\frac{r}{R}\biggl)^{k+d-1}\biggl(\frac{1}{R^d}\int|h|^2\bigg)^\frac{1}{2}.
\end{align*}
Since $h$ was arbitrary besides being supported in $\{ |x| > R \}$, we obtain (\ref{Lext22}).

\end{proof}


\begin{proof}[Proof of Corollary~\ref{Lbil}]

The first part of Corollary \ref{Lbil} is an immediate consequence of a) $\Rightarrow$ b) of Lemma \ref{LF1}. We turn to the non-trivial implication in \eqref{Liouv.decay}: Let $v \in Y_k$ be  $a$-harmonic in $\{|x| > r \}$ for some $r \geq r_*$.
By estimate \eqref{f48} shown in the proof of Proposition \ref{Pdual} 
 we have for every $u \in X_k$ and $R \geq r$
\begin{align*}
 |( u, v)| &\lesssim \biggl( \int_{|x|< 2R} |\nabla u|^2 \biggr)^\oh \biggl( \int_{|x| > R} |\nabla v|^2 \biggr)^\oh\\
 &\stackrel{\eqref{F2}}\lesssim\biggl(\frac{R}{r_*} \biggr)^{k- 1 + \frac d 2 }\biggl( \int_{|x|< r_*} |\nabla u|^2 \biggr)^\oh \biggl( \int_{|x| > R} |\nabla v|^2 \biggr)^\oh\\
& \lesssim  r_*^{-k + 1 - \frac d 2 }\biggl( \int_{|x|< r_*} |\nabla u|^2 \biggr)^\oh R^{d + k -1}\biggl( \fint_{|x| > R} |\nabla v|^2 \biggr)^\oh.
\end{align*}
Therefore, sending $R \uparrow +\infty$ yields $(u,v) = 0$ for every $u \in X_k$.
By Lemma \ref{LF1} this is equivalent to $v \in  Y_{k+1}$. This yields the $\Rightarrow$ implication in \eqref{Liouv.decay} and thus concludes the argument for Corollary \ref{Lbil}.
\end{proof}

\subsection{Proof of Lemmas~\ref{Lext} and \ref{propositionlocalize}}\hspace*{\fill} 

\begin{proof}[Proof of Lemma~\ref{Lext}]

Let $\eta$ be a cut-off function for $\{ |x| < r \}$ in $\{ |x| < 2r\}$ with $|\nabla\eta|\lesssim\frac{1}{r}$.
We set $\bar u:=(1-\eta)(u-c)$ where
$c$ is a constant to be adjusted later and where we think of $\bar u$ as being extended by zero on $\{ |x| < r \}$. 
This definition implies
\begin{align}\label{Lext5}
\nabla\bar u=(1-\eta)\nabla u-(u-c)\nabla\eta,
\end{align}
which with (\ref{Lext1}) combines to
\begin{align}\label{Lext3}
-\nabla\cdot a\nabla\bar u=\nabla\cdot\big((1-\eta)g+(u-c)a\nabla\eta\big)
+\nabla\eta\cdot(a\nabla u+g).
\end{align}

\medskip

We now claim for the last term on the right-hand side of (\ref{Lext3}) that
\begin{align}\label{Lext10}
\int_{r < |x| < 2r}\nabla\eta\cdot(a\nabla u+g)=0.
\end{align}
We distinguish the cases of hypothesis (\ref{Lext27}), in which case (\ref{Lext10}) is immediate, and 
hypothesis (\ref{Lext2}).
In case of the latter, by (\ref{Lext1}), $a\nabla u+g$ is a divergence-free in 
$\{ |x| > r \}$. Hence for $R\gg r$ and a cut-off function $\eta_R$ for $\{ |x| < R \}$ in $ \{ |x| < 2R \}$ with $|\nabla\eta_R|\lesssim\frac{1}{R}$
we have
\begin{equation*}
\biggl|\int_{r < |x| < 2r}\nabla\eta\cdot(a\nabla u+g)\biggl|
=\biggl|\int_{R < |x| < 2R}\nabla\eta_R\cdot(a\nabla u+g)\biggl| \lesssim \frac{1}{R}\int_{R < |x| < 2R}|a\nabla u+g|,
\end{equation*}
so that (\ref{Lext10}) follows from (\ref{Lext2}) in the limit $R\uparrow\infty$. 
\BR5 In view of~\eqref{Lext10} \ER there exists a function $w$
that \BR5 solves \ER the Neumann boundary problem on the annulus $\{ r < |x| < 2r \}$
\begin{align*}
-\nabla\cdot a\nabla w&=\nabla\eta\cdot(a\nabla u+g)\quad&&\mbox{in}\;\{ r < |x| < 2r \},
\\
\nu\cdot a\nabla w&=0\quad&&\mbox{on}\;\partial(\{ r < |x| < 2r \}),
\end{align*}
interpreted in the weak sense, so that if we extend $-a\nabla w$ by zero outside the annulus to
a vector field $\tilde g$ we have
\begin{align*}
\nabla\cdot\tilde g=\nabla\eta\cdot(a\nabla u+g)\quad\mbox{everywhere}.
\end{align*}
Hence if we set 
\begin{align}\label{Lext6}
\bar g:=(1-\eta)g+(u-c)a\nabla\eta+\tilde g,
\end{align}
identity (\ref{Lext3}) implies the desired equation (\ref{Pd12}). The extension
property (\ref{Pd3}) follows from (\ref{Lext5}) and (\ref{Lext6}) by the support properties of $\eta$ and~$\tilde g$. 

\medskip

It remains to establish the estimate (\ref{Pd1}). Using the triangle inequality on (\ref{Lext5}) in form of
\begin{align*}
\biggl(\int_{ |x| < 2r}|\nabla\bar u|^2\bigg)^\frac{1}{2}\lesssim\biggl(\int_{r < |x| < 2r}|\nabla u|^2\bigg)^\frac{1}{2}
+\frac{1}{r}\biggl(\int_{r < |x| < 2r}(u-c)^2\bigg)^\frac{1}{2}
\end{align*}
the desired
\begin{align*}
\biggl(\int_{ |x| < 2r}|\nabla\bar u|^2\bigg)^\frac{1}{2}\lesssim\biggl(\int_{r < |x| < 2r}|\nabla u|^2\bigg)^\frac{1}{2}
\end{align*}
follows from our hypothesis (\ref{poinc}) of Poincar\'e's inequality. 
In view of (\ref{Lext6}), the $\bar g$-part of (\ref{Pd1}), that is,
\begin{align*}
\biggl(\int_{ |x| < 2r}|\bar g|^2\bigg)^\frac{1}{2}\lesssim\biggl(\int_{r < |x| < 2r}|\nabla u|^2+|g|^2\bigg)^\frac{1}{2},
\end{align*}
follows from the same argument provided we establish
\begin{align*}
\biggl(\int_{ |x| < 2r}|\tilde g|^2\bigg)^\frac{1}{2}\lesssim\biggl(\int_{r < |x| < 2r}|\nabla u|^2+|g|^2\bigg)^\frac{1}{2}.
\end{align*}
In view of the definition of $\tilde g$ as the extension of $-a\nabla w$, where $w$ solves the
Neumann problem on the annulus with right-hand side $f:=\nabla\eta\cdot(a\nabla u+g)$ (of vanishing average),
it suffices to show
\begin{align*}
\biggl(\int_{r < |x| < 2r}|\nabla w|^2\bigg)^\frac{1}{2}\lesssim R\biggl(\int_{r < |x| < 2r}f^2\bigg)^\frac{1}{2}.
\end{align*}
In order to see this, we test the Neumann problem with $w-\bar c$ for some constant $\bar c$
and apply once more Poincar\'e's inequality (\ref{poinc}).
\end{proof}

\medskip

{ \newcommand{\W}{\bar w}
\begin{proof}[Proof of Lemma~\ref{propositionlocalize}]

\BR5
For any $\rho$ dyadic multiple of $r_*$ smaller than $R$ we define
\begin{align}\nonumber
h_\rho:=I(\{ \rho < |x| < 2\rho\})g\quad\mbox{for}\;r_*<\rho<R\quad\mbox{and}\quad h_{r_*}:=I(\{ |x| < r_* \})g,
\end{align}
where $I(B)$ denotes the characteristic function of the set $B$. Without loss of generality we assume that $ R = 2^N r_*$, since otherwise we can replace $R$ with the largest $2^Nr_*$ which is not larger than $R$.

\medskip
Let us define $\W= \sum_{\rho<R} w_\rho$ with $ w_\rho$ being the Lax-Milgram solution of 
\begin{equation}\label{6.1}
-\nabla\cdot a\nabla  w_\rho=\nabla\cdot h_\rho. 
\end{equation}

Since by \eqref{ld.1b} and by the definition of $\W$ the difference $w-\W$ is $a$-harmonic in $\{ |x| < R \}$, we may appeal to \BR5 the mean-value property \eqref{cw05}\ER, the triangle inequality and the first estimate in \eqref{properties} to get
\begin{align*}
\biggl(\fint_{|x|<r_*}|\nabla (\W - w)|^2\biggr)^\oh &{\lesssim} \biggl(\fint_{|x|<R}|\nabla (\W - w)|^2\biggr)^\oh \lesssim \biggl(\fint_{|x|<R}|\nabla \W |^2\biggr)^\oh + 1\\
&{\le} \sum_{\rho<R} \biggl(\fint_{|x|<R}|\nabla w_\rho|^2\biggr)^\oh + 1.
\end{align*}
The energy estimate for \eqref{6.1} and the definition of $h_\rho$ then imply
\begin{equation*}
 \biggl(\fint_{|x|<r_*}|\nabla (\W - w)|^2\biggr)^\oh \lesssim \sum_{\rho<R} \biggl(\frac{1}{R^d}\int | h_\rho |^2\biggr)^\oh + 1 
 \lesssim \sum_{\rho<R} \biggl(\fint_{|x|< 2\rho}|h |^2\biggr)^\oh + 1.
\end{equation*}
Hence we are done once we establish
\begin{equation}\label{6.3b}
\biggl(\fint_{|x|<r_*}|\nabla \W|^2\biggr)^\oh \lesssim \sum_{\rho<R} \biggl( \fint_{|x|< 2\rho}| h|^2 \biggr)^\oh.
\end{equation}
Indeed, using the triangle inequality, the two estimates above and the second estimate in \eqref{properties} yields \eqref{ld.4loc}.

\medskip

We argue for \eqref{6.3b} in the following way: For every $r_* < \rho<R$, dyadic multiple of $r_*$, the construction of $ h_\rho$ and equation (\ref{6.1}) imply that the function $ w_\rho$ is $a$-harmonic in $\{|x| < \rho \}$, and thus we may use the \BR5 mean-value property \eqref{cw05} from Lemma~\ref{Lup} \ER and the energy estimate for (\ref{6.1}) to bound
\begin{equation*}
\biggl(\fint_{|x|<r_*}|\nabla w_\rho|^2\biggr)^\oh
\lesssim
\biggl(\fint_{|x|< \rho}|\nabla w_\rho|^2\biggr)^\oh\lesssim \BR5 \biggl(\fint_{|x|< 2\rho}| h|^2\biggr)^\oh.
\end{equation*}
In the case of $\rho=r_*$, the energy estimate alone yields
\begin{equation*}
\biggl(\fint_{|x|<r_*}|\nabla w_{r_*}|^2\biggr)^\oh\lesssim\biggl(\fint_{|x|<r_*}|h|^2\biggr)^\oh.
\end{equation*}
Hence, for $r_* \le \rho <R$ we have that
\begin{equation*}
\biggl(\fint_{|x|<r_*}|\nabla w_\rho|^2\biggr)^\oh {\lesssim} \biggl( \fint_{|x|< 2\rho}| h|^2 \biggr)^\oh.
\end{equation*}
This, together with the triangle inequality for $\W$
\begin{equation}\nonumber
\biggl(\fint_{|x| < r_*}|\nabla  \W|^2\biggr)^\oh\le
\sum_{\rho<R}\biggl(\fint_{|x|< r_*}|\nabla  w_\rho|^2\biggr)^\oh,
\end{equation}
allows us to conclude \eqref{6.3b} and thus also the proof of this lemma. 
\end{proof}
}

\medskip



\section{Auxiliary results in the constant-coefficient case}\label{constant.r}
The characterization of the spaces $\{X_m^\h\}_m$ and $\{Y_k^\h\}_k$, that is,
in the Euclidean case of a constant coefficient $\ah$, is folklore. Since for part ii) of Lemma \ref{folklore} we could not find a proof
in case of systems, and since it is a pleasing application of the results of Section \ref{abstract.r}, we give
a proof of Lemma \ref{folklore}.

\begin{lemma}\label{folklore}
i) For all $m\ge 0$, $X^\h_m$ consists of polynomials of degree $\le m$. The space $W_m^\h\subset X^\h_m$ of polynomials
homogeneous of degree $m$ is a complement of $X^\h_{m-1}$ in $X^\h_{m}$ (and thus is isomorphic to $X_m^\h/X_{m-1}^\h$).

\smallskip

ii) Let $G$ denote the fundamental solution of $-\nabla\cdot \ah^*\nabla$. 
Then for all $k\ge 1$, the linear span $Z_k^\h$ of the set of functions $\{\partial^\alpha G\}_{\alpha}$, where $\alpha$ runs
over all multi-indices of degree $k$, provides a complement for $Y_{k+1}^\h$ in $Y_k^\h$ (and thus is isomorphic to $Y_k^\h/Y_{k+1}^\h$).

\smallskip

iii) For all $m\ge 2$, $(\cdot,\cdot)_\h$ provides an isomorphism between
$Y_k^\h/Y_{m+1}^\h$ and $(X_m^\h/X_{k-1}^\h)^*$ and between $(Y_k^\h/Y_{m+1}^\h)^*$ and $X_m^\h/X_{k-1}^\h$.

\smallskip

iv) For all $k\ge 1$ and $r>0$, the projection $Y_k^\h(r)\ni v\mapsto w\in Z_k^\h$ defined
by the direct sum $Y_k^\h(r)=Z_{k}^\h\oplus Y_{k+1}^\h(r)$ provided in ii) is continuous
in the sense of
\begin{align}\label{ao04}
\bigg(\fint_{R<|x|<2R}|\nabla v_k|^2\bigg)^\frac{1}{2}\lesssim \bigg(\fint_{R<|x|<2R}|\nabla v|^2\bigg)^\frac{1}{2}
\quad\mbox{for all}\;R\ge r.
\end{align}

\end{lemma}

The following lemma defines a linear map $Y_k^{\h}/Y_{k+1}^{\h}\ni v\mapsto\vt$,
which canonically (and trivially) extends to a map $Y_k^{\h}\ni v\mapsto\vt$,
needed in Theorem \ref{main} to construct the isomorphism $Y_k^{\h}/Y_{k+2}^\h\cong Y_{k}/Y_{k+2}$
in the non-symmetric case.
 
\begin{lemma}\label{Rom1}
Let $k\ge 1$ be given and let us identify the quotient space 
$X_{k+1}^\h/X_{k}^\h$ with the space $W_{k+1}^\h$ of $\ah$-harmonic polynomials homogeneous of degree $k+1$, 
see part i) of Lemma \ref{folklore},
and the quotient space $Y_k^\h/Y_{k+1}^\h$
with  $Z_k^\h:=\text{span}\{\partial^\alpha G\}_{|\alpha|=k}$, see part ii) of Lemma \ref{folklore}.

\smallskip

In this sense for every $v\in Y_k^\h/Y_{k+1}^\h$ there exists a unique function 
$\vt$, homogeneous of degree $-(d-2)-(k+1)$, such that
\begin{align}\label{R2}
-\nabla\cdot \ah^*\nabla\vt&=\nabla\cdot \partial_{ij}vC_{ij}^{*,\mathrm{sym}}\quad\mbox{in}\;\mathbb{R}^d \backslash \{0\},
\end{align}
and such that for all $u\in X_{k+1}^\h/X_k^\h$ and $R> 0$
\begin{align}\label{R1}
\int_{|x| = R}\nu \cdot\big(\vt \ah\nabla u&-u(\ah^*\nabla\vt + \partial_{ij}vC_{ij}^{*,\mathrm{sym}})\big)
=\int_{|x| = R} \nu_k \partial_i v \partial_j u C_{ijk}^{\mathrm{sym}}.
\end{align}
\end{lemma} 

Quite analogously, in order 
to construct the isomorphism $X_m^{\h}/X_{m-2}^\h\cong X_{m}/X_{m-2}$
in the non-symmetric case, we need to define a suitable map
$X_m^{\h}/X_{m-1}^{\h}\ni u\mapsto\ut$.

\begin{lemma}\label{Rom2}
Let $m\ge 3$ be given and let us as in Lemma \ref{Rom1} identify the quotient space
$X_{m}^\h/X_{m-1}^\h$ with $\ah$-harmonic polynomials homogeneous of degree $m$
see part i) of Lemma \ref{folklore},
and the quotient space $Y_{m-1}^\h/Y_{m}^\h$
with the span of $\{\partial^\alpha G\}_{|\alpha|=m-1}$, see part ii) of Lemma \ref{folklore}.

\smallskip

In this sense for every $u\in X_m^\h/X_{m-1}^\h$ there exists a unique polynomial
$\ut$, homogeneous of degree $m-1$, such that
\begin{align}\label{R7}
-\nabla\cdot \ah\nabla\ut&=\nabla\cdot \partial_{ij}uC_{ij}^{\mathrm{sym}}
\end{align}
and such that for all $v\in Y_{m-1}^\h/Y_{m}^\h$ and $R> 0$
\begin{align}\label{R11}
\int_{|x| = R} \nu \cdot\big(v(\ah\nabla\ut+\partial_{ij}uC_{ij}^{\mathrm{sym}})-\ut \ah^*\nabla v\big)= 0.
\end{align}
\end{lemma}


\subsection{Proofs of Lemmas~\ref{folklore}, \ref{Rom1}, and~\ref{Rom2}}\hspace*{\fill} 

\begin{proof}[Proof of Lemma~\ref{folklore}]

In the case of systems, where ``functions'' $u(x)$ have values in some finite-dimensional vector space $V$,
the fundamental solution $G(x)$ is in fact a field with values in the space of linear functions from $V$ to $V$.
Hence by the statement i) we mean that the elements are polynomials in $x$ with coefficients in $V$
and by statement ii) we mean that the set $\{\partial^\alpha G v\}_{\alpha,v}$,
where $\alpha$ runs over multi-indices with $|\alpha|=k$ and $v$ runs over vectors in $V$,
generates $Y_k^\h/Y_{k+1}^\h$.

\medskip

We start with the argument for part i): The first statement is a consequence of \cite[Proposition 2.1 and discussion at p.83]{Giaquinta}, 
Caccioppoli and Liouville. The second statement is elementary algebra.

\medskip

We now turn to part iii).
This follows by Propositions \ref{Pdual} and \ref{Pbidual} and we have to check the assumptions
on the sequence of spaces $\{X_k\}_k$ required in Section~\ref{abstract.r}: Assumption (\ref{F2}) follows immediately from the characterization in part i);
assumption \eqref{F1} is a consequence of the $C^{1,\alpha}$-regularity of $\ah$-harmonic functions,
which extends to $C^{k,1}$-regularity, see for instance \cite{Giaquinta}.
Assumptions (\ref{vol}) and (\ref{poinc}) on the ambient space are trivially satisfied for $\mathbb{R}^d$.

\medskip

We now tackle part ii). By (\ref{R5}), see the proof of Lemma \ref{Rom1}, we have that
for $\alpha$ with $|\alpha|=k$ and $k\ge 1$,
$\nabla\partial^\alpha G$ is homogeneous of degree $-(d-2)-k-1$ --- next to $\partial^\alpha G$ being $\ah^*$-harmonic in
$\mathbb{R}^d-\{0\}$ and thus by definition $\partial^\alpha G\in Y_k^\h$ and $Z_k^\h\cap Y_{k+1}^\h=\emptyset$.
By part iii) we have in particular that $(\cdot,\cdot)_\h$ provides an isomorphism between
$Y_k^\h/Y_{k+1}^\h$ and $(X_k^\h/X_{k-1}^\h)^*$. Hence for $Z_k^\h$ to be a complement of $Y_{k+1}^\h$ in $Y_k^\h$,
where it remains to argue that $Y_k^\h=Z_k^\h \BR5 \oplus Y_{k+1}^\h$,
by linear algebra it is enough to show the following implication for any $u\in X_k^\h$
\begin{align}\label{o09}
\forall\;|\alpha|=k\;\;(u,\partial^\alpha G)_\h=0\quad\Rightarrow\quad u\in X_{k-1}^\h.
\end{align}
Indeed, on the one hand by definition of $(\cdot,\cdot)_\h$ we have that
$(u,\partial^\alpha G)_\h$ $=\int u(-\nabla\cdot \ah^*\nabla \partial^\alpha G)$
$=(-1)^k\int\partial^\alpha u(-\nabla\cdot \ah^*\nabla G)$ $=(-1)^k\partial^\alpha u(0)$.
On the other hand, since $|\alpha|=k$ and
since $u$ is in $X_k^\h$ and thus by part i) (componentwise)
a polynomial of degree $\le k$, $\partial^\alpha u$ is constant. Hence
the left-hand side  of (\ref{o09}) yields that $\partial^\alpha u$ vanishes identically provided $|\alpha|=k$.
This implies that $u$ is a polynomial of degree $k-1$; since $u$ is also $\ah$-harmonic, it is as desired
an element of $X_{k-1}^\h$.

\medskip

We finally establish iv). By scale invariance, we may assume $R=1$. Because of 
the isomorphisms $Z_k^\h\cong Y_k^\h/Y_{k+1}^\h\cong (X_k^\h/X_{k-1}^\h)^*\cong W_k^\h$ via $(\cdot,\cdot)_\h$,
and the equivalence of norms on finite-dimensional spaces, we have for our $w\in Z_k^\h$
\begin{align}\label{ao02}
\bigg(\int_{1<|x|<2}|\nabla w|^2\bigg)^\frac{1}{2}\lesssim \sup_{u\in W_k^\h}\frac{(u,w)_\h}{(\int_{1<|x|<2}|\nabla u|^2)^\frac{1}{2}}.
\end{align}
An inspection of the proof of (\ref{f48}) actually shows that it holds in the strengthened form of
\begin{align}\label{ao01}
|(u,v)_\h|\lesssim\bigg(\int_{R<|x|<2R}|\nabla u|^2\bigg)^{\frac12}\bigg(\int_{R<|x<2R}|\nabla v|^2\bigg)^\frac{1}{2}.
\end{align}
Since in our case $v-w\in Y_{k+1}^\h$ we thus obtain from using (\ref{ao01}) with $v-w$ playing the role of $v$
and letting $R\uparrow\infty$
\begin{align}\label{ao03}
(u,w)_\h=(u,v)_\h\quad\mbox{for all}\;u\in X_k^\h.
\end{align}
Using now (\ref{ao01}) for $R=1$ and (\ref{ao03}) in (\ref{ao02}) yields the desired (\ref{ao04}).
\end{proof}


\begin{proof}[Proof of Lemma~\ref{Rom1}]

Let us start with a remark on the homogeneity of fundamental solutions; next to the
fundamental solution $G$ of $-\nabla\cdot \ah^*\nabla$, we also consider the fundamental solution
$K$ of the ``bi-Laplacian'' $(-\nabla\cdot \ah^*\nabla)^2$. (Note that despite the fact that $\ah$
is constant, we may have $-\nabla\cdot \ah^*\nabla\not=-\nabla\cdot \ah\nabla$
in case of systems). While fundamental solutions may not
be homogeneous and not even unambiguously defined (like $G$ in $d=2$ or $K$ in $d=4$ when logarithms appear, 
at least in case
of $\ah^*={\rm id}$), sufficiently high derivatives of fundamental solutions are. More precisely, we claim
that irrespective of the dimension $d$,
\begin{align}
\forall\;|\alpha|\ge 2\quad\partial^\alpha G\;\mbox{is homogeneous of degree}\;-(d-2)-|\alpha|,\label{R5}\\
\forall\;|\alpha|\ge 4\quad\partial^\alpha K\;\mbox{is homogeneous of degree}\;-(d-4)-|\alpha|.\label{R4}
\end{align}
In particular, provided $|\alpha|\ge 4$, $\partial^\alpha K$ is unambiguously
defined and we have in a distributional sense
\begin{align}\label{R3}
\forall\;|\alpha|\ge 4\quad
-\nabla\cdot \ah^*\nabla\partial^\alpha K=\partial^\alpha G\quad\mbox{in}\;\mathbb{R}^d-\{0\}.
\end{align}
The claims (\ref{R5}) and (\ref{R4}) can be easily seen on the level of Fourier transform ${\mathcal F}$, where we focus on (\ref{R4}). 
In this description we formally have for $x\not=0$
\begin{align}\label{R6}
({\mathcal F}\partial^\alpha K)(\xi)=(-i)^{|\alpha|}
\int \xi^\alpha\exp(-i\xi\cdot x)(\xi\cdot \ah^*\xi)^{-2}d\xi,
\end{align}
where we use the standard notation of $\xi^\alpha=\xi_1^{\alpha_1}\cdots\xi_d^{\alpha_d}$.
Since we allow for systems, where the ``functions'' have values in some finite-dimensional Euclidean space
$V$ and where the fundamental solutions have values in the space of linear functions ${\mathcal L}(V,V)$,
$\xi\cdot a^*_h\xi$ in fact is also an element of ${\mathcal L}(V,V)$, which by ellipticity is invertible
and by $(\xi\cdot \ah^*\xi)^{-2}$ $\in {\mathcal L}(V,V)$ we mean the square of the inverse. 
Provided we can give a sense to the integral (\ref{R6}), the desired homogeneity 
$K(\frac{x}{\ell})$ $=\ell^{|\alpha|-4+d}K(x)$ follows from
a change of variables $\hat\xi=\ell\xi$. In view of our assumption of $|\alpha|\ge 4$, the 
integrand in fact is bounded so the issue is the summability for $|\xi|\uparrow\infty$.
To infer the latter, it is enough to show the summability of
$\int(1-\eta)\xi^\alpha\exp(-i\xi\cdot x)(\xi\cdot \ah^*\xi)^{-2}d\xi$
with $\eta=\eta(\xi)$ a smooth cut-off function for $\{|x|<1\}$ in $\{|x|<2\}$. However, in this expression,
we may formally integrate by parts ($m$ times) according to 
$(-1)^m(\frac{\partial}{\partial\xi_1})^m\frac{1}{i^m x_1^m}\exp(-i\xi\cdot x)$ $=\exp(-i\xi\cdot x)$,
where we assume that w.\ l.\ o.\ g.\  the first component of $x$ does not vanish. After these $m$ 
integrations by parts, the integrand, which outside of $\{|x|<2\}$ is a homogeneous rational function in $\xi$ with
values in ${\mathcal L}(V,V)$, now has homogeneity $|\alpha|-4-m$ and thus is 
absolutely summable provided $m>|\alpha|-4+d$.

\medskip

Let us now address uniqueness: Suppose that we had another function with these properties, then 
in view of (\ref{R2}) their difference $w$ is $\ah^*$-harmonic in $\mathbb{R}^d-\{0\}$
so that because $w$ decays at rate $-(d-2)-(k+1)$ we in particular have $w\in Y_{k+1}^\h$.
In view of (\ref{R1}), $(u,w)_\h$ vanishes for all $u\in X_{k+1}^\h/X_{k}^\h$, so that 
by part iii) of Lemma \ref{folklore} in form of $Y_{k+1}^\h/Y_{k+2}^\h$ $\cong(X_{k+1}^\h/X_{k}^\h)^*$,
we have $w\in Y_{k+2}^\h$. Since $w$ however is homogeneous of degree $-(d-2)-(k+1)$, this yields $w=0$.
 
\medskip

Let us now argue that it is enough to establish the existence of a $\vt$, homogeneous of degree
$-(d-2)-(k+1)$, with just (\ref{R2}) and not necessarily (\ref{R1}). Indeed, $\partial_{ij}v\,C_{ij}$, as a linear combination of
directional derivatives of $G$ of order $k+2$, is homogeneous of degree $-(d-2)-(k+2)$ $=-(d-1)-(k+1)$
by (\ref{R5}).
Together with the homogeneity of $\vt$ of order $-(d-2)-(k+1)$ and that of $u$ of order
$k+1$ we see that the integrands  in (\ref{R1}) is homogeneous of degree $-(d-1)$. This implies that the surface integrals are independent of $R$. Their difference thus defines a linear form $\ell$ on $X_{k+1}^\h/X_k^\h$. By the isomorphism
$(X_{k+1}^\h/X_k^\h)^*$ $\cong Y_{k+1}^\h/Y_{k+2}^\h$, see part iii) of Lemma \ref{folklore}, 
there exists a $w\in Y_{k+1}^\h$ such that $\ell.u$ $=(u,w)$ for all $u\in X_{k+1}^\h/X_k^\h$. 
By part ii) of Lemma \ref{folklore} we may identify $w$ with a linear combination of 
$\{\partial^\beta G\}_{|\beta|=k+1}$,
so that by (\ref{R5}) and because of our assumption $k\ge 1$, $w$ is homogeneous of degree $-(d-2)-(k+1)$ next to being $\ah^*$-harmonic outside the origin.
Therefore the modified $\ut-w$ retains the properties of $\ut$ while satisfying also
(\ref{R1}).

\medskip
 
We finally turn to the construction of $\vt$ given $v$.
By linearity we may w.l.o.g. assume $v=\partial^\alpha G$ for some multi-index $\alpha$ with $|\alpha|=k$.
We then set
\begin{align*}
\vt:=-\nabla\cdot\partial_{ij}\partial^\alpha K C_{ij}.
\end{align*}
Hence $\vt$ is a linear combination of $\{\partial^\beta K\}_{|\beta|=k+3}$. Since 
by assumption $k\ge 1$,
by (\ref{R4}), $\vt$ has the desired homogeneity of $-(d-4)-(k+3)=-(d-2)-(k+1)$
and by (\ref{R3}), we obtain the desired equation (\ref{R2}). 
\end{proof}



\begin{proof}[Proof of Lemma~\ref{Rom2}]

We start by arguing in favor of uniqueness. Clearly, the difference $w$ of two admissible $\ut$'s is $\ah$-harmonic
in view of (\ref{R7}), and thus, because of its homogeneity of order $m-1$, an element of $X_{m-1}^\h$.
Moreover, in view of (\ref{R11}), it satisfies $(w,v)_\h=0$ for all $v\in Y_{m-1}^\h/Y_m^\h$. By Lemma \ref{folklore}, iii) it holds
\begin{align}\label{R10}
X_{m-1}^\h/X_{m-2}^\h\cong(Y_{m-1}^\h/Y_m^\h)^*\quad\mbox{via}\;(\cdot,\cdot)_\h,
\end{align}
and thus $w$ must be an element of $X_{m-2}^\h$, thus be of growth rate $<(m-1)$, which is not compatible with a
homogeneity of $m-1$ unless $w=0$.

\medskip

We now argue that for any homogeneous polynomial $u$ of degree $m$ there exists a
homogeneous polynomial $\ut$ of degree $m-1$ such that (\ref{R7}) holds. We pick
a smooth cut-off function $\eta$ for $\{|x|<1\}$ in $\{|x|<2\}$ and let $\ut'$ be the Lax-Milgram solution of
\begin{align}\label{R9}
-\nabla\cdot \ah\nabla\ut'=\nabla\cdot \eta \partial_{ij}uC_{ij}^{\mathrm{sym}},
\end{align}
which as a constant-coefficient solution is smooth.
We claim that the leading order (i.e. $m-1$-homogeneous) part of the Taylor polynomial $\ut''$ of $\ut'$ at $0$
of degree $m-1$ has the desired property. We first note that it is enough to show that the
Taylor polynomial $\ut''$ of degree $m-1$ satisfies
\begin{align}\label{R8}
-\nabla\cdot \ah\nabla\ut''=\nabla\cdot \partial_{ij}uC_{ij}^{\mathrm{sym}}.
\end{align}
Indeed, $-\nabla\cdot \ah\nabla$ maps the part of $\ut''$ that is of degree $<m-1$ 
onto a polynomial of degree $<m-3$. Since the right-hand side of (\ref{R8}) is a polynomial that is homogeneous of degree
$m-3$, that polynomial vanishes. Hence the part of $\ut''$ that is of degree $<m-1$                        
is $\ah$-harmonic and thus can be subtracted off $\ut''$ to obtain $\ut$.
It is easy to pass from (\ref{R9}) to (\ref{R8}): Clearly, $-\nabla\cdot \ah\nabla \ut''$
is the Taylor polynomial of $-\nabla\cdot \ah\nabla \ut'$ at $0$ of degree $m-3$. By (\ref{R9}),
the latter agrees with $\nabla\cdot \partial_{ij}uC_{ij}$ in a neighborhood of $0$. Since the latter
is itself a polynomial of degree $m-3$, the two polynomials of degree $m-3$ must coincide, which amounts to (\ref{R8}).

\medskip

We finally argue that we may modify the $\ut$ constructed in the previous paragraph to also satisfy (\ref{R11}).
Since with our understanding of $Y_{m-1}^\h/Y_{m}^\h$ as being spanned by $\{\partial^\alpha G\}_{|\alpha|=m-1}$
and because of (\ref{R5}) together with our assumption $m\ge 3$, we see that $v$ in (\ref{R11}) is homogeneous
of degree $-(d-1)-(m-1)$. Since $u$ is a polynomial of degree $m$ and $\ut$ one of degree $m-1$, the homogeneity
of the terms in the integrals in (\ref{R11}) is of order $-(d-1)$. Hence, both surface integrals are independent of $R$ and their difference defines
a linear form on $Y_{m-1}^\h/Y_{m}^\h$. Hence by (\ref{R10}), there exists an element $w\in X_{m-1}^\h/X_m^\h$,
which we may take to be a polynomial of degree $m-1$, such that
\begin{align*}
(w,v)=
\int_{\partial |x|<R}\nu\cdot\big(v(\ah\nabla\ut+\partial_{ij}uC_{ij})-\ut \ah^*\nabla v + \partial_i u\partial_j v \, C_{ij}^{\mathrm{sym}}\big) \quad
\mbox{for all}\;R>0
\end{align*}
for all $v\in Y_{m-1}^\h/Y_{m}^\h$. It remains to replace $\ut$ by $\ut-w$.
\end{proof}


\section{Proofs of Theorem~\ref{Rom3} and Theorem~\ref{T3}}\label{deterministic.r}
We presently list the main ingredients needed in the proofs of Theorem~\ref{Rom3} and Theorem~\ref{T3}. Throughout this section the notation $\lesssim$ stands for $\leq C$ with the constant depending on the dimension $d$, the ellipticity ratio $\lambda$, the exponents $\beta$ and $\alpha$ and the order of the space in consideration ($m$ for $u \in X_m$ or $X_m^\h$ and $k$ for $v \in Y_k$ or $Y_k^\h$).

\begin{proposition}\label{Rom4}
For $m\ge 2$ there exists a linear map $X_m^\h\ni u_\h\mapsto u\in X_m$ with $\|u-Eu_\h\|_{m-\beta}$ $\lesssim\|u_\h\|_m$.

\smallskip

Likewise, for $k\ge 1$ and $r\ge r_*$, there exists a linear map $Y_k^\h(r)\ni v_\h\mapsto v\in Y_k(r)$ 
with $\|v-Ev_\h\|_{k+\beta,r}$ $\lesssim\|v_\h\|_{k,r}$.

\end{proposition}

\begin{proposition}\label{localization}
Let $u_\h \in X_m^\h$ and $u \in X_m$ be such that \BR5 $\| u - Eu_\h \|_{m-\beta}\le C_1 \| u_\h \|_m$ for some constant $C_1 \ge 1$. \ER Then, under the assumptions of Theorem \ref{T3}  it holds
\begin{align}
\biggl(\fint_{|x-y| < r_*(y)}|\nabla( u - Eu_\h)|^2\biggr)^\frac{1}{2}&\lesssim \BR5 C_1 \ER \biggl(\frac{|y|}{ r_*(0) }\biggr)^{m - 1 - \beta} \| u_\h\|_{m}\label{g.loc}.
\end{align}
Similarly, for $r \geq r_*(0)$ let $v_\h\in Y_{k}^\h(r)$ and $v \in Y_k(r)$ satisfy $\| v - Ev_\h\|_{k+\beta,r}\le \BR5 C_1 \|v_\h \|_{k, r}$ for some constant $C_1 \ge 1$. Then for $|y|\geq 2(r_*(x_0)+ r)$ it also holds
\begin{align}
\biggl(\fint_{|x-y| < r_*(y)}|\nabla(v - Ev_\h)|^2\biggr)^\frac{1}{2}&\lesssim C_1 \biggl(\frac{r}{ |y| }\biggr)^{(d -1) + k + \beta} \| v_\h\|_{k,r}.\label{d.loc}
\end{align}
\end{proposition}
\begin{proposition}\label{invariants.preserved}
For any $m \geq 2$, we set $k= m-1$. Then, the linear maps defined in Proposition \ref{Rom4} preserve the bilinear form:
\begin{align}\nonumber
( u_\h  , v_\h )_\h = (u,v).
\end{align}
\end{proposition}

\smallskip

\begin{lemma}\label{Rom5}
For all $m\ge 2$, $u\in X_m^\h$, and radii $R \gg r_*$ we have
\begin{align}
\Big(\fint_{|x| < R}|\nabla Eu|^2\Big)^\frac{1}{2}&\sim
\Big(\fint_{|x|<R}|\nabla u|^2\Big)^\frac{1}{2}.\label{m08}
\end{align}


For $k\ge 1$, we recall that $Z_k^\h$ denotes the linear space spanned by $\{\partial^\alpha G\}_{|\alpha|=k}$,
where $G$ denotes the fundamental solution of $-\nabla\cdot \ah^*\nabla$; recall $Z_k^\h\cong Y_k^\h/Y_{k+1}^\h$,
see Lemma \ref{folklore}.
For all $m>k\ge 1$, $v\in Z_k^\h\oplus\cdots\oplus Z_{m-1}^\h$, and radii $R\ge r_*$
we have
\begin{align}
\Big(\fint_{|x| > R}|\nabla Ev|^2\Big)^\frac{1}{2}&\sim
\Big(\fint_{|x| > R}|\nabla v|^2\Big)^\frac{1}{2}.\label{m11}
\end{align} 
In addition, for any $r \ge r_*$ and any $v \in Y_k^\h(r)$ the above holds with $\lesssim$ sign for any $R \gg r$.
\end{lemma}
We remark that the previous lemma is in the spirit of Section \ref{abstract.r} since, as it will be pointed out in the proof, it actually only requires the ingredients of 
Proposition~\ref{Ck.1}, and in particular that the first-order correctors are sublinear and that $\psi$ is only subquadratic.


\begin{proof}[Proof of Theorem~\ref{Rom3}]

We first observe that we may choose as $L_X^\h$ and $L_Y^\h$ the linear maps constructed in Proposition \ref{Rom4}. We turn to $L_X$ and $L_Y$ and start with the first one: 
We consider the spaces $W_m \oplus W_{m-1}$ as in Corollary~\ref{Cor5.b} (with $k=m$) and recall that inequality \eqref{m07} means that the projection $X_m\ni u\mapsto u'\in W_m\oplus W_{m-1}$ induced
by the direct sum $X_m=W_m$ $\oplus W_{m-1}$ $\oplus X_{m-2}$ has the continuity property
\begin{align}\label{m07bis}
\|u'\|_m+\|u-u'\|_{m-2}\lesssim \|u\|_m.
\end{align}

\medskip

We upgrade the map of Proposition \ref{Rom4}, which we identified with $L_X^\h$, in three ways:
\begin{itemize}
\item[a)] We have in addition $\|u\|_m\lesssim \|u_\h\|_m$.
\item[b)] We may compose $L_X^\h$ with the projection from $X_m$ onto $W_m$ $\oplus W_{m-1}$
without changing the error estimate and a).
\item[c)] If we specify $u_\h$ to be in $W_m^\h\oplus W_{m-1}^\h$, where $W_m^\h$ denotes the space of
$\ah$-harmonic polynomial homogeneous of degree $m$, see Lemma \ref{folklore} and Lemma \ref{Rom1}, 
then we have in addition $\|u_\h\|_m$ $\lesssim\|u\|_m$.
\end{itemize}
Upgrade a) follows from
the triangle inequality $\|u\|_m$ $\le\|u-Eu_\h\|_m+\|Eu_\h\|_m$, the ordering of norms
$\|u-Eu_\h\|_{m}$ $\le\|u-Eu_\h\|_{m-\beta}$ combined with the error estimate of Proposition \ref{Rom4},
and (\ref{m08}) in Lemma \ref{Rom5} in form of $\|Eu_\h\|_m\lesssim\|u_\h\|_m$.
For upgrade b) we argue that we may replace $u$ by its projection $u'$ onto $W_m\oplus W_{m-1}$,
both in the error estimate and in a). For the latter, this follows from the first part in (\ref{m07bis}).
For the former, we use the triangle inequality $\|u'-Eu_\h\|_{m-\beta}$ $\le\|u-Eu_\h\|_{m-\beta}$ $+\|u-u'\|_{m-\beta}$,
and note that by the ordering of norms, the second part of (\ref{m07bis}), and a) we have as desired 
$\|u-u'\|_{m-\beta}$ $\le\|u-u'\|_{m-2}$ $\lesssim\|u\|_m$ $\lesssim\|u_\h\|_{m}$.

\medskip

We now turn to upgrade c). Since $u_\h\in W_m^\h\oplus W_{m-1}^\h$, so that $\nabla u_\h$ is the sum of an $(m-1)$-homogeneous
and an $(m-2)$-homogeneous polynomial, we obtain from definition (\ref{m01}) for any $r\ge r_*$
(by distinguishing the $R$-ranges of $R\ge r$ and $r\ge R\ge r_*$)
\begin{align}\label{m10}
\|u_\h\|_m\lesssim \biggl(\frac{r_*}{r}\biggl)^{m-2}\bigg(\fint_{|x|< r}|\nabla u_\h|^2\bigg)^\frac{1}{2}.
\end{align}

\BR5 Provided $r \gg r_*$, we \ER then have by (\ref{m08}) in Lemma \ref{Rom5}, the triangle inequality, and once more the definition
(\ref{m01})
\begin{align}\nonumber
\bigg(\fint_{|x|<r}|\nabla u_\h|^2\bigg)^\frac{1}{2}
&\lesssim
\bigg(
\fint_{|x|<r}|\nabla Eu_\h|^2\bigg)^\frac{1}{2}
\\
\nonumber
&\le\bigg(
\fint_{|x|<r}|\nabla(u-Eu_\h)|^2\bigg)^\frac{1}{2} 
+\bigg(
\fint_{|x|<r}|\nabla u|^2\bigg)^\frac{1}{2}
\\
\nonumber
&\le\biggl(\frac{r}{r_*}\biggl)^{m-\beta-1}\|u-Eu_\h\|_{m-\beta}+\biggl(\frac{r}{r_*}\biggl)^{m-1}\|u\|_m.
\end{align}
\BR5
Inserting combination of the estimate $\|u-Eu_\h\|_{m-\beta}$ $\lesssim\|u_\h\|_m$ and estimate~\eqref{m10} into the right-hand side 
 yields
\begin{equation*}
\bigg(\fint_{|x|<r}|\nabla u_\h|^2\bigg)^\frac{1}{2} \lesssim \biggl(\frac{r_*}{r}\biggl)^{\beta-1}
\bigg(\fint_{|x|<r}|\nabla u_\h|^2\bigg)^\frac{1}{2}+\biggl(\frac{r}{r_*}\biggl)^{m-1}\|u\|_m.
\end{equation*}
Recalling  that $\beta > 1$, by choosing $r \gg r_*$ the first term on the right-hand side can be absorbed to the left-hand side, which then leads to the desired $\|u_\h\|_m\lesssim\|u\|_m$. 
\ER

\medskip

We now claim that the linear map $L_X^\h$ provided by Proposition \ref{Rom4}, when restricted to
$u_\h\in W_m^\h\oplus W_{m-1}^\h$ and projected onto $u\in W_m\oplus W_{m-1}$, is an isomorphism between these
normed spaces. Indeed, the above upgrade
of Proposition \ref{Rom4} shows that $\|u\|_m\sim\|u_\h\|_m$ so that the map is an isomorphism of
$W_m^\h\oplus W_{m-1}^\h$ and its image in $W_m\oplus W_{m-1}$. In order to show that this image
coincides with $W_m\oplus W_{m-1}$, we show that $W_m^\h\oplus W_{m-1}^\h$ and $W_m\oplus W_{m-1}$
have the same dimension:
Since $X_m=W_{m}\oplus W_{m-1}\oplus X_m$ we have 
${\rm dim}(W_m\oplus W_{m-1})$ 
$={\rm dim}X_m$ $-{\rm dim}X_{m-2}$ and the analogue statement for $W_m^\h\oplus W_{m-1}^\h$.
Since by (\ref{Fischer1}) the dimensions of $\{X_m\}_m$ and $\{X_m^\h\}_m$ agree, we obtain
${\rm dim}(W_m\oplus W_{m-1})$ $={\rm dim}(W_m^\h\oplus W_{m-1}^\h)$.

\medskip

We now are in a position to construct also the linear map $L_X: X_m\ni u\mapsto u_\h\in X_m^\h$ with $\|u-Eu_\h\|_{m-\beta}\lesssim\|u\|_m$.
To $u\in X_m$ we associate the linear projection $u'\in W_m\oplus W_{m-1}$. Since by the above argument
the linear map $L_X^\h$ provided by Proposition \ref{Rom4} restricted to $W_m^\h\oplus W_{m-1}^\h$ and
projected onto $W_m\oplus W_{m-1}$ is onto, there exists
$u_\h\in W_m^\h\oplus W_{m-1}^\h$ with $\|u'-Eu_\h\|_{m-\beta}$ $\lesssim\|u_\h\|_m$ and \mbox{$\|u_\h\|_m$ $\lesssim\|u'\|_m$}.
It remains to upgrade the resulting $\|u'-Eu_\h\|_{m-\beta}$ $\lesssim\|u'\|_m$ to $\|u-Eu_\h\|_{m-\beta}$ $\lesssim\|u\|_m$,
which again follows from (\ref{m07bis}).
\BR5 For later purpose we retain
\begin{equation}\label{m99}
 \|u_\h\|_{m} \lesssim \|u\|_m,
\end{equation}
which follows from already established $\|u_\h\|_m \lesssim\|u'\|_m$ and~\eqref{m07bis}.
\ER

\medskip

We note that the projection $X_m\ni u\mapsto u'\in W_m\oplus W_{m-1}$ lifts to an isomorphism between
the {\it normed} spaces $X_m/X_{m-2}$ and $W_m\oplus W_{m-1}$; the boundedness of this bijection
follows from the first part of (\ref{m07bis}), the boundedness of its inverse is tautological by
definition of the norm on the quotient space. Since this property also holds for
$X_m^\h/X_{m-2}^\h$ and $W_m^\h\oplus W_{m-1}^\h$, the isomorphism of $W_m^\h\oplus W_{m-1}^\h$ and $W_m\oplus W_{m-1}$
constructed above lifts to the desired isomorphism of the normed linear spaces $X_m^\h/X_{m-2}^\h$ and $X_m/X_{m-2}$.

\medskip

Finally, we argue that the relation $u_\h\leftrightarrow u$ provided by the
isomorphism of $X_m^\h/X_{m-2}^\h$ and $X_m/X_{m-2}$ constructed above coincides with the one defined through
the qualitative relation (\ref{m03}). By construction, the relation $u_\h\leftrightarrow u$ defined through
the isomorphism has the following characterizing property: There exists $u_\h'\in X_m^\h$ with $u_\h'-u_\h\in X_{m-2}^\h$ 
and $u'\in X_m$ with $u'-u\in X_{m-2}$ such that $u_\h'\mapsto u'$ under the linear map \BR5 $L_X^\h$ \ER provided
by Proposition \ref{Rom4}; in particular $\|u'-Eu_\h'\|_{m-\beta}$ $\lesssim\|u_\h'\|_m$.
By the triangle inequality, the ordering of the norms, and (\ref{m08}) the latter implies
\begin{align*}
\|u-Eu_\h\|_{m-\beta}\lesssim\|u'-Eu_\h'\|_{m-\beta}
+\|u'-u\|_{m-2}+\|u_\h'-u_\h\|_{m-2}<\infty,
\end{align*}
which thanks to $\beta>1$ yields the qualitative relation (\ref{m03}). The opposite direction is slightly more
subtle: Let $\tilde u\in X_m$ be related to $u_\h\in X_m^\h$ by (\ref{m03}); pick a $u\in X_m$ related to
$u_\h$ by the isomorphism. By the above also the pair $(u_\h,u)$ satisfies the relation (\ref{m03}). By the
triangle inequality we have for $\tilde u-u\in X_m$
\begin{align*}
\lim_{R\uparrow\infty}\frac{1}{R^{(m-1)-1}}\bigg(\fint_{|x|<R}|\nabla(\tilde u-u)|^2\bigg)^\frac{1}{2}=0.
\end{align*}
By \eqref{Fischer2} of Lemma \ref{Lup} this implies $\tilde u-u\in X_{m-2}$ so that next to $u$ also $\tilde u$ is related to $u_\h$ 
via the isomorphism. Likewise, let $\ut \in X_m^\h$ be related to $u\in X_m$ by (\ref{m03}); pick
a $u_\h\in X_m^\h$ related to $u$ by the isomorphism to the effect that also $(u_\h,u)$ satisfies the relation
(\ref{m03}). By the triangle inequality and by (\ref{m08}) we have for $\ut-u_\h\in X_m^\h$
\begin{align*}
\lim_{R\uparrow\infty}\frac{1}{R^{(m-1)-1}}\bigg(\fint_{|x| < R}|\nabla(\ut-u_\h)|^2\bigg)^\frac{1}{2}=0,
\end{align*}
which implies that $\ut-u_\h\in X_m^\h$, which is a polynomial (see Lemma \ref{folklore}), is of degree $\le m-2$
and thus an element of $X_{m-2}^\h$. Hence also $\ut$ and $u$ are related by the isomorphism.

\medskip

We now prove the statement for the spaces $Y_{m-1}^\h(r)$ and $Y_{m-1}(r)$ in a similar way: We pick the space $Z_{m-1}(r_*) \oplus Z_m(r_*)$ as in Corollary~\ref{Cor5.b} (here, $k=m-1$). Recalling that $Y_{m-1}(r) = Z_{m-1}(r_*)\oplus Z_m(r_*) \oplus Y_{m+1}(r)$, \BR5 \eqref{eqC301} in \ER 
Corollary~\ref{Cor5.b} for $v$ means that the projection $Y_{m-1}(r) \ni v \mapsto v' \in Z_{m-1}(r_*) \oplus Z_{m}(r_*)$ is continuous in the sense that
\begin{align}\label{m07bbis}
 \| v'\|_{m-1,r} + \| v-v'\|_{m+1,r} \lesssim \| v\|_{m-1,r}. 
\end{align}
As in the case of the spaces $X_m$, we need to upgrade Proposition \ref{Rom4} for $v$ and $v_\h$, according to the points a), b) and c). 
We begin with a): \BR5 To \ER $v_\h \in Y_{m-1}^\h(r)$ we associate $v \in Y_{m-1}(r)$ according to Proposition \ref{Rom4}. This, together with the triangle inequality, the ordering $\| \cdot  \|_{m -1,r} \leq \| \cdot \|_{m-1 + \beta,r}$ since $\beta \geq 0$, and Lemma \ref{Rom5} yields
$$
\| v\|_{m-1,r} \leq \| v - Ev_\h\|_{m-1+\beta,r} + \|Ev_\h\|_{m-1,r}\lesssim \|v_\h\|_{m-1, r}.
$$

\smallskip

We now turn to b): We upgrade the map \BR5 $L_Y^{\h}$ \ER of Proposition \ref{Rom4} to $Y_{m-1}^\h(r) \ni v_\h \mapsto v \in Z_{m-1}(r_*) \oplus Z_{m}(r_*)$ such that $\| v - Ev_\h \|_{m-1 + \beta, \BR4 r}\lesssim \| v_\h\|_{m-1,r}$ holds. Note that this implies that $a)$ 
holds also for this choice of $v$, by the same argument above. Let us consider $v_\h \in Y_{m-1}(r)$ and let $v \in Y_{m-1}(r)$ be its image under the map of Proposition \ref{Rom4}. Let $v'$ be the projection of $v$ onto $Z_{m-1}(r_*) \oplus Z_{m}(r_*)$.
We now apply in sequence the triangle inequality, Proposition \ref{Rom4}, the ordering of the norms (recall $\beta < 2$), inequality \eqref{m07bbis} and a) for $v_\h$ and $v$ and get
\begin{equation}\label{eq101}
 \|v' - E v_\h \|_{m-1 +\beta,r} \lesssim \|v - E v_\h \|_{m-1+\beta, r}+ \| v -v'\|_{m+1,r} {\lesssim} \| v_\h\|_{m-1,r},
\end{equation}
i.e. the desired estimate.

\smallskip

It remains to prove c): 
\BR5
We restrict to $v_\h \in Z_{m-1}^\h \oplus Z_{m}^\h$, the space of functions spanned by $\partial^{\alpha} G$, where $\alpha$ runs over all multi-indices $\alpha$ with $|\alpha| \in \{m-1,m\}$, and where $G$ denotes the fundamental solution of $-\nabla \cdot \ah^* \nabla$, see~Lemma~\ref{folklore}. 
Since $\nabla v_\h$ is the sum of a $-(d-1)-(m-1)$ and a $-(d-1)-m$ homogeneous function, see~\eqref{R5}, we have 
\ER
\begin{align}\label{m10b}
\| v_\h \|_{m-1,r} \lesssim \biggl( \frac{R}{r}\biggr)^{(d -1)+m} \biggl( \fint_{|x|> R} |\nabla v_\h|^2 \biggr)^\oh.  
\end{align}
Therefore for $R \BR5 \gg r$, by Lemma \ref{Rom5} and the choice of $v\in Z_{m-1}(r_*)\oplus Z_{m}(r_*)$ according to b) we have that
\BR5
\begin{align*}
 \biggl(\fint_{|x|> R}|\nabla v_\h|^2 \biggr)^\oh &\stackrel{\eqref{m11}}{\lesssim} 
 \biggl(\fint_{|x|> R}|\nabla Ev_\h|^2 \biggr)^\oh\\
 &\lesssim 
 \biggl(\fint_{|x|> R}|\nabla( v - Ev_\h)|^2 \biggr)^\oh + 
 \biggl(\fint_{|x|> R}|\nabla v|^2 \biggr)^\oh\\
 &\stackrel{\eqref{eq101}}{\lesssim} \biggl(\frac r R \biggr)^{(d-1)+ m-1 + \beta} \|v_\h\|_{m-1,r} + \biggl(\frac r R \biggr)^{(d-1)+ m-1} \| v\|_{m-1,r}\\
 &\stackrel{\eqref{m10b}}{\lesssim} \biggl(\frac{r}{R} \biggr)^{\beta-1 }
 \biggl(\fint_{|x|> R}|\nabla v_\h|^2 \biggr)^\oh + 
 \biggl(\frac r R \biggr)^{(d-1)+ m-1} \| v\|_{m-1,r}.
\end{align*}
Since $\beta > 1$, for $R \gg r$ \ER we may absorb the first term on the right-hand side and get 
$$
 \biggl(\frac R r \biggr)^{(d-1)+ m-1}\biggl(\fint_{|x|> R}|\nabla v_\h|^2 \biggr)^\oh \lesssim \| v\|_{m-1,r}.
$$
We apply \eqref{m10b} again and conclude $\| v_\h\|_{m-1,r}\lesssim \| v\|_{m-1,r}$.

\medskip

The remaining part of the proof follows similarly to the case of the spaces $X_m$ and $X_m^\h$.  We consider the upgrade of $L_Y^\h$ constructed above and restrict it to $Z_{m-1}^\h\oplus Z_{m}^\h$. Since by Proposition~\ref{Pdual} and Lemma~\ref{folklore} we have that 
\begin{equation}\label{isomorph.}
\begin{aligned}
Z_{m-1}(r_*)\oplus Z_{m}(r_*)&\cong Y_{m-1}(r_*) / Y_{m+1}(r_*),\\
Z_{m-1}^\h\oplus Z_{m}^\h&\cong Y_{m-1}^\h / Y_{m+1}^\h,
\end{aligned}
\end{equation}
by 
\eqref{approx.decay} it follows that
$$
 \text{dim}(Z_{m-1}^\h\oplus Z_{m}^\h) =\text{ dim}(Z_{m-1}(r_*)\oplus Z_{m}(r_*) ).
$$
By a), b) and c) above this yields that the restriction of $L_Y^\h$, $Z_{m-1}^\h\oplus Z_{m}^\h \rightarrow Z_{m-1}(r_*)\oplus Z_{m}(r_*)$ provides an isomorphism between the normed spaces $Z_{m-1}^\h\oplus Z_{m}^\h$ and $Z_{m-1}(r_*)\oplus Z_{m}(r_*)$.
Since also the isomorphisms in \eqref{isomorph.} are between normed spaces thanks to \eqref{m07bbis}, we may lift the previous map also to an isomorphism between $Y_{m-1}^\h/Y_{m+1}^\h$ and $Y_{m-1}/Y_{m+1}$. 

\medskip

As map $L_Y$ we consider the composition of the projection \mbox{$Y_{m-1}(r) \rightarrow Z_{m-1}(r_*)\oplus Z_{m}(r_*)$} with the inverse of the map \mbox{$Z_{m-1}^\h\oplus Z_{m}^\h \rightarrow Z_{m-1}(r_*)\oplus Z_{m}(r_*)$}: If for $v\in Y_{m-1}(r)$, 
the function $v'$ denotes its projection onto $Z_{m-1}(r_*)\oplus Z_{m}(r_*)$, we consider the element $v_\h \in Z_{m-1}^\h(r_*)\oplus Z_{m}^\h(r_*)$ such that $\| v' - Ev_\h \|_{m-1+\beta, r} \lesssim \| v_\h\|_{m -1, r}$ and, by c),
$\| v_\h \|_{m-1,r} \lesssim \| v'\|_{m-1,r}$. Therefore by these two inequalities, the ordering of the norms, the triangle inequality and \eqref{m07bbis} we also conclude
$$
\| v - Ev_\h\|_{m-1+\beta, r} \lesssim \| v - v' \|_{m+1,r} + \| v' - Ev_\h\|_{m-1+\beta,r} \lesssim \|v \|_{m-1,r},
$$
i.e. the desired estimate for $L_Y$.

\medskip

It remains to show that the isomorphism \BR5 between 
$Z_{m-1}^\h\oplus Z_{m}^\h$ and $Z_{m-1}(r_*)\oplus Z_{m}(r_*)$, induced by $L_Y^{\h}$, \ER coincides with the relation \eqref{m04}. In other words, we claim that $v \leftrightarrow v_\h$ via \eqref{m04} 
is equivalent to having their projections $v', v_\h'$ onto the spaces $Z_{m-1}^\h\oplus Z_{m}^\h$ and $Z_{m-1}(r_*)\oplus Z_{m}(r_*)$ satisfy $\| v' - Ev_\h'\|_{m-1+\beta}\lesssim \|v_\h'\|_{m-1}$.
Let us assume that $v \leftrightarrow v_\h$ via \eqref{m04} and that $v\in Y_{m-1}(r)$ and $v_\h\in Y_{m-1}^\h(r)$ for some $r_* \leq r < +\infty$. For every $R \geq r$ we have by the triangle inequality
\begin{align*}
\biggl(\fint_{|x| > R}|\nabla( v' - Ev_\h')|^2\biggr)^\oh 
&\le 
\biggl(\fint_{|x| > R}|\nabla( v - Ev_\h)|^2\biggr)^\oh
+ 
\biggl(\fint_{|x| > R}|\nabla( v' - v)|^2\biggr)^\oh 
\\
&\quad + 
\biggl(\fint_{|x| > R}|\nabla E_h( v_\h - v_\h')|^2\biggr)^\oh,
\end{align*}
\BR5 so that by~\eqref{m11} in Lemma~\ref{Rom5}
\ER 
\begin{align*}
R^{(d-1)+ m}&\biggl(\fint_{|x| > R}|\nabla( v' - Ev_\h')|^2\biggr)^\oh \lesssim R^{(d-1)+ m}\biggl(\fint_{|x| > R}|\nabla( v - Ev_\h)|^2\biggr)^\oh\\
& + r^{(d-1) + m +1}R^{-1}\bigl(\| v' - v\|_{m+1,r} + \|v_\h - v_\h'\|_{m+1,r} \bigr).
\end{align*}
\BR5 
Noting that by~\eqref{m07bbis} $\|v'-v\|_{m+1,r} < \infty$ (and analogously also $\|v_\h - v_\h'\|_{m+1,r}<\infty$) 
and 
\ER 
taking the limit $R\rightarrow +\infty$ yields that also $v' \leftrightarrow v_\h'$ via \eqref{m04}.
It remains to show that $\| v' - Ev_\h'\|_{m-1+\beta}\lesssim \|v_\h'\|_{m-1}$. By the isomorphism $Z_{m-1}^\h\oplus Z_{m}^\h \rightarrow Z_{m-1}(r_*)\oplus Z_{m}(r_*)$, we may find $\tilde v \in Z_{m-1}(r_*)\oplus Z_{m}(r_*)$ such that 
$\| \tilde v - Ev_\h'\|_{m-1+\beta}\lesssim \|v_\h'\|_{m-1}$. We claim that $\tilde v = v'$: Indeed, since $v \leftrightarrow v_\h$ and $v'\leftrightarrow v_\h'$ via \eqref{m04}, it follows by the triangle inequality that
\begin{align*}
\lim_{R\rightarrow +\infty}R^{(d-1)+ m}\biggl(\fint_{|x| > R}|\nabla( v' - \tilde v)|^2\biggr)^\oh = 0,
\end{align*}
which implies by \eqref{Liouv.decay} in Corollary \ref{Lbil} that $v' - \tilde v \in Y_{m+1}(r_*)$,
\BR5
which by Proposition~\ref{Pdual} implies $v'=\tilde v$ since both sides are in $Z_{m-1}(r_*) \oplus Z_{m}(r_*)$.
\ER
Let us now assume that for $v\in Y_{m-1}(r)$ and $v_\h\in Y_{m-1}^\h(r)$ for some $r_* \leq r < +\infty$, it holds $\| v' - Ev_\h'\|_{m-1+\beta,r}\lesssim \|v_\h'\|_{m-1,r}$. By the triangle inequality we have
\begin{align*}
\biggl(\fint_{|x| > R}|\nabla( v - Ev_\h)|^2\biggr)^\oh &\le 
\biggl(\fint_{|x| > R}|\nabla( v' - Ev_\h')|^2\biggr)^\oh +
\biggl(\fint_{|x| > R}|\nabla( v' - v)|^2\biggr)^\oh 
\\
&\quad + 
\biggl(\fint_{|x| > R}|\nabla(E_h( v_\h - v_\h'))|^2\biggr)^\oh,
\end{align*}
\BR5
so that by our assumption and~\eqref{m11} in Lemma~\eqref{Rom5}
\ER
\begin{align*}
R^{(d-1)+ m}&\biggl(\fint_{|x| > R}|\nabla( v - Ev_\h)|^2\biggr)^\oh \lesssim r^{(d-1)+m+\beta} R^{1-\beta}\|v_\h'\|_{m-1,r}\\
& + r^{(d-1) + m +1}R^{-1}\bigl(\| v' - v\|_{m+1,r} + \|v_\h - v_\h'\|_{m+1,r} \bigr).
\end{align*}
Since $\beta > 1$, we may send $R\rightarrow +\infty$ and obtain that $v \leftrightarrow v_\h$ via \eqref{m04}.

\medskip

We finally conclude the proof of the theorem by showing that the bilinear forms are preserved by the isomorphisms: We take $u_\h, u$ and $v, v_\h$ related by the isomorphisms. 
This implies by construction that $u_\h \leftrightarrow u$ and $v_\h \leftrightarrow v$ according to \eqref{m03},\eqref{m04}. For the same functions $u_\h$ and $v_\h$, let $u'$ and $v'$ be their images under the maps $L_X^\h$ and $L_Y^\h$. 
Moreover, as shown above in the construction of the isomorphisms, we have that $u-u' \in X_{m-2}$ and $v-v' \in Y_{m+1}(r)$. Since the maps $L_X^\h$ and $L_Y^\h$ coincide with the ones of Proposition \ref{Rom3}, it follows that $(u_\h, v_\h)_\h = (u', v')$ and, 
by Corollary \ref{Lbil}, also that $(u_\h, v_\h)_\h = (u, v)$. The proof of Theorem \ref{Rom3} is complete.
\end{proof}


\begin{proof}[Proof of Theorem~\ref{T3}]

\BR5 Estimates~\eqref{growth.loc} and~\eqref{dec.loc} 
for $L_X^\h$ and $L_{Y}^\h$, respectively follow immediately from the corresponding estimates in Theorem~\ref{Rom3} and \ER 
Proposition \ref{localization}. It remains to argue in favor of \eqref{growth.loc2} and \eqref{dec.loc2}. \BR5 Appealing to~\eqref{m99}, estimate~\eqref{growth.loc2} follows from~\eqref{growth.loc}. Arguing analogously we also obtain \eqref{dec.loc2}.\ER 
\end{proof}


\subsection{Proof of Lemma~\ref{Rom5}}\hspace*{\fill} 

\begin{proof}[Proof]

We start by remarking that since $X_m^\h$ is a space of $\ah$-harmonic polynomials of degree at most $m$ (see Lemma \ref{folklore}), we have that
\begin{align}\label{equi.u}
\sup_{|x|< R} ( |\nabla u|^2 + R^2 |\nabla^2 u| + R^4|\nabla^3 u|^2) \lesssim \fint_{|x|<R} |\nabla u|^2 .
\end{align}
Moreover, by the boundedness of the linear map $ X_m^\h / X_{m-1}^\h \ni u \mapsto \ut$ and the homogeneity of the polynomial $\ut$, it follows that we also have
\begin{align}\label{equi.u.tilde}
 \sup_{|x|< R} ( |\ut|^2 + R^2 |\nabla \ut|^2+ R^4 |\nabla^2 \ut|^2)\lesssim  \fint_{|x|<R} |\nabla u|^2 .
\end{align}

\medskip
\BR5
To show a similar statement for $v \in Z_k^\h\oplus\cdots\oplus Z_{m-1}^\h$ in form of
\begin{equation}\label{equi.v}
\sup_{R < |x|< 2R} ( |\nabla v|^2 + R^2 |\nabla^2 v|^2 + R^4|\nabla^3 v|^2)\lesssim  \fint_{R <|x| < 2R} |\nabla v|^2,
\end{equation}
after decomposing $v = v_k + \cdots + v_{m-1}$ and using triangle inequality together with the boundedness of the projection $v \mapsto v_k$ (see~\eqref{ao04}), we see that it is enough to show ~\eqref{equi.v} for $v \in Z_n^\h$ for any $n \ge 1$. By homogeneity of elements in $Z_n^\h$ we can assume $R=1$, and~\eqref{equi.v} then follows from finite-dimensionality of $Z_n^\h$. 
\ER


\medskip
The boundedness of the linear map $Z_k^\h\ni v\mapsto \vt$ constructed in Lemma \ref{Rom1} on the finite-dimensional space $Z_k^\h$, the homogeneity of $\vt$, and the equivalence of norms finally also yield
\begin{align}\label{equi.v.tilde}
\sup_{R < |x| < 2R}(|\vt|^2 + R^2|\nabla\vt|^2 + R^4 |\nabla^2 \vt|^2)\lesssim \fint_{R< |x| < 2R}|\nabla v|^2.
\end{align}

\medskip 
In the case $v \in Y_k^\h(r_*)$ we cannot rely anymore on the equivalence of norms to infer \eqref{equi.v}. However, since $v$ is $\ah^*$-harmonic in $\{ |x| > r_*\}$, by standard elliptic regularity theory, we may infer
\begin{align}\label{equi.v.full}
\sup_{R < |x|< 2R} ( |\nabla v|^2 + R^2 |\nabla^2 v|^2 + R^4|\nabla^3 v|^2)\lesssim  \fint_{\frac R 2 <|x| < 4R} |\nabla v|^2,
\end{align}
for $R \geq 2r_*$. For the sake of completeness, we give the proof of this result in the appendix. We remark that estimate \eqref{equi.v.tilde} holds nonetheless in this case by the same means as above and by the fact that the
projection $Y_k^\h(r_*)\ni v \mapsto v' \in Z_k^\h(r_*)$ is continuous by Proposition~\ref{Pdual}.

\medskip

As already mentioned after its statement, the estimates of this lemma hold under assumptions on the correctors which are weaker than \eqref{T.1} and \eqref{T.2}, namely 
\begin{align}\label{weak.sub}
R^{-1}\biggl(\fint_{|x|< R} |(\phi_i,\sigma_i)|^2 \biggr)^\oh + R^{-2}\biggl(\fint_{|x| < R} (\psi_{ij})^2 \biggr)^\oh \ll 1, \qquad R \gg r_*,
\end{align}
and the analogous statement for $\phi^*$, $\sigma^*$, and $\psi^*$. \BR5 Because of our simplifying assumption~\eqref{psi.avgzero} only $\psi_{ij}$ and not $\psi_{ij} - \fint \psi_{ij}$ appeared in~\eqref{weak.sub}. \ER 
We start therefore by proving that, under the sole assumption \eqref{weak.sub}, we have estimate \eqref{m11} for $v \in Z_k^\h(r_*) \oplus \cdots \oplus Z_m^\h(r_*)$  as in the statement, but \eqref{m11} for $v \in Y_k^\h(r)$ provided $R \gg r$. 
Similarly, by assuming only \eqref{weak.sub}, we obtain \eqref{m08} exclusively for $R\gg r_*$. If we further assume \eqref{T.2} and \eqref{T.1}, then the condition above holds for $R \geq r_*$ and therefore we recover \eqref{m11} and \eqref{m14} for $R$ 
in the same range required by the statement of the lemma.

\smallskip



We first tackle the second norm equivalence (\ref{m11}), which we may divide into dyadic annuli
\begin{align}\label{ir08}
\int_{R<|x|<2R}|\nabla E v|^2\sim\int_{R<|x|<2R}|\nabla v|^2.
\end{align}
We start with the case of $v\in Z_k^\h\oplus\cdots\oplus Z_{m-1}^\h$, and the $\gtrsim$-part.
We note that for these $v$, we have the (uniform in $R$) norm equivalence
\begin{align}\label{ir10}
\bigg(\fint_{R<|x|<2R}|\nabla v|^2\bigg)^\frac{1}{2}\sim R^{-1}\inf_{c\in\mathbb{R}}\bigg(\fint_{R<|x|<2R}(v-c)^2\bigg)^\frac{1}{2}.
\end{align}
The $\gtrsim$-part of (\ref{ir10}) coincides with Poincar\'e's inequality. The $\lesssim$-part follows from 
Caccioppoli's estimate (recall that $v$ is $\ah$-harmonic outside the origin) on annuli in form of
\begin{align*}
\bigg(\fint_{\frac{4}{3}R<|x|<\frac{5}{3}R}|\nabla v|^2\bigg)^\frac{1}{2}
\lesssim R^{-1}\inf_{c\in\mathbb{R}}\bigg(\fint_{R<|x|<2R}(v-c)^2\bigg)^\frac{1}{2}
\end{align*}
in conjunction with the inverse estimate
\begin{align}\label{ir09}
\bigg(\fint_{R<|x|<2R}|\nabla v|^2\bigg)^\frac{1}{2}\lesssim\bigg(\fint_{\frac{4}{3}R<|x|<\frac{5}{3}R}|\nabla v|^2\bigg)^\frac{1}{2}.
\end{align}
The latter can be seen as follows: Since the space $Z_k^\h\oplus\cdots\oplus Z_{m-1}^\h$ is invariant under scaling, it suffices
to establish (\ref{ir09}) for $R=1$, in which case it follows from the fact that the space is of finite dimension and
that the right-hand side constitutes a semi-norm on this space that vanishes only on constants. This establishes (\ref{ir10}).
Thanks to the inverse estimate \eqref{ir09}, more precisely extension of~\eqref{ir09} where the domain of integration on the left-hand side can be taken much larger at the expense of larger prefactor on the right-hand side, it suffices to establish the $\gtrsim$-part of (\ref{ir08}) only for $R\gg r_*$. 

\medskip

In view of the norm equivalence (\ref{ir10}), and once more by Poincar\'e's inequality, for the $\gtrsim$-part of (\ref{ir08}) 
it is enough to establish 

%
\begin{align*}
\inf_{c\in\mathbb{R}}\int_{R < |x| < 2R}(v-c)^2\lesssim\inf_{c\in\mathbb{R}}\int_{R<|x|<2R}(Ev-c)^2\quad\mbox{for}\;R\gg r_*.
\end{align*}
Hence by the triangle inequality it is enough to establish
\begin{align*}
\int_{R<|x|<2R}(Ev-v)^2\ll\inf_{c\in\mathbb{R}}\int_{R<|x|<2R}(v-c)^2\quad\mbox{for}\;R\gg r_*,
\end{align*}
which by~\eqref{ir10} 
follows from
\begin{align*}
R^{-2}\int_{R<|x|<2R}(Ev-v)^2\ll\int_{R<|x|<2R}|\nabla v|^2\quad\mbox{for}\;R\gg r_*.
\end{align*}
By definition (\ref{m15}) of $Ev$ we may split this in two statements
\begin{align}
R^{-2}\bigg(\int_{R<|x|<2R}(\phi_i^*\partial_iv)^2+(\psi_{ij}^*\partial_{ij}v)^2
\bigg)&\ll\int_{R<|x|<2R}|\nabla v|^2\quad\mbox{for}\;R\gg r_*,\label{m13}\\
R^{-2}\bigg(\int_{R<|x|<2R}\vt^2+(\phi_i^*\partial_i\vt)^2\bigg)&
\ll\int_{R<|x|<2R}|\nabla v|^2\quad\mbox{for}\;R\gg r_*.\label{m14}
\end{align}
By \eqref{equi.v}, estimate~\eqref{m13} reduces to \eqref{weak.sub}. 
Estimate~\eqref{m14} can be proven similarly, using \eqref{equi.v.tilde} and \eqref{weak.sub}. 
%

\medskip

We now turn to the $\lesssim$-part of~\eqref{m11}, which follows from the $\lesssim$-part of~\eqref{ir08} 
\begin{align}\label{Rom5.c}
\biggl( \fint_{R < |x| < 2R} |\nabla Ev|^2 \biggr)^\oh \lesssim \biggl( \fint_{R <|x| < 2R} |\nabla v|^2 \biggr)^\oh.
\end{align}
By definition~\eqref{m15} of $Ev$, we infer that
\begin{align*}
 \nabla Ev& = \partial_i (v + \vt) ( e_i + \nabla \phi_i^*) + \phi_i^* \nabla \partial_i( v  + \vt) +  \partial_{ij}v\nabla \psi_{ij}^*+ \psi_{ij}^*\nabla\partial_{ij}v,
\end{align*}
\BR5 using which we have \ER
\begin{align}\label{Rom5.d}
\biggl( \fint_{R < |x|< 2R} |\nabla E v|^2 \biggr)^\oh \leq &\sup_{R < |x| < 2R}\bigl(| \nabla v| + |\nabla\vt|\bigr) \biggl( \fint_{|x|<2R} | e_i + \nabla \phi_i^* |^2 \biggl)^\oh\notag\\
& + \sup_{R < |x| < 2R}\bigl(| \nabla^2 v| + |\nabla^2\vt|\bigr) \biggl( \biggl( \fint_{|x|<2R} ( \phi_i^*)^2 \biggr)^\oh + \biggl( \fint_{|x| < 2R}|\nabla \psi^*_{ij} |^2 \biggl)^\oh\biggr)\notag\\
& + \sup_{R < |x| < 2R} | \nabla^3 v|\ \biggl( \fint_{|x| < 2R} (\psi_{ij}^*)^2 \biggl)^\oh.
\end{align}
Since by \eqref{i2} the functions $x_i +\phi_i^*$ are $a^*$-harmonic, we apply Caccioppoli's inequality in the form of
\begin{align}\label{grad.phi}
\biggl(\fint_{|x|<R}|e_i + \nabla \phi_i^*|^2\biggr)^\oh \lesssim R^{-1}\biggl(\fint_{|x|< 2R} ( x_i + \phi_i^*)^2 \biggr)^\oh \stackrel{\eqref{weak.sub}}{\lesssim} 1
\end{align}
\BR5 for $R \ge r_*$. \ER 
Similarly, we use equation \eqref{ir13} for~$\psi_{ij}^*$ and test it (as for Caccioppoli's inequality) with $\eta^2 \psi^*_{ij}$ with $\eta$ cut-off for $\{|x| < R \}$ in $\{|x| < 2R \}$. We use ellipticity of $a^*$, together with Cauchy-Schwarz and Young's inequalities to bound 
\begin{align}\label{grad.psi}
\biggl( \fint_{|x| < R} |\nabla \psi_{ij}^*|^2\biggr)^\oh & \lesssim R^{-1}\biggl( \fint_{|x| < 2R} (\psi^*_{ij})^2\biggr)^\oh+ \biggl(\fint_{|x|< 2R}(\phi^*_i)^2 + |\sigma^*_j|^2\biggr)^\oh\notag\\
& \BR5 \stackrel{\eqref{weak.sub}}{\le} R
\end{align}
\BR5 whenever $R \le r_*$. \ER
By plugging the previous two estimates and~\eqref{weak.sub} into \eqref{Rom5.d}, and using  \eqref{equi.v} and \eqref{equi.v.tilde}, we get \eqref{Rom5.c} for $R \gg r_*$. The statement for the case $R \ge r_*$ follows by the previous argument combined with the inverse estimate~\eqref{ir09}. We thus established \eqref{m11} in case $v\in Z_k^\h\oplus\cdots\oplus Z_{m-1}^\h$. Since estimate \eqref{Rom5.c} may be proven in the same way for $v \in Y_k^\h$ using \eqref{equi.v.full} instead of \eqref{equi.v}, the proof of the lemma for the spaces $Y_k^\h $ is complete.

\smallskip

We turn to \eqref{m08}: The $\lesssim$-part may be proven as above, this time appealing to \eqref{equi.u} and~\eqref{equi.u.tilde} instead of \eqref{equi.v} and~\eqref{equi.v.tilde}. It remains to show the inequality with the
$\gtrsim$ sign: By Poincar\'e's inequality, it is enough to show that
\begin{align*}
R^{-2}\inf_{c\in \mathbb{R}} \fint_{|x|< R} (Eu - c)^2 \gtrsim \fint_{|x|< R}|\nabla u|^2. 
\end{align*}
By definition of $Eu$ and the triangle inequality, this follows once we establish
\begin{align*}
R^{-2} \inf_{c\in \mathbb{R}} \fint_{|x|< R} (u - c)^2 &\gtrsim \fint_{|x|< R}|\nabla u|^2 + R^{-2} \fint_{|x|< R}( \ut^2  + (\phi_i \partial_i(u+ \ut))^2 + (\psi_{ij} \partial_{ij}u)^2 ).
\end{align*}
Like for $v$ in (\ref{ir10}), we need the following equivalence of semi-norms for $u\in X_m^\h$
\begin{align*}
\bigg(\fint_{|x|<R}|\nabla u|^2\bigg)^\frac{1}{2}\sim R^{-1}\inf_{c\in\mathbb{R}}\bigg(\fint_{|x|<R}(u-c)^2\bigg)^\frac{1}{2}.
\end{align*}
This equivalence follows from Poincar\'e's inequality, Caccioppoli's estimate for the $\ah$-harmonic $u$, 
and from the fact that on a finite-dimensional space of polynomials we have the inverse estimate
$\big(\fint_{|x|<R}|\nabla u|^2\big)^\frac{1}{2}$ $\lesssim\big(\fint_{|x|<R/2}|\nabla u|^2\big)^\frac{1}{2}$. In view
of this norm equivalence, it suffices to show 
\begin{align*}
 R^{-2} \fint_{|x|< R}( \ut^2 + (\phi_i \partial_i(u+ \ut))^2 + (\psi_{ij} \partial_{ij}u)^2 ) \ll \fint_{|x|< R}|\nabla u|^2,
\end{align*}
when $R \gg r_*$. The argument for this estimate is similar to the one for \eqref{m13} and \eqref{m14}, 
and uses \eqref{equi.u}, \eqref{equi.u.tilde}, and the assumption \eqref{weak.sub} on the correctors. 
\end{proof}


\subsection{Proof of Proposition~\ref{Rom4}}\hspace*{\fill} 

\begin{proof}

W.l.o.g. we may assume that $r_*=1$.
\PfStart{C:Rom4}

\PfStep{C:Rom4}{C:Rom4.1}
Two-scale expansion error estimate; 
 we start by showing that for every $u_\h \in X_m^\h$ and $v_\h \in Y_k^\h(r)$, the two-scale expansions $Eu_\h$ and $Ev_\h$ (see definition \eqref{m15}) satisfy the equations
\begin{equation}\label{two.scale.1}
 \begin{aligned}
 -\nabla \cdot a \nabla (Eu_\h) &= \nabla \cdot g \ \ \ \text{ in $\mathbb{R}^d$,}\\
 -\nabla \cdot a^* \nabla (Ev_\h) &= \nabla \cdot h \ \ \ \text{ in $\{ |x| > r \}$,}
 \end{aligned}
\end{equation}
with right-hand side estimated as follows:
\begin{equation}\label{rhs.estimate}
 \begin{aligned}
 \biggl( \fint_{|x| < R} |g|^2 \biggr)^\oh &\lesssim R^{m-1-\beta} \biggl( \fint_{|x|< 1}|\nabla u_\h|^2 \biggr)^\oh \quad && R \geq 2, \\
\biggl( \fint_{|x| > R} |h|^2 \biggr)^\oh &\lesssim \biggl(\frac{ r}{ R} \biggr)^{(d-1)+ k + \beta} \biggl( \fint_{|x| > r}|\nabla v_\h|^2 \biggr)^\oh \quad &&R \geq 2r.
 \end{aligned}
\end{equation}
We begin with \eqref{two.scale.1} for $u_\h$ and claim that, if we denote by $u_\h'$ the projection of $u_\h$ onto the space $W_m^\h$, then $g$ takes the form
\begin{align}\label{definition.g}
g:= - \bigl( \partial_{ij}(u_\h-u_\h')C_{ij}^{\mathrm{sym}}
+(\phi_i a-\sigma_i)\nabla\partial_j\ut+(\psi_{ij}a-\Psi_{ij})\nabla\partial_{ij}u_\h \bigr).
\end{align}
We first show that the vector field $g$ as defined in \eqref{definition.g} satisfies \eqref{rhs.estimate}: By applying H\"older's inequality together with the assumptions \eqref{T.1} on $(\phi, \sigma)$, the boundedness of $a$ and \eqref{C.bdd} we indeed have for every $R \geq 1$ that 
\begin{align}\label{g.1}
\biggl(\fint_{|x|< R} |g |^2 \biggr)^\oh&\lesssim \biggl( \fint_{|x|< R} |(\psi, \Psi)|^2 \biggl)^\oh  \sup_{|x| < R}|\nabla^3 u_\h(x)|\\
&\quad + \biggl( \fint_{|x|< R} |(\phi, \sigma)|^2 \biggl)^\oh \sup_{|x| < R} |\nabla^2 \ut(x)|
+ \biggl(\fint_{|x|<R}| \nabla^2 (u_\h - u_\h')|^2\biggr)^\oh.\notag
\end{align}
We now appeal to \eqref{equi.u} and~\eqref{equi.u.tilde} for the terms with $u_\h$ and $\ut$ to infer that
\begin{align*}
 \sup_{|x| < R}|\nabla^3 u_\h(x)| + \sup_{|x| < R} |\nabla^2 \ut(x)| \lesssim {R}^{m-3} \biggl( \fint_{|x| < 1}|\nabla u_\h|^2 \biggr)^\oh.
\end{align*}
In addition, the continuity of the projection $X_m^\h \ni u_\h \mapsto u_\h' \in W_m^\h$ and the triangle inequality yield
$$
\biggl( \fint_{|x|< 1} |\nabla (u_\h - u_\h')|^2 \biggr)^\oh\lesssim \|\nabla u_\h \|_m \stackrel{\eqref{F2}}{\lesssim}\biggl( \fint_{|x|< 1} |\nabla u_\h|^2 \biggr)^\oh,
$$
so that if we combine the two previous estimates with \eqref{T.2} and \eqref{T.1}, the estimate  \eqref{g.1} turns into \eqref{rhs.estimate}, provided we show that for every $r \geq 2$
\begin{align}\label{sub.correct.}
 \biggl( \fint_{|x| <r} |(\psi, \Psi)|^2 \biggr)^\oh \lesssim r^{2-\beta}.
\end{align}
Since in~\eqref{psi.avgzero} we assumed $\fint_{|x| < 1} (\psi, \Psi) = 0$, we use the triangle inequality to bound
\begin{align*}
\biggl( \fint_{|x | < r}& |(\psi,\Psi) |^2 \biggr)^\oh = \biggl( \fint_{|x | < r} \biggl|(\psi,\Psi) -  \fint_{|x| <1} (\psi,\Psi)\biggl|^2 \biggr)^\oh\notag\\
&\leq \biggl( \fint_{|x | < r} \biggl|(\psi,\Psi) - \fint_{|x | < r}(\psi,\Psi)\biggl|^2 \biggr)^\oh 
+ \biggl|\fint_{|x | < r}(\psi,\Psi) -  \fint_{|x| <1} (\psi,\Psi) \biggl|\notag\\
& \stackrel{\eqref{T.2}}{\lesssim} r^{2-\beta} + \biggl|\fint_{|x | < r}(\psi,\Psi) -  \fint_{|x| <1} (\psi,\Psi) \biggl|,
\end{align*}
and also 
\begin{align*}
 \biggl|\fint_{|x | < r}(\psi,\Psi) -  \fint_{|x| <1} (\psi,\Psi) \biggl| &\leq \sum_{n=0}^{\log_2r} \biggl|\fint_{|x| < 2^n}(\psi,\Psi) - \fint_{ |x| < 2^{n+1}}(\psi,\Psi) \biggl|\notag\\
 &\lesssim \sum_{n=0}^{\log_2r} \biggl( \fint_{ |x| < 2^{n+1}} \biggl|(\psi,\Psi) - \fint_{|x| < 2^{n+1}}(\psi,\Psi) \biggl|^2 \biggr)^\oh\\
 &\stackrel{\eqref{T.2}}{\lesssim} r^{2-\beta}.
\end{align*}
We thus established \eqref{sub.correct.} as well as \eqref{rhs.estimate}.

\medskip

We now show \eqref{two.scale.1} with $g$ defined as in \eqref{definition.g}: Applying the gradient to the definition (\ref{m15}) of the second-order two-scale expansion $Eu_\h$, we obtain
from Leibniz' rule
\begin{align*}
\nabla(Eu_\h)=\partial_i(u_\h+\ut)(e_i+\nabla\phi_i)+\phi_i\nabla\partial_i(u_\h+\ut)
+\partial_{ij}u_\h\nabla\psi_{ij}+\psi_{ij}\nabla\partial_{ij}u_\h.
\end{align*}
Applying the tensor field $a$, we obtain by the characterizing property (\ref{i5}) of $\sigma_i$,
and again Leibniz' rule,
\begin{align*}
a \nabla(Eu_\h)&=\ah\nabla(u_\h+\ut)+\nabla\cdot(\partial_i(u_\h+\ut)\sigma_i)\nonumber\\
&\quad +\partial_{ij}(u_\h+\ut)(\phi_i a-\sigma_i)e_j
+\partial_{ij}u_\h a\nabla\psi_{ij}+\psi_{ij}a\nabla\partial_{ij}u_\h.
\end{align*}
We now appeal to the characterizing property (\ref{Psi}) of $\Psi_{ij}$, and Leibniz' rule, to reformulate the flux further:
\begin{align}\label{ir1}
a \nabla(Eu_\h)&=\ah\nabla(u_\h+\ut)+\partial_{ij}u_\h C_{ij}\nonumber\\
&\quad +\nabla\cdot(\partial_i(u_\h+\ut)\sigma_i+\partial_{ij}u_\h\Psi_{ij})\nonumber\\
&\quad +(\phi_i a-\sigma_i)\nabla\partial_j\ut+(\psi_{ij}a-\Psi_{ij})\nabla\partial_{ij}u_\h.
\end{align}
We finally apply the divergence to this identity. We note that by the symmetry of third derivatives we have
$\nabla\cdot(\partial_{ij}u_\h C_{ij})$ $=\nabla\cdot(\partial_{ij}u_\h C_{ij}^{\mathrm{sym}})$. We also note that the terms in the
second right-hand side line vanish because for any skew-symmetric $\sigma$ we have $\nabla\cdot\nabla\cdot\sigma=0$ by the
symmetry of second derivatives. Hence by the $\ah$-harmonicity of $u_\h$ 
and the defining equation (\ref{u.tilde}) for $\ut$ we obtain
\begin{align*}
\lefteqn{\nabla\cdot a \nabla(Eu_\h)}\nonumber\\
&=\nabla\cdot\Big(\partial_{ij}(u_\h-u_\h')C_{ij}^{\mathrm{sym}}
+(\phi_i a-\sigma_i)\nabla\partial_j\ut+(\psi_{ij}a-\Psi_{ij})\nabla\partial_{ij}u_\h\Big),
\end{align*}
which is (\ref{two.scale.1}) with (\ref{definition.g}).

\smallskip

We now turn to \eqref{two.scale.1} and \eqref{rhs.estimate} for $Ev_\h$. Similarly as above, we denote by $v_\h'$ the projection of $v_\h$ onto the space $Z_k^\h(1)$ defined in Lemma \ref{folklore}.
We remark that the proof of \eqref{two.scale.1} with g solving \eqref{definition.g} relies on the structure of the two-scale expansion $Eu_\h$ and 
does not depend on any property of $u_\h$ and $\ut$ besides the $\ah$-harmonicity of $u_\h$ and the equation \eqref{R7} for $\ut$.
Therefore, by definition of $v_\h$, equation \eqref{R2} and the property $-C^{sym,*} = C^{\mathrm{sym}}$ (see Proposition \ref{stoch.results}), the same argument works also for $Ev_\h$ which therefore solves \eqref{two.scale.1} with 
\begin{align}\label{definition.h}
h:= \bigl( (\psi_{ij}^*a^* - \Psi^*) \nabla \partial_{ij}v_\h + (\phi_i^*a^* - \sigma_i^*)\nabla \partial_i \vt -  \partial_{ij}(v_\h - v_\h')C_{ij}^{\mathrm{sym}} \bigr).
\end{align}
We estimate $h$ first on dyadic annuli. As for $g$, we apply H\"older's inequality and the assumptions \eqref{T.1} for $(\phi^*, \sigma^*)$, \eqref{C.bdd}, and \eqref{sub.correct.} for $(\psi^*, \Psi^*)$ to obtain for every $R \geq 2r$ that
\begin{align*}
 \biggl(\fint_{R < |x|< 2R} |h |^2 \biggr)^\oh &\lesssim {R}^{2-\beta} \sup_{R <|x| < 2R}|\nabla^3 v_\h(x)| + R^{2-\alpha} \sup_{R <|x| < 2R} |\nabla^2 \vt(x)| \\
&\quad + \biggl(\fint_{R < |x|< 2R}| \nabla^2 (v_\h - v_\h')|^2\biggr)^\oh.
\end{align*}
It remains to apply  to $v_\h$ and $\vt$ estimates \eqref{equi.v.tilde} and \eqref{equi.v.full} in the proof of Lemma \ref{Rom5} and Corollary~\ref{Cor5.b} on $v_\h - v_\h'$ to conclude that
\begin{align*}
 \biggl(\fint_{R < |x|< 2R} |h |^2 \biggr)^\oh\lesssim  \biggl( \frac{r}{R}\biggr)^{d - 1 + k + \beta}\biggl( \fint_{|x| > r}|\nabla v_\h|^2 \biggr)^\oh.
\end{align*}
By summing over dyadic annuli, we obtain inequality \eqref{rhs.estimate} also for $h$ and thus conclude the proof of Step~\ref{C:Rom4.1}.

\PfStep{C:Rom4}{C:Rom4.2}
Conclusion by means of Proposition \ref{Lg} and \ref{Ld};
we start by tackling the case of the spaces of growing functions. Since by Lemma \ref{Rom5} the function $Eu_\h$ satisfies for $R \geq 2$ 
(here, thanks to assumption \eqref{T.1} on the growth of the correctors, we may take $R \geq 2$)
\begin{align*}
\biggl( \fint_{|x| < R} |\nabla E u_\h|^2 \biggr)^\oh \stackrel{\eqref{m08}}{\sim} \biggl( \fint_{|x| <R} |\nabla u_\h|^2 \biggr)^\oh \stackrel{\eqref{F2}}{\lesssim}
{R}^{m-1}\biggl(\fint_{|x|< 1}|\nabla u_\h|^2 \biggr)^\oh,
\end{align*}
we may apply Proposition \ref{Lg} to $Eu_\h$ which, by the previous steps, satisfies \eqref{two.scale.1} and~\eqref{rhs.estimate}. We infer that there exists $u \in X_m$ such that the first inequality in \eqref{Pg4} holds, i.e
\begin{align*}
\bigr(\fint_{|x| < R} |\nabla ( u - Eu_\h)|^2 \biggr)^\oh \lesssim {R}^{m-1-\beta} \biggl( \fint_{|x| < 1} |\nabla u_\h|^2 \biggr)^\oh.
\end{align*}
This implies $\|u -E u_\h\|_{m-\beta} \lesssim \| u_\h\|_m$.

\smallskip

We now turn to the spaces of decaying functions: We fix $v_\h \in Y_k(r)$. By Lemma \ref{Rom5} and Lemma \ref{LF1}, (c), the function $Ev_\h$ satisfies 
\begin{align*}
 \biggl( \fint_{|x|> R} |\nabla E v_\h|^2 \biggr)^\oh \lesssim \biggl( \frac{ r}{ R} \biggr)^{d + k -1}\biggl( \fint_{|x|> r}|\nabla v_\h|^2 \biggr)^\oh,
\end{align*}
and, by Step 1., also \eqref{two.scale.1} and~\eqref{rhs.estimate}. It remains to apply Proposition \ref{Pd} with $r_*=r$ and select an element $v \in Y_k(r)$ such that for every $R \geq 2r$ 
\begin{align*}
 \biggl(\fint_{|x|> R} |\nabla( v - Ev_\h)|^2 \biggr)^\oh \lesssim \biggl(\frac{ r} {R }\biggr)^{(d-1) + k + \beta}\biggl( \fint_{|x| > r}|\nabla v_\h|^2\biggr)^{\oh},
\end{align*}
i.e. $\| v - Ev_\h\|_{k+\beta,r} \lesssim \| v_\h \|_{k, r}$.
\end{proof}


\subsection{Proof of Proposition~\ref{localization}}\hspace*{\fill} 

\begin{proof}
W.l.o.g. we assume that $r_*=1$. We give the argument for the estimate \eqref{d.loc}; estimate~\eqref{g.loc} may be proven in an analogous way. For the sake of simplicity, since we restrict our argument only to the case of the decaying functions, we omit the $^*$-superscript for the correctors and write $(\phi, \sigma)$ and $(\Psi, \psi)$ instead of $(\phi^*, \sigma^*)$ and $(\Psi^*, \psi^*)$. We set $D:= \frac{|y|}{4}$.

\medskip

We begin by showing that for $v_\h \in Y_k^\h(r)$ and its image $v\in Y_k(r)$ under $L_Y^\h$ it holds 
\begin{align}\label{localization.1}
\| v - \tilde E v_\h \|_{k+\beta,r} \lesssim \| v_\h \|_{k, r},
\end{align}
where $\tilde E v_\h$ is the two-scale expansion of $v_\h$ centered in $x_0$, namely
\begin{align}\label{Etilde}
 \tilde Ev_\h = \bigg(1 + \phi_i \partial_i + \bigg(\psi_{ij}- \fint_{|x-y|< r_*(y)}\psi_{ij}\bigg)\partial_{ij}\bigg) v_\h + (1 + \phi_i\partial_i)\vt.
\end{align}
By homogeneity, w.l.o.g. we assume that $ \| v_\h \|_{k,r}= 1$. By the estimate for $L_Y^\h$ of Proposition \ref{Rom4} and the triangle inequality, to prove \eqref{localization.1} it suffices to show that for every $R \geq D$
\begin{align}\label{localization.2}
\biggl( \fint_{|x| > R} |\nabla ( Ev_\h - \tilde E v_\h)|^2 \biggr)^\oh \lesssim R^{-(d-1 + k + \beta)}.
\end{align}
\BR5 By the definitions of $Ev_\h$ and $\tilde Ev_\h$ (see~\eqref{m15} and \eqref{Etilde}\ER) and estimate \eqref{equi.v.full} for $v_\h$ (after a summation over dyadic annuli) we have
\begin{align}\label{localization.3}
\biggl( \fint_{|x| > R} |\nabla ( Ev_\h - \tilde E v_\h)|^2 \biggr)^\oh \lesssim R^{-(d-1 + k + 3)}\bigg|\fint_{|x-y|< r_*(y)}\psi_{ij}\bigg|.
\end{align}
Since we assumed $\fint_{|x|< 1}(\psi, \Psi)= 0$, we have
\begin{align*}
\biggl| \fint_{|x-y|< r_*(y)}(\psi, \Psi)\biggl| &= \biggl|\fint_{|x-y|<  r_*(y)}(\psi, \Psi) - \fint_{|x|< 1}(\psi, \Psi)\biggl|
\end{align*}
and claim that 
\begin{align}\label{Psi.centers}
\biggl|\fint_{|x-y|<  r_*(y)}(\psi, \Psi) - \fint_{|x|< 1}(\psi, \Psi)\biggl| \lesssim D^{2-\beta}.
\end{align}
Indeed, by the triangle inequality we write
\begin{equation*}
\begin{aligned}
\biggl| \fint_{|x| <1} &(\psi,\Psi)- \fint_{|x-y| < r_*(y)} (\psi,\Psi)\biggr| 
\\
&\lesssim 
\biggl( \fint_{|x - y | < 2D} \biggl|(\psi,\Psi) -  \fint_{|x| <1} (\psi,\Psi)\biggl|^2 \biggr)^\oh\\
&\quad\quad\quad+ \biggl( \fint_{|x - y | < 2D} \biggl|(\psi,\Psi) -  \fint_{|x-y| < r_*(y)} (\psi,\Psi)\biggl|^2 \biggr)^\oh 
\\
&\lesssim
\biggl( \fint_{|x | < 4D} \biggl|(\psi,\Psi) -  \fint_{|x| <1} (\psi,\Psi)\biggl|^2 \biggr)^\oh\\
&\quad\quad\quad+ \biggl( \fint_{|x-y| < 4D} \biggl|(\psi,\Psi) -  \fint_{|x-y| < r_*(y)} (\psi,\Psi)\biggl|^2 \biggr)^\oh .
\end{aligned}
\end{equation*}
The first term is bounded by $D^{2-\beta}$ by \eqref{sub.correct.}; for the second one we use the same argument as in \eqref{sub.correct.} since we may use assumption \BR5 \eqref{T.2} at $y$ \ER and bound this term by $\lesssim \bigl(\frac{D}{r_*(y)} \bigr)^{2-\beta}\lesssim D^{2-\beta}$.
By inserting the bound \eqref{Psi.centers} into \eqref{localization.3} yields \eqref{localization.2} and thus also \eqref{localization.1}.

\medskip

Since $v$ is $a$-harmonic in $\{|x| > r \}$ with $r \leq D$, by the same calculation carried out in Step 1 in the proof of Proposition \ref{Rom4} we infer that $v- \tilde E v_\h$ solves
\begin{align}\label{localization.4}
-\nabla \cdot a\nabla ( v - \tilde E v_\h) = \nabla \cdot h \ \ \ \text{in $\{|x| > r \}$,}
\end{align}
with $h$ as in \eqref{definition.h} where $\psi$ and $\Psi$ are substituted by $\psi - \fint_{|x-y| < r_*(y)}\psi$ and $\Psi - \fint_{|x-y| < r_*(y)}\Psi$. An argument analogous to the one for the second line of \eqref{rhs.estimate}, 
yields that for every $r_*(y) \leq R' \leq D$
\begin{align*}
\biggl( \fint_{|x - y| < R'}|h|^2 \biggr)^\oh \lesssim D^{-(d + 1 + k)} (r_*(y))^\beta (R')^{2-\beta}.
\end{align*}
In particular, since $r_*(y) \le D$, 
this implies for every $ r_*(y) \leq R' \leq D$ 
\begin{align*}
\biggl( \fint_{|x - y| < R'}|h|^2 \biggr)^\oh \lesssim D^{-(d - 1 + k + \beta)}\biggl( \frac{R'}{D} \biggr)^{2-\beta}.
\end{align*}
We now appeal to Lemma \ref{propositionlocalize} of Section \ref{abstract.r}, this time over a ball centered at $y$; we remark that the assumption \eqref{T.1} yields the result of Proposition \ref{Ck.1} also centered at $y$, with $r_*$ substituted by $r_*(y)$. Therefore, conditions \eqref{F1} and \eqref{F2} hold for the radius $r_*(y)$ also if centered at $y$, and the results of Section \ref{abstract.r} as well. We thus apply Lemma \ref{propositionlocalize} in $\{ |x - y| < D \}$ to the function $D^{d-1 + k +\beta }(v - \tilde E v_\h)$ and infer, by equation \eqref{localization.4}, inequality \eqref{localization.1} and the last estimate above, that
\begin{align*}
\biggl(\fint_{|x-y| < r_*(y)}|v - \tilde E v_\h|^2 \biggr)^\oh \lesssim D^{-(d -1 + k +\beta)}.
\end{align*}
By reasoning as in \eqref{localization.2}, we may substitute $\tilde E v_\h$ with $E v_\h$ in the above estimate and thus conclude the proof of Proposition \ref{localization}.
\end{proof}


\subsection{Proof of Proposition~\ref{invariants.preserved}}\hspace*{\fill} 

\begin{proof}
\PfStart{P10}

For $m\ge 2$ and $r\ge r_*$ let $u_\h\in X_m^\h$, $v_\h\in Y_{m-1}^\h(r)$, $u\in X_m$, and $v\in Y_{m-1}(r)$ 
be as in Proposition \ref{Rom4} (with $k=m-1$); let $(\ut,\vt)$ be defined in terms of
$(u_\h,v_\h)$ according to Lemmas \ref{Rom1} and \ref{Rom2}. 
We give a short overview of the proof and its steps.
Introducing an ``asymptotic invariant'' of the pair of two-scale expansions $(Eu_\h,Ev_\h)$, see  Step~\ref{P10:3}, 
we split the identity $(u,v)$ $=(u_\h,v_\h)_\h$ into two parts which are stated in Step~\ref{P10:3} and Step~\ref{P10:Conclusion}. 
Step~\ref{P10:0} collects the (obvious) asymptotic properties of $(u_\h,\ut,v_\h,\vt)$;
Step~\ref{P10:2} deduces the (easy) asymptotic properties of $(Eu_\h,Ev_\h)$. Step \ref{P10:1} recalls the error estimates on $\nabla(Eu_\h-u,Ev_\h-v)$ and deduces those on $(Eu_\h-u,Ev_\h-v)$. Steps~\ref{P10:flux}, \ref{P10:reduce}, and
\ref{P10:reduce2} serve to get rid of lower-order terms in the asymptotic invariant. Of structural significance
are Step~\ref{P10:reduce3}, where the normalization of the correctors is used, and Steps~\ref{P10:sym} and \ref{P10:Null},
where cancellation properties are uncovered.  

\PfStep{P10}{P10:0}
Estimates of $(u_\h,\ut,v_\h,\vt)$;
we claim
\begin{align}
\limsup_{|x|\uparrow\infty}|x|^{-m}(|u_\h|+|x||\nabla u_\h|+|x|^2|\nabla^2 u_\h|+|x|^3|\nabla^3 u_\h|)<\infty,\label{wi36}\\
\limsup_{|x|\uparrow\infty}|x|^{-m+1}(|\ut|+|x||\nabla\ut|+|x|^2|\nabla^2\ut|)<\infty,\label{wi37}\\
\limsup_{|x|\uparrow\infty}|x|^{(d-2)+(m-1)}(|v_\h|+|x||\nabla v_\h|+|x|^2|\nabla^2 v_\h|+|x|^3|\nabla^3 v_\h|)<\infty,\label{wi38}\\
\limsup_{|x|\uparrow\infty}|x|^{(d-2)+m}(|\vt|+|x||\nabla\vt|+|x|^2|\nabla^2\vt|)<\infty.\label{wi39}
\end{align}
Besides the estimate on the functions themselves,
these are just a qualitative reformulation of the statements (\ref{equi.u}), \eqref{equi.u.tilde},
(\ref{equi.v}), and (\ref{equi.v.tilde}), at the beginning of the proof of Lemma \ref{Rom5}. For $u_\h$ and
$\ut$ themselves, the statement is obvious since these are polynomials of degree $m$ and $m-1$, respectively.
For $v_\h$ and $\vt$, this follows from the decay estimate on $\nabla v_\h$ and $\nabla \vt$ by integration
from infinity.

\PfStep{P10}{P10:2}
Estimate of two-scale expansion; we claim
\begin{align*}
\limsup_{R\uparrow\infty}R^{-m}\bigg(\fint_{|x|<R}(Eu_\h)^2\bigg)^\frac{1}{2},
\limsup_{R\uparrow\infty}R^{(d-2)+(m-1)}\bigg(\fint_{|x|>R}(Ev_\h)^2\bigg)^\frac{1}{2}<\infty.\nonumber
\end{align*}
Indeed,
for $Eu_\h$, see \eqref{m15},
this follows from the $L^\infty$-control on the growth of $(u_\h,\nabla u_\h,\nabla^2 u_\h,\ut,\nabla\ut)$
stated in (\ref{wi36}) and (\ref{wi37}) together with the (sub)-linear $L^2$-growth of $(\phi,\psi)$
assumed in (\ref{T.1}) and (\ref{T.2}).
For $Ev_\h$ this follows from the $L^\infty$-control on the decay of $(v_\h,\nabla v_\h,\nabla^2 v_\h,\vt,\nabla\vt)$,
see \eqref{wi38} and (\ref{wi39}), together with the same $L^2$-control on $(\phi,\psi)$.

\PfStep{P10}{P10:1}
Estimate of error in two-scale approximation; we claim
\begin{align}
\lim_{R\uparrow\infty}R^{-m+1}\bigg(\fint_{|x|<R}(|Eu_\h-u|+R|\nabla(Eu_\h-u)|)^2\bigg)^\frac{1}{2}=0,\label{wi34}\\
\lim_{R\uparrow\infty}R^{(d-2)+m}\bigg(\fint_{|x|>R}(|Ev_\h-v|+R|\nabla(Ev_\h-v)|)^2\bigg)^\frac{1}{2}=0.\label{wi41}
\end{align}
Indeed, the statement on the gradients is a consequence of the outcome of Proposition \ref{Rom4} with $k=m-1$, which 
on a qualitative level reads as
\begin{align}
\limsup_{R\uparrow\infty}R^{-m+\beta}\bigg(\fint_{|x|<R}(R|\nabla(Eu_\h-u)|)^2\bigg)^\frac{1}{2}<\infty,\label{wi30}\\
\limsup_{R\uparrow\infty}R^{(d-2)+(m-1)+\beta}\bigg(\fint_{|x|>R}(R|\nabla(Ev_\h-v)|)^2\bigg)^\frac{1}{2}<\infty,\label{wi31}
\end{align}
see the definitions (\ref{m01}) and (\ref{m02.bis}) of the norms, and of $\beta>1$.
It remains to address the functions themselves. In view of (\ref{wi30}) and (\ref{wi31})
it suffices to establish for arbitrary exponent $\alpha>0$ and function $w$
\begin{align}
\sup_{R\ge r}R^{-\alpha}\bigg(\fint_{|x|<R}(w-\fint_{|x|<r}w)^2\bigg)^\frac{1}{2}
\lesssim\sup_{R\ge r}R^{-\alpha}\bigg(\fint_{|x|<R}(R|\nabla w|)^2\bigg)^\frac{1}{2},\label{wi32}\\
\sup_{R\ge r}R^{\alpha}\bigg(\fint_{|x|>R}w^2\bigg)^\frac{1}{2}
\lesssim\sup_{R\ge r}R^{\alpha}\bigg(\fint_{|x|>R}(R|\nabla w|)^2\bigg)^\frac{1}{2}.\label{wi29}
\end{align}
Let us point out that in (\ref{wi32}), which we apply to $\alpha=m-\beta$ (which is positive by our assumptions
of $m\ge 2>\beta$) and to $w=Eu_\h-u$, the averaged value $\fint_{|x|<r}w$ is over the {\it fixed} ball
$\{|x|<r\}$ and thus does not affect the qualitative statement (\ref{wi34}).

\medskip

We first turn to (\ref{wi32}); by scaling we may w.\ l.\ o.\ g.\ assume that $r=1$, and that $R=2^N$ is dyadic,
so that it suffices to show
\begin{align}\label{wi36bis}
\bigg(\fint_{|x|<2^N}\biggl(w-\fint_{|x|<1}w\biggl)^2\bigg)^\frac{1}{2}
\lesssim\sup_{0\le n\le N}2^{\alpha(N-n)}\bigg(\fint_{|x|<2^n}(2^n|\nabla w|)^2\bigg)^\frac{1}{2}.
\end{align}
This follows via telescoping from Poincar\'e's inequality: 
\begin{align*}
\lefteqn{\bigg(\fint_{|x|<2^N}\biggl(w-\fint_{|x|<1}w\biggl)^2\bigg)^\frac{1}{2}}\nonumber\\
&\le \bigg(\fint_{|x|<2^N}\biggl(w-\fint_{|x|<2^N}w\biggl)^2\bigg)^\frac{1}{2}
+\sum_{n=1}^N\biggl|\fint_{|x|<2^n}w-\fint_{|x|<2^{n-1}}w\biggl|\nonumber\\
&\lesssim \sum_{n=1}^N\bigg(\fint_{|x|<2^n}\biggl(w-\fint_{|x|<2^n}w\biggl)^2\bigg)^\frac{1}{2}
\lesssim \sum_{n=1}^N\bigg(\fint_{|x|<2^n}(2^n|\nabla w|)^2\bigg)^\frac{1}{2}.
\end{align*}
That this sum over $n$ in can be estimated by sup over $n$ in (\ref{wi36bis})
is a consequence of the prefactor $2^{\alpha(N-n)}$ with $\alpha>0$ in the latter.

\medskip

We now turn to (\ref{wi29}). By a division of the left-hand side integral
into dyadic annuli and the convergence of the geometric series (thanks to $\alpha>0$), this reduces to
\begin{align*}
\bigg(\fint_{R<|x|<2R}w^2\bigg)^\frac{1}{2}\lesssim R\bigg(\fint_{|x|>R}|\nabla w|^2\bigg)^\frac{1}{2}
\end{align*}
for any radius $R$,
which follows from H\"older's inequality on the annulus $\{R<|x|<2R\}$ followed by Poincar\'e's inequality
on the exterior domain $\{|x|>R\}$ (recall $d>2$):
\begin{align*}
\bigg(\fint_{R<|x|<2R}w^2\bigg)^\frac{1}{2}\lesssim R\bigg(\fint_{R<|x|<2R}|w|^\frac{2d}{d-2}\bigg)^{\frac{1}{2}-\frac{1}{d}}
\lesssim R\bigg(\fint_{|x|>R}|\nabla w|^2\bigg)^\frac{1}{2}.
\end{align*}

\PfStep{P10}{P10:3}
Invariant equals asymptotic invariant of two-scale expansion; we claim
\begin{align*}
(u,v)=-\lim_{R\uparrow\infty}\int\nabla\eta\cdot(Ev_\h a\nabla Eu_\h-Eu_\h a^*\nabla Ev_\h),
\end{align*}
where $\eta_R(x)=\eta(\frac{x}{R})$ with $\eta$ a smooth cut-off function for $\{|x|<1\}$ in $\{|x|<2\}$.
Indeed, we recall definition (\ref{bil.def}) of the invariants, and write
\begin{align}
\lefteqn{\big(Ev_\h a\nabla Eu_\h-Eu_\h a^*\nabla Ev_\h\big)-\big(va\nabla u-ua^*\nabla v\big)}\nonumber\\
&=(Ev_\h-v)a\nabla u-(Eu_\h-u)a^*\nabla v\label{wi40}\\
&+Ev_\h a\nabla(Eu_\h-u)-Eu_\h a^*\nabla(Ev_\h-v).\label{wi42}
\end{align}
By definition of $\eta_R$ we have for any functions $u$ and $v$
\begin{align*}
\biggl|\int\nabla\eta_R\cdot(va\nabla u)\biggl|
&\lesssim R^{d-1}\bigg(\fint_{|x|>R}v^2\bigg)^\frac{1}{2}\bigg(\fint_{|x|<2R}|\nabla u|^2\bigg)^\frac{1}{2},\\
\biggl|\int\nabla\eta_R\cdot(ua^*\nabla v)\biggl|
&\lesssim R^{d-1}\bigg(\fint_{|x|<2R}u^2\bigg)^\frac{1}{2}\bigg(\fint_{|x|>R}|\nabla v|^2\bigg)^\frac{1}{2}.
\end{align*}
Hence the contribution from the first term in line (\ref{wi40}) vanishes in the limit $R\uparrow\infty$
by (\ref{wi41}) in Step \ref{P10:1} and by the defining
property of $u$ being in $X_m$; likewise, the second term in line (\ref{wi40}) vanishes by (\ref{wi34}) in Step \ref{P10:1}
and $v\in Y_{m-1}(r)$. The contribution from the two terms in line (\ref{wi42}) vanishes by Steps \ref{P10:2} and \ref{P10:1}.

\PfStep{P10}{P10:flux}
A suitable representation of asymptotic invariant; we claim 
\begin{align}\label{wi25}
\lefteqn{\int\nabla\eta_R\cdot(Ev_\h a\nabla Eu_\h)=\int\nabla\eta_R\cdot}\nonumber\\
&\Big(Ev_\h\big(\ah\nabla(u_\h+\ut)+\partial_{ij}u_\h C_{ij}
+(\phi_ia-\sigma_i)\nabla\partial_i\ut+(\psi_{ij}a-\psi_{ij})\nabla\partial_{ij}u_\h\big)\nonumber\\
&-(\partial_i(u_\h+\ut)\sigma_i+\partial_{ij}u_\h\Psi_{ij})\nabla Ev_\h\Big)
\end{align}
and a similar formula with the roles of $v_\h$ and $u_\h$ exchanged and with $a$ replaced by $a^*$.
Indeed, this follows from the flux representation (\ref{ir1}) together with the integration by parts formula
\begin{align*}
\int\nabla\eta_R\cdot(v\nabla\cdot\sigma)=-\int\nabla\eta_R\cdot\sigma\nabla v.
\end{align*}
that holds for any function $v$ and any skew-symmetric tensor field $\sigma$ that
are smooth on $\{|x|\ge r\}$.

\PfStep{P10}{P10:reduce}
Getting rid of lower-order terms
in asymptotic invariant; we claim
\begin{align}\label{wi27}
\lim_{R\uparrow\infty}&\bigg|\int\nabla\eta_R\cdot(Ev_\h a\nabla Eu_\h)
-\int\nabla\eta_R\cdot\Big(v_\h\big(\ah\nabla(u_\h+\ut)+\partial_{ij}u_\h C_{ij}\big)\nonumber\\
&+(\phi_i^*\partial_iv_\h+\vt)\ah\nabla u_\h
-\partial_iv_\h\partial_ju_\h\sigma_j(e_i+\nabla\phi_i^*)\Big)\bigg|=0
\end{align}
and a similar statement with the roles of $v_\h$ and $u_\h$ exchanged and with $a$ replaced by $a^*$.

\medskip

Here comes the argument, where we drop the index $h$ (besides on $\ah$): 
We start from the representation from Step \ref{P10:flux}. 
The integrand of (\ref{wi25}) consists of two summands, 
``Product A'' involving $Ev$ and ``Product B'' involving $\nabla Ev$. 
We now list the contributions of $(v,\tilde v)$ to these products with increasing order of $L^2$-decay, indicating
the decay exponent (where a $-$ means slightly better and a $+$ as little worse as we wish) and indicating
which of the two products they belong to:
\begin{align*}
\begin{array}{lc@{\hspace{5ex}}ll@{\hspace{5ex}}l}
v1)&v&-(d-2)-m+1&&\mbox{A},\\
v2)&\phi_i^*\partial_iv&-(d-2)-m&+&\mbox{A},\\
v3)&\tilde v&-(d-2)-m&&\mbox{A},\\
v4)&\partial_iv(e_i+\nabla\phi_i^*)&-(d-2)-m&&\mbox{B},\\
v5)&\phi_i^*\partial_i\tilde v+\psi_{ij}^*\partial_{ij}v&-(d-2)-m&-&\mbox{A},\\
v6)&\phi_i^*\nabla\partial_iv&-(d-2)-m-1&+&\mbox{B},\\
v7)&\partial_i\tilde v(e_i+\nabla\phi_i^*)+\partial_{ij}v\nabla\psi_{ij}^*&-(d-2)-m-1&&\mbox{B},\\
v8)&\phi_i^*\nabla\partial_i\tilde v+\psi_{ij}^*\nabla\partial_{ij}v&-(d-2)-m-1&-&\mbox{B}.
\end{array}
\end{align*}
The basis for these qualitative estimates are the $L^\infty$-estimates of $(v,$ $\nabla v,$ $\nabla^2 v,$ $\nabla^3 v,$ $\tilde v,$
$\nabla\tilde v,$ $\nabla^2\tilde v)$ in Step \ref{P10:0},
see (\ref{wi38}) and (\ref{wi39}).
Clearly, the pluses in v2) and v6) (and of u1) below) are a consequence of the subalgebraic $L^2$-growth (\ref{T.1}) of $(\phi,\sigma)$,
while the minuses in v5) and v8) (and of u3) and u5) below) are a consequence of the sublinear $L^2$-growth (\ref{T.2}) of $(\psi,\Psi)$.
The $L^2$-boundedness of the gradients $\nabla(\phi^*,\psi^*)$ needed for v4) and v7) was established through
Caccioppoli estimates, see (\ref{grad.phi}) and 
(\ref{grad.psi}). Based on (\ref{wi36}) and (\ref{wi37}) we obtain for the contributions of $(u,\tilde u)$, 
listed with decreasing growth,
\begin{align*}
\begin{array}{lc@{\hspace{5ex}}ll@{\hspace{5ex}}l}
u1)&\partial_j u\sigma_j&m-1&+&\mbox{B},\\
u2)&\ah\nabla u&m-1&&\mbox{A},\\
u3)&\partial_j\tilde u\sigma_j+\partial_{ij}u\Psi_{ij}&m-1&-&\mbox{B},\\
u4)&\ah\nabla\tilde u+\partial_{ij}uC_{ij}&m-2&&\mbox{A},\\
u5)&(\phi_ia-\sigma_i)\nabla\partial_i\tilde u+(\psi_{ij}a-\Psi_{ij})\nabla\partial_{ij}u&m-2&-&\mbox{A}.
\end{array}
\end{align*}
Since $\int|\nabla\eta|\lesssim R^{d-1}$, if the sum of two rates is strictly less than $-d+1$, 
then the contribution to the integral is negligible in the limit $R\uparrow\infty$. Therefore, the contribution from v1) $\times$ u4)
to product A, which gives rise to $-d+1-$, is negligible. Likewise the contribution from v2) $\times$ u5) 
to A is negligible, since it gives rise to $-d+$ because the sum of two $+$ acts like a plus $+$. 
Likewise, the contributions from v4) $\times$ u3) and of v6) $\times$ u1) to Product B are negligible. 
All subordinate contributions are a fortiori negligible. This leaves us with the four surviving terms
from A and a single surviving term from B:
\begin{align*}
\mbox{v1)}\times\mbox{u2)},\quad
\mbox{v1)}\times\mbox{u5)},\quad
\mbox{v2)}\times\mbox{u2)},\quad
\mbox{v3)}\times\mbox{u2)},\quad
\mbox{v4)}\times\mbox{u1)}.
\end{align*}
The statement with the roles of $v_\h$ and $u_\h$ exchanged, and with $a$ replaced by $a^*$,
follows by the same argument since it is the {\it sum} of the decay rates that matters.

\PfStep{P10}{P10:reduce2} Getting rid of more lower-order terms in asymptotic invariant; we claim
\begin{align*}
\lim_{R\uparrow\infty}&\bigg|\int\nabla\eta_R\cdot(Ev_\h a\nabla Eu_\h)
-\int\nabla\eta_R\cdot\Big(v_\h\ah\nabla u_\h+v_\h'(\ah\nabla\ut+\partial_{ij}u_\h'C_{ij})\nonumber\\
&+(\phi_i^*\partial_iv_\h'+\vt)\ah\nabla u_\h'
-\partial_iv_\h'\partial_ju_\h'\sigma_j(e_i+\nabla\phi_i^*)\Big)\bigg|=0
\end{align*}
and a similar statement with the roles of $v_\h$ and $u_\h$ exchanged, and with $a$ replaced by $a^*$.
We drop the index $h$ (besides on $\ah$). 
Our goal is to substitute $(u,v)$ by $(u',v')$ 
in all but the first \BR6 right-hand side term \ER in (\ref{wi27}). We recall that $u'$ is the $m$-homogeneous part of $u$ and $v'$ 
is the $(d-2)-(m-1)$-homogeneous part of $v$. The difference satisfies 
$(\delta u,\delta v)\in X_{m-1}^\h\times Y_m^\h$ so that as in Step \ref{P10:0}
\begin{align*}
\limsup_{|x|\uparrow\infty}|x|^{-m+1}(|\delta u|+|x||\nabla\delta u|+|x|^2|\nabla^2\delta u|)<\infty,\\
\limsup_{|x|\uparrow\infty}|x|^{(d-2)+m}(|\delta v|+|x||\nabla\delta v|)<\infty.
\end{align*}
Equipped with this additional information we may proceed as in Step \ref{P10:reduce} by dealing with the effect on the
three products
\begin{align*}
\mbox{A}:\;v\;(\ah\nabla\tilde u+\partial_{ij}uC_{ij}),
\quad\mbox{B}:\phi_i^*\partial_iv\;\ah\nabla u,\quad\mbox{C}:\;\partial_ju\sigma_j\;\partial_iv(e_i+\nabla\phi_i^*).
\end{align*}
This time, the tables are given by 
\begin{align*}
\begin{array}{lc@{\hspace{5ex}}ll@{\hspace{5ex}}l}
v1)&v'&-(d-2)-m+1&&\mbox{A},\\
v2)&\phi_i^*\partial_iv'&-(d-2)-m&+&\mbox{B},\\
v3)&\delta v&-(d-2)-m&&\mbox{A},\\
v4)&\partial_iv'(e_i+\nabla\phi_i^*)&-(d-2)-m&&\mbox{C},\\
v5)&\phi_i^*\partial_i\delta v&-(d-2)-m-1&+&\mbox{B},\\
v6)&\partial_i\delta v(e_i+\nabla\phi_i^*)&-(d-2)-m-1&&\mbox{C}\\
\end{array}
\end{align*}
and
\begin{align*}
\begin{array}{lc@{\hspace{5ex}}ll@{\hspace{5ex}}l}
u1)&\partial_j u'\sigma_j&m-1&+&\mbox{C},\\
u2)&\ah\nabla u'&m-1&&\mbox{B},\\
u3)&\partial_j\delta u\sigma_j&m-2&+&\mbox{C},\\
u4)&\ah\nabla\tilde u+\partial_{ij}u'C_{ij}&m-2&&\mbox{A},\\
u5)&\ah\nabla\delta u&m-2&&\mbox{B},\\
u6)&\partial_{ij}\delta uC_{ij}&m-3&&\mbox{A}.
\end{array}
\end{align*}
The only with sum of rates $\ge -d+1$ and thus surviving terms are the leading-order
terms for the three products
\begin{align*}
v1)\;\times\;u4),\quad v2)\;\times u2),\quad v4)\;\times u1).
\end{align*}

\PfStep{P10}{P10:reduce3}
Making use of the normalization
of the corrector; we claim that for all $R>0$ 
\begin{align}
\lefteqn{\lim_{R'\uparrow\infty}\int\nabla\eta_{R'}\cdot(Ev_\h a\nabla Eu_\h-Eu_\h a^*\nabla Ev_\h)
=-(u_\h,v_\h)_\h}\nonumber\\
&+\int\nabla\eta_R\cdot\big(v_\h'(\ah\nabla\ut+\partial_{ij}u_\h'C_{ij})-\ut \ah^*\nabla v_\h'\big)\nonumber\\
&-\int\nabla\eta_R\cdot\big(u_\h'(\ah^*\nabla\vt+\partial_{ij}v_\h'C_{ij}^*)-\vt \ah\nabla u_\h'\big)\nonumber\\
&-\int\partial_k\eta_R\partial_iv_\h'\partial_ju_\h'C_{jki}.\label{wi26}
\end{align}
We note that $\partial_iv_\h'\partial_ju_\h'$ is homogeneous of degree $(-(d-2)-m)+(m-1)=-d+1$ so that 
$\partial_k\eta_R\partial_iv_\h'\partial_ju_\h'(x)$ 
$=\frac{1}{R^d}\partial_k\eta_1\partial_iv_\h'\partial_ju_\h'(\frac{x}{R})$, i.\ e.\
$\{\partial_k\eta_R\partial_iv_\h'\partial_ju_\h'\}_{R\uparrow\infty}$ acts like sequence of smooth averaging functions. 
Hence by Corollary~\ref{ergodicity} we obtain
\begin{align*}
\lim_{R'\uparrow\infty}\int\nabla\eta_R\cdot(\phi_i^*\partial_iv_\h'\ah\nabla u_\h')&=0,\nonumber\\
\lim_{R'\uparrow\infty}\int\nabla\eta_R\cdot(\partial_iu_\h'\partial_jv_\h'\sigma_ie_j)&=0,\nonumber
\end{align*}
and a similar statement with the roles of $v_\h$ and $u_\h$ exchanged and with $a$ replaced by $a^*$.
Using~\eqref{definition.C} we 
likewise learn from Corollary~\ref{ergodicity}
\begin{align*}
\lim_{R'\uparrow\infty}\int\nabla\eta_{R'}\cdot\big(\partial_iv_\h'\partial_ju_\h'
(\sigma_i^*\nabla\phi_j-\sigma_j\nabla\phi_i^*)\big)=
-\int\partial_k\eta_R\partial_iv_\h'\partial_ju_\h'C_{jki}.
\end{align*}
Recall that the integrands $v_\h'(\ah\nabla\ut+\partial_{ij}u_\h'C_{ij})-\ut \ah^*\nabla v_\h'$
and $u_\h'(\ah^*\nabla\vt+\partial_{ij}v_\h'C_{ij}^*)-\vt \ah\nabla u_\h'$ are homogeneous
of degree $-d+1$, see the proof of Lemmas \ref{Rom1} and \ref{Rom2}, so that the two middle
integrals of (\ref{wi26}) are independent of $R$.
Finally, because of $u_\h$ and $v_\h$ are $\ah$-harmonic and $\ah^*$-harmonic in $\{|x|>r\}$, respectively,
we have that $\int\nabla\eta_{R}\cdot(v_\h\ah\nabla u_\h-u_\h\ah^*\nabla v_\h)$ does not depend on $R$,
it is in fact equal to $-(u_\h,v_\h)_\h$.
Hence (\ref{wi26}) follows from Step \ref{P10:reduce2}.

\PfStep{P10}{P10:sym}
Only the symmetric part of $C$ matters; we claim
\begin{align}\label{ir23}
\lefteqn{\int\partial_k\eta_R\big(v_\h'\partial_{ij}u_\h'C_{ijk}-u_\h'\partial_{ij}v_\h'C_{ijk}^*-\partial_iv_\h'\partial_ju_\h' C_{jki}\big)}
\nonumber\\
&=\int\partial_k\eta_R\big(v_\h'\partial_{ij}u_\h'C_{ijk}^{\mathrm{sym}}-u_\h'\partial_{ij}v_\h'C_{ijk}^{*,\mathrm{sym}}
-\partial_iv_\h'\partial_ju_\h' C_{ijk}^{\mathrm{sym}}\big).
\end{align}
For notational simplicity we drop the indices $R$, $h$ and the primes. We first note that it is sufficient
to establish the formula
\begin{align}\label{ir22}
\lefteqn{\int\partial_k\eta\big(v\partial_{ij}uC_{ijk}-u\partial_{ij}vC_{ijk}^*-\partial_iv\partial_ju C_{jki}\big)}
\nonumber\\
&=-\int\partial_k\eta\partial_iv\partial_ju(C_{ijk}+C_{kij}+C_{jki})
  +\int\partial_{ijk}\eta v u C_{ijk}.
\end{align}
Indeed, (\ref{ir22}), which we apply to both $C$ and $C^{\mathrm{sym}}$ implies (\ref{ir23}) because 
1) by the symmetry of the third derivatives we may replace $C$ by $C^{\mathrm{sym}}$ in the second right-hand side term
of (\ref{ir22}) and 2) because of Step \ref{P10:Null} (which relies on the symmetry of second derivatives) 
we also may replace $C$ by $C^{\mathrm{sym}}$ in the first right-hand side term.

\medskip

We now turn to (\ref{ir22}). Using $C^*_{ijk}=-C_{kji}$, which immediately follows from~\eqref{definition.C}, 
and removing the identical
term on both sides, it turns into
\begin{align*}
\lefteqn{\int\partial_k\eta\big(v\partial_{ij}uC_{ijk}+u\partial_{ij}vC_{kji}\big)}
\nonumber\\
&=-\int\partial_k\eta\partial_iv\partial_ju(C_{ijk}+C_{kij})
  +\int\partial_{ijk}\eta v u C_{ijk}.
\end{align*}
By relabelling we see that it is enough to show for any triplet $i,j,k=1,\ldots,d$
\begin{align*}
\int\big(\partial_k\eta\,v\,\partial_{ij}u+\partial_i\eta\,\partial_{kj}v\,u\big)
&=\int\big(-\partial_k\eta\,\partial_iv\,\partial_ju
-\partial_i\eta\,\partial_jv\,\partial_ku
+\partial_{ijk}\eta\,v\,u\big).
\end{align*}
This formula follows  from the identity
\begin{align*}
\partial_k\eta\,v\,\partial_{ij}u+\partial_i\eta\,\partial_{kj}v\,u
+\partial_k\eta\,\partial_iv\,\partial_ju+\partial_i\eta\,\partial_jv\,\partial_ku-\partial_{ijk}\eta\,v\,u\\
=\partial_i(\partial_k\eta\,v\,\partial_ju)+\partial_k(\partial_i\eta\,\partial_jv\,u)
-\partial_j(\partial_{ik}\eta\,v\,u)
\end{align*}
since $\nabla\eta$ is compactly supported in $\mathbb{R}^d-\{0\}$, where $v$ and $u$ are smooth.

\PfStep{P10}{P10:Null}
Cancellation via Null-Lagrangian; we claim
\begin{align}\label{ir21}
\int\partial_k\eta_R\partial_iv_\h'\partial_ju_\h'(C_{ijk}+C_{kij}+C_{jki}-3C_{ijk}^{\mathrm{sym}})=0.
\end{align}
For notational simplicity we drop the indices $R$ and $h$ and the primes, and start by noting that
by definition of $C^{\mathrm{sym}}$,
\begin{align*}
C_{ijk}+C_{kij}+C_{jki}-3C_{jki}^{\mathrm{sym}}=\frac{1}{2}(C_{ijk}+C_{kij}+C_{jki}-C_{ikj}-C_{kji}-C_{jik}).
\end{align*}
Hence we have for the integrand of (\ref{ir21})
\begin{align}\label{ir20}
2\partial_k\eta\partial_iv\partial_ju(C_{ijk}+C_{kij}+C_{jki}-3C_{ijk}^{\mathrm{sym}})=D_{ijk} C_{ijk},
\end{align}
where $D_{ijk}$ denotes 
\begin{align*}
D_{ijk}&:=\partial_iv\partial_ju\partial_k\eta+\partial_kv\partial_iu\partial_j\eta
+\partial_jv\partial_ku\partial_i\eta\nonumber\\
&-\partial_iv\partial_ku\partial_j\eta-\partial_kv\partial_ju\partial_i\eta-\partial_jv\partial_iu\partial_k\eta,
\end{align*}
which is a $3$-minor
of the derivative matrix of the map $(v,u,\eta)$ from $\mathbb{R}^d-\{0\}$ into
$\mathbb{R}^3$. It is a classical result that such a minor can be written as a divergence, for the convenience
of the reader, and since $\eta$ plays a special role, we now display the short argument. 
One first has to realize that by symmetry of the second derivatives,
\begin{align*}
\partial_i(\partial_ju\partial_k\eta-\partial_ku\partial_j\eta)
+\partial_k(\partial_iu\partial_j\eta-\partial_ju\partial_i\eta)
+\partial_j(\partial_ku\partial_i\eta-\partial_iu\partial_k\eta)=0.
\end{align*}
This implies the desired representation as divergence:
\begin{align*}
\lefteqn{D_{ijk}}\nonumber\\
&=\partial_iv(\partial_ju\partial_k\eta-\partial_ku\partial_j\eta)
+\partial_kv(\partial_iu\partial_j\eta-\partial_ju\partial_i\eta)
+\partial_jv(\partial_ku\partial_i\eta-\partial_iu\partial_k\eta)\nonumber\\
&=\partial_i\big(v(\partial_ju\partial_k\eta-\partial_ku\partial_j\eta)\big)
+\partial_k\big(v(\partial_iu\partial_j\eta-\partial_ju\partial_i\eta)\big)\nonumber\\
&+\partial_j\big(v(\partial_ku\partial_i\eta-\partial_iu\partial_k\eta)\big).
\end{align*}
Since $\nabla\eta$ is compactly supported in $\mathbb{R}^d-\{0\}$, where $v$ and $u$ are smooth,
this yields $\int D_{ijk}=0$, which in view of (\ref{ir20}) yields (\ref{ir21}).

\PfStep{P10}{P10:Conclusion}
Asymptotic invariant of two-scale expansion
equals $\ah$-invariant; we claim
\begin{align*}
\lim_{R\uparrow\infty}\int\nabla\eta_{R}\cdot(Ev_\h a\nabla Eu_\h-Eu_\h a^*\nabla Ev_\h)
=-(u_\h,v_\h)_\h.
\end{align*}
Indeed, inserting Step \ref{P10:sym} into (\ref{wi26}) we obtain
\begin{align*}
\lefteqn{\lim_{R'\uparrow\infty}\int\nabla\eta_{R'}\cdot(Ev_\h a\nabla Eu_\h-Eu_\h a^*\nabla Ev_\h)
=-(u_\h,v_\h)_\h}\nonumber\\
&+\int\nabla\eta_R\cdot\big(v_\h'(\ah\nabla\ut+\partial_{ij}u_\h'C_{ij}^{\mathrm{sym}})-\ut \ah^*\nabla v_\h'\big)\nonumber\\
&-\int\nabla\eta_R\cdot\big(u_\h'(\ah^*\nabla\vt+\partial_{ij}v_\h'C_{ij}^{*,\mathrm{sym}})-\vt \ah\nabla u_\h'\big)\nonumber\\
&-\int\partial_k\eta_R\partial_iv_\h'\partial_ju_\h'C_{jki}^{\mathrm{sym}}.
\end{align*}
Now the boundary conditions of $\ut$ and $\vt$ at infinity
are defined just so that the last three terms cancel, see (\ref{R11}) and (\ref{R1}).
\end{proof}

\section{Proof of Theorem \ref{main}}
\begin{proof}
We fix $\beta > 1$ as prescribed in the statement of the theorem. As \BR5 stated in Propositions~\ref{stoch.results} and~\ref{Ck.1} \ER in Section \ref{stochastic.r}, under the hypothesis of stationarity and LSI$(\rho)$ for $\langle \cdot \rangle$, for $\langle \cdot \rangle$-almost every realization of $a$ 
\BR5 there are the first and second-order correctors $(\phi, \sigma), (\phi^*, \sigma^*)$ and $(\psi, \Psi), (\psi^*, \Psi^*)$ which satisfy the estimates \eqref{T.2} and \eqref{T.1} at the origin for an $r_* < +\infty$ and with exponents $\bar\beta > \beta$ and $\alpha \geq \bar\beta-1$. \ER
According to Corollary \ref{r.star.log}, we have for every $y \in \frac{1}{2\sqrt{d}}\mathbb{Z}^d$~\footnote{While in Corollary~\ref{r.star.log} we consider $\mathbb{Z}^d$, the same argument works verbatim also if $\mathbb{Z}^d$ is replaced by $\frac{1}{2\sqrt{d}}\mathbb{Z}^d$.}
\begin{align}\label{r.star.log2}
r_*(y) \lesssim \tilde c ( 1+ |y|^\varepsilon ) \quad \textrm{ with } \varepsilon := (\bar\beta - \beta)\BR5 \frac1d.
\end{align}
The above constant $\tilde c$ does depend on the realization $a$, \BR5 and \ER not only via the ellipticity ratio. In spite of this, since the estimates of Theorem \ref{main} are all purely qualitative, 
we do not care for the dependence of this constant on $a$ and use in this section the notation $\lesssim$ for $\leq C$ with $C$ depending only on $d$, $\lambda$, $m$ and $\tilde c$. We thus fix $a$ such that the above properties hold and identify the functions $\phi_i$ and $\psi_{ij}$ with the first and second-order correctors $\phi_i(a, \cdot)$, $\psi_{ij}(a, \cdot)- \fint_{|x|< r_*}\psi_{ij}(a,\cdot)$ in the case of the spaces $X_m$ and $X_m^\h$, and with  $\phi^*=\phi_i(a^*, \cdot)$, $\psi^*=\psi_{ij}(a^*, \cdot)- \fint_{|x|< r_*}\psi_{ij}(a^*,\cdot)$ in the case of the spaces $Y_{m-1}$ and $Y_{m-1}^\h$. For the tensor $C_{ijk}$ given by Proposition 1, we construct the maps $X_m^\h\ni u_\h \mapsto \ut$ and $Y_{m-1}^\h\ni v_\h \mapsto \vt$ according to Lemmas \ref{Rom1} and \ref{Rom2}. As in Section \ref{deterministic.r}, 
we have $\ut =0$ \BR5 for $m=2$\ER. We are therefore in the position to appeal to Theorem \ref{Rom3}.

\medskip

We now claim that the relation $X_m^\h \ni u_\h \leftrightarrow u \in X_m$ given by \eqref{o04} is equivalent to the one given by condition \eqref{m03} of Theorem \ref{Rom3}. Similarly, $Y_{m-1}^\h \ni v_\h\leftrightarrow v\in Y_{m-1}$ via \eqref{o05} is equivalent to $v_\h\leftrightarrow v$ via \eqref{m04}. Since by Theorem \ref{Rom3} this last two conditions define an isomorphism between $X_m^\h / X_{m-2}^\h$ and $X_m / X_{m-2}$ and between $Y_{m-1}^\h/Y_{m+1}^\h$ and $Y_{m-1}/Y_{m+1}$ which preserves the bilinear form, this is enough to conclude the proof of the theorem.

\medskip

We show the argument in the case of the growing functions. The other case follows analogously. \BR5 For the easy direction, if $u \in X_m$ and $u_\h \in X_m^\h$ are coupled by the relation \eqref{o04}, then using $\beta >1$ and $m\ge 2$ we immediately obtain~\eqref{m03}. 

%

\medskip
\ER

Vice versa, if we assume that $u$ and $u_\h$ are related by \eqref{m03}, then by Theorem \ref{Rom3} there exist $u' \in X_m$ and $u_\h' \in X_m^\h$ with $u'':= u-u' \in X_{m-2}$ and $u_\h'':= u_\h - u_\h'\in X_{m-2}^\h$ for which we have
\begin{align*}
\| u' - Eu_\h'\|_{m-\beta} \lesssim \| u_\h'\|_m.
\end{align*} 
Moreover, 
we apply Theorem \ref{T3} and obtain that for $y\in\frac{1}{2\sqrt{d}}\mathbb{Z}^d$,
\begin{align*}
\biggl( \fint_{|x - y| < 1} |\nabla (u' - Eu_\h')|^2\biggr)^\oh &\leq r_*(y)^{\frac d 2} \biggl( \fint_{|x - y| < r_*(y)} |\nabla (u' - Eu_\h')|^2\biggr)^\oh
\\
&\stackrel{\eqref{growth.loc}}{\lesssim} r_*(y)^{\frac d2} \biggl( \frac{|y|}{r_*(0)} \biggr)^{m-1-\bar\beta} \|u_\h'\|_m,
\\
&\stackrel{\eqref{r.star.log2}}{\lesssim} \BR5 r_*(0)^{\bar\beta-\beta}|y|^{\frac{\beta-\bar\beta}{2}}  \ER \biggl( \frac{|y|}{r_*}\biggr)^{m-1-\beta} \|u_\h'\|_m,
\end{align*} 
which, since we may cover each ball $\{ |x-y| < 1 \}$ with an order-one number of unit balls at points of the lattice $\frac{1}{2\sqrt{d}}\mathbb{Z}^d$, implies that the couple $u', u_\h'$ satisfies \eqref{o04}. By the triangle inequality it remains to show that also $u'', u_\h''$ satisfy \eqref{o04}. For every $y \in \mathbb{R}^d$ with $|y|=: R \gg 1$ we may use the triangle inequality and Lemma \ref{Rom5} to bound
\begin{align*}
|y|^{-(m-1-\beta)} \biggl( \fint_{|x - y| < 1}&|\nabla (u'' - Eu_\h'')|^2\biggr)^\oh \\
&\lesssim r_*(y)^{\frac d 2} R^{-(m-1-\beta)}\biggl( \fint_{|x-y| < r_*(y)} |\nabla u''|^2 + |\nabla u_\h''|^2\biggr)^\oh.
\end{align*}
We now appeal to Lemma \ref{Lup}, 
this time centered at $y$; we remark that this is allowed since the estimate on $r_*(y)$ and its definition (see Proposition \ref{stoch.results}) imply that conditions \eqref{F1} and \eqref{F2} are satisfied also when centered at $y$, with $r_*$ substituted by $r_*(y)$. 
Therefore, we may bound
\begin{align*}
|y|^{-(m-1-\beta)} \biggl( \fint_{|x - y| < 1}&|\nabla (u'' - Eu_\h'')|^2\biggr)^\oh\\
& \lesssim r_*(y)^{\frac d 2} R^{-(m-1-\beta)}\biggl( \fint_{|x - y| < R} |\nabla u''|^2+ |\nabla u_\h''|^2\biggr)^\oh,
\end{align*}
and use the fact that $|y|= R$, \BR5 and thus $\{|x-y| < R\} \subset \{|x| < 2R\}$, \ER to infer
\begin{align*}
|y|^{-(m-1-\beta)} \biggl( \fint_{|x - y| < 1}&|\nabla (u'' - Eu_\h'')|^2\biggr)^\oh\\
& \lesssim r_*(y)^{\frac d 2} r_*^{-(m-3)} R^{\beta-2}\bigl( \|u''\|_{m-2}+ \|u_\h''\|_{m-2} \bigr),
\end{align*}
which implies that \eqref{o04} holds also for the couple $u'', u_\h''$. This yields the theorem in the case of the spaces $X_m$ and $X_m^\h$.
\end{proof}

\appendix
\section{Proof of Corollary~\ref{r.star.log}.}

\begin{proof}

\BR5
Without loss of generality we show the argument only for $\phi$, and also assume $y=0$. 
\ER 
Let $0 < \alpha < 1$ be fixed, and define $\tilde r = \tilde r(a) \ge 1$ as the smallest radius such that for $R \ge \tilde r$
\begin{equation}\nonumber
 \frac1R \biggl( \fint_{|x| < R} \phi^2 \biggr)^{\frac12} \le \biggl( \frac{\tilde r}{R} \biggr)^\alpha.
\end{equation}
Given $\tilde r$, we redefine $r_*$ from Proposition~\ref{stoch.results} by $\max(r_*,\tilde r)$, so that for the new $r_*$ both~\eqref{T.2} and~\eqref{T.1} are satisfied.

\medskip
We now show that all algebraic moments of $\tilde r$ are finite, which then implies the same for the redefined $r_*$ (since the original $r_*$ also had finite algebraic moments, see Proposition~\ref{stoch.results}).  

\medskip
For that purpose we define $r_1=r_1(a)$ as the smallest (dyadic) radius satisfying for any dyadic $R \ge r_1$ (i.e. $R=2^k$ for some $k \in \mathbb{N}$) 
\begin{equation}\label{c22}
 \biggl( \fint_{|x| < R} \phi^2 \biggr)^{\frac12} \le 2^{-(d/2+1)} r_1^\alpha R^{1-\alpha}.
\end{equation}
By the choice of the prefactor $2^{-1-d/2}$ we see that for any (non-dyadic) $R \ge r_1$, and for $2^k$ such that  $2^{k-1} < R \leq 2^{k}$, we have as desired
\begin{align}\nonumber
 \frac1R \biggl( \fint_{|x| < R} \phi^2 \biggr)^{\frac12} 
&\le \frac1R \frac{2^{kd/2}}{R^{d/2}} \biggl( \fint_{|x| < 2^k} \phi^2 \biggr)^{\frac12}
\le  \biggl( \frac{r_1}{R} \biggr)^\alpha,
\end{align}
i.e. $r_1$ provides an upper bound for $\tilde r$. Hence, in order to show that $\en{ \tilde r^p } < \infty$ for any $1 \le p < \infty$ it is enough to show that $\en{ r_1^p } < \infty$ for any $1 \le p < \infty$. 

\medskip 

To prove the previous moment bounds we inspect the probability of the event $A_n := \{ r_1 > n \}$. In case of this event, by the definition of $r_1$ 
there exists a dyadic (random) $R \geq n$ for which~\eqref{c22} is false, in particular (after taking $p$-th power), for which it holds
\begin{equation}\nonumber
 \biggl( \fint_{|x| < R} \phi^2 \biggr)^{p/2} \ge 2^{-p(d/2+1)} n^{p\alpha} R^{p(1-\alpha)}.
\end{equation}
The left-hand side can be bounded from above using Jensen's inequality (while we momentarily restrict to $p \geq 2$):
\begin{equation}\nonumber
 \biggl( \fint_{|x| < R} \phi^2 \biggr)^{p/2} \lesssim R^{-d} \sum_{x_i} \biggl( \int_{|x-x_i|<1} \phi^2 \biggr)^{p/2},
\end{equation}
where we covered $\{|x| < R\}$ with $\lesssim R^d$ unit balls $\{ |x-x_i|<1 \}$, and where $\lesssim$ means $\le C(d,p)$. If we combine the two previous inequalities we thus obtain that, whenever $r_1 > n$, for some (random) dyadic $R \geq n$ one has
\begin{equation}\nonumber
n^{-p\alpha} R^{-p(1-\alpha)} R^{-d} \sum_{x_i} \biggl( \int_{|x-x_i|<1} \phi^2 \biggr)^{p/2} \gtrsim 1.
\end{equation}
This implies, by the definition of $A_n$, 
\BR5
the fact that the random variable $R$ only assumes values $r \ge n, r=2^k$, 
\ER and the stationarity of $\phi$, that
\begin{align*}
 \en{ I(A_n) } &\lesssim \en{ n^{-p\alpha} R^{-p(1-\alpha)-d} \sum_{x_i} \biggl( \int_{|x-x_i|<1} \phi^2 \biggr)^{p/2}}\\
 & \lesssim \sum_{r \ge n, r=2^k} n^{-p\alpha} r^{-p(1-\alpha)} 
\biggl< \biggl( \int_{|x| < 1} \phi^2 \biggr)^{p/2} \biggr>
\\
&\lesssim n^{-p} \biggl< \biggl( \int_{|x| < 1} \phi^2 \biggr)^{p/2} \biggr>.
\end{align*}
Thanks to Proposition \ref{stoch.results}, this last estimate implies that for any $2 \le p < \infty$ and $n \in \mathbb{N}$, the probability of $A_n$ is $\lesssim n^{-p}$. This yields that $\langle\, r_1^p \, \rangle \lesssim 1$ for any $2 \leq p < +\infty$ 
and, by Jensen's inequality, that the same bound holds for any $1\leq p < +\infty$.


\medskip

We now show~\eqref{rstar:bound}. Using the stationarity of $r_*$ we bound
\begin{align*}
\bigg\langle \sup_{y\in \mathbb{Z}^d} ((1+ |y|)^{-\varepsilon}r_*(y))^q \bigg\rangle^{\frac 1 q} \leq \bigg\langle \sum_{y\in \mathbb{Z}^d}(1+|y|)^{-\varepsilon q}r_*(y)^q \bigg\rangle^{\frac 1 q }
\le \left\langle r_*^q \right\rangle^{\frac 1 q }\biggl( \sum_{y\in \mathbb{Z}^d}(1+|y|)^{-\varepsilon q} \biggr)^{\frac 1 q}.
\end{align*}
We now restrict to $q > \frac{2d}{\varepsilon}$ and appeal to the assumption on the moments of $r_*$ (see Proposition \ref{stoch.results}) to get that $\langle \sup_{y\in \mathbb{Z}^d} ((1+ |y|)^{-\varepsilon}r_*(y))^q \rangle < \infty$. This yields that \eqref{rstar:bound} holds almost surely. 
\end{proof}


\section{Proof of Lemma \ref{tensor.C}. \ } \hspace*{\fill} 
\label{subsect_lemma1}

\begin{proof}
We begin by observing that \eqref{C.bdd} is an immediate consequence of definition \eqref{ir14} of $C$ and the second moment bounds on $\nabla \psi_{ij}$, $\phi_i$, $\sigma_i$ of Proposition~\ref{stoch.results}, together with the
boundedness \eqref{i01} of $a$.

\medskip

We now turn to identity \eqref{definition.C}. 
%
%
By appealing to the stationarity of $\phi^*$ and $\nabla\psi$ (see Proposition \ref{stoch.results}) and equation \eqref{i2} for $\psi^*_k$ lifted to the probability space\footnote{For a rigorous proof of the previous identity see, for instance, \cite[formula after (26)]{BellaFehrmanFischerOtto}.}, we infer that 
\begin{align}\label{eq.a}
\langle e_k\cdot a\nabla \psi_{ij} \rangle = -\langle \nabla\phi^*_k \cdot a\nabla \psi_{ij}\rangle,
\end{align}
and so rewrite~\eqref{ir14} as
\begin{align*}
C_{ijk}= \langle  e_k\cdot (\phi_ia-\sigma_i)e_j - \nabla\phi^*_k\cdot a\nabla \psi_{ij} \rangle \stackrel{\eqref{ir14bis}}{=}\langle \phi_i e_k\cdot a e_j - \nabla\phi^*_k\cdot a\nabla \psi_{ij} \rangle.
\end{align*}
Equation \eqref{ir13} for $\psi_{ij}$ allows us, by arguing as for \eqref{eq.a}, to reduce to
\begin{align*}
C_{ijk} = \langle  \phi_i e_j\cdot \BR2 a^* \ER ( e_k + \nabla\phi^*_k) - \nabla\phi^*_k\cdot \sigma_i e_j \rangle.
\end{align*}
We finally use equation \eqref{i5} for $\sigma_k^*$ and \eqref{ir14bis} in form of $\langle \phi \rangle =0$ to get
\begin{align*}
C_{ijk}= \langle \phi_i e_j \cdot \nabla \cdot \sigma_k^*  - \nabla\phi^*_k \cdot \sigma_i e_j \rangle.
\end{align*}
Using integration by parts on the first terms yields $\en{ \phi_i e_j \cdot \nabla \cdot \sigma_k^* } = \en{- e_j \cdot \sigma_k^* \nabla \phi_i}$. Owing to skew-symmetry of $\sigma_i$ we rewrite the second term $- \nabla\phi_k^* \cdot \sigma_i e_j = e_j \cdot \sigma_i \nabla\phi_k^*$, to infer~\eqref{definition.C}.  


\medskip

It remains to show ii): We stress that the ensemble $\langle\cdot\rangle$ being centrally symmetric means that the two coefficient fields $a$ and $x\mapsto a(-x)$, which we denote by $a(-\cdot)$, 
have the same distribution under $\langle\cdot\rangle$. We note that if $\phi_i=\phi_i(a,\cdot)$
solves (\ref{i2}) for a given realization $a$, then $x\mapsto-\phi_i(a,-x)$ solves (\ref{i2}) for
the coefficient field $a(-\cdot)$. By uniqueness of a stationary corrector $\phi_i$ of
vanishing expectation we thus have (in \BR3 a $\en{\cdot}$-almost \ER sure sense) the transformation rule
\begin{align*}
\phi_i(a(-\cdot),x)=-\phi_i(a,-x).
\end{align*}
This entails $\nabla\phi(a(-\cdot),x)=\nabla\phi_i(a,-x)$ and thus by (\ref{i5})
\begin{align*}
\nabla\cdot\sigma_i(a(-\cdot),x)=\nabla\cdot\sigma_i(a,-x).
\end{align*}
Since the same holds for $\phi_k^*$ and $\nabla\cdot\sigma_k^*$ we obtain
\begin{align*}
(\phi_i\nabla\cdot\sigma_k^*)(a(-\cdot),x)=-(\phi_i\nabla\cdot\sigma_k^*)(a,-x)
\end{align*}
and thus by invariance of $\langle\cdot\rangle$ under central symmetry (and translation)
\begin{align*}
\langle\phi_i\nabla\cdot\sigma_k^*\rangle=0,
\end{align*}
which by integration by parts and (\ref{definition.C}) in particular yields $C_{ijk}=0$.
\end{proof}

\section{Construction of the second-order correctors $(\psi,\Psi)$}
\label{construction_psiPsi}

We adapt the proof \cite[Lemma 1]{GNO4} to construct a second-order corrector $\psi_{ij}$ and its flux potential $\Psi_{ij}$:

\begin{lemma}
Let $\langle \cdot \rangle$ be stationary and ergodic. In addition let us assume that the correctors $(\phi, \sigma)$ are stationary with bounded second moments. Then there exists a random stationary vector field $\nabla \psi$, which is curl-free, satisfies $\en{\nabla \psi} = 0$, has bounded second moment $\en{ |\nabla \psi|^2 } \lesssim \en{ |(\phi,\sigma)|^2}$, and $\en{\cdot}$-almost surely solves~\eqref{ir13}.
We define the random vector field
\begin{align}\label{Q}
q_{ij}:= a\nabla\psi_{ij} + (\phi_ia - \sigma_i)e_j - C_{ij},
\end{align}
with $C_{ij}$ as defined in~\eqref{ir14} and $i,j=1,\ldots,d$. 
Then 
there exists a random fourth-order tensor field $\Psi$, skew symmetric in its last two indices, such that its gradient field $\nabla\Psi$ is stationary, $\en{ \nabla \Psi } = 0$, it satisfies
\begin{align}\label{Q0}
\langle |\nabla\Psi|^2 \rangle \lesssim \en{ |(\phi,\sigma)|^2},
\end{align}
and $\langle \cdot \rangle$-almost surely solves
\begin{equation}\label{PsiE}
\nabla\cdot \Psi_{ij} = q_{ij}.
\end{equation}
Moreover, for all $k,l = 1,..., d$ it also holds
\begin{align}\label{PsiE2}
-\Delta\Psi_{ij,kl} = \partial_k q_{ij,l} -\partial_l q_{ij,k}.
\end{align}
\end{lemma}

\begin{proof}
We first construct $\nabla \psi_{ij}$. Since we basically follow the argument of the first step of the proof of~\cite[Lemma 1]{GNO4}, we only briefly sketch the idea. We seek $\nabla \psi_{ij}$ as an element of $X := \{ g \in L^2(\Omega,\R^d) : D_j g_k = D_k g_j \textrm{ distributionally}, \en{g_j} = 0 \}$, where $D$ denotes the horizontal derivative, and observe that ellipticity of $a$ implies that $\en{ g \cdot a g } \ge \lambda \en{|g|^2}$ for all $g \in X$. By the Lax-Milgram there exists a unique $g \in X$ such that for all $\tilde g \in X$ we have $\en{\tilde g \cdot a g}= - \en{ \tilde g \cdot ((\phi_i a - \sigma_i)e_j)}$, which can be shown to satisfy $\en{ |g|^2 } \lesssim C_0$ and $D \cdot a g = - D \cdot ((\phi_ia -\sigma_i)e_j)$. It is this $g$ which plays the role $\nabla\psi_{ij}$. 

\medskip
It remains to construct $\nabla \Psi$. We start by observing that the stationarity of $(\phi,\sigma)$ and $\nabla\psi$ yield that $q_{ij}$ defined in \eqref{Q} is a stationary random field. 
In addition, we have that
\begin{align}\label{Q1}
\langle q_{ij} \rangle = 0 , \qquad \langle |q_{ij}|^2 \rangle \lesssim C_0, \qquad  D\cdot q_{ij}  \stackrel{(\ref{ir13})}{=}0.
\end{align}
In the remaining part of the proof, we consider $ij$ fixed and suppress it in the notation of $q$ and $\Psi$. 
Throughout this proof, we do not use Einstein's summation convention on repeated indices. We introduce the space of curl-free symmetric tensor fields of vanishing expectation
\begin{align*}
B:= \biggl\{ \tilde b \in L^2(\Omega, \mathbb{R}_{sym}^{d\times d}) \ : \  \partial_{k}\tilde b_{ij}= \partial_{j}\tilde b_{ik}, \ \ \langle \tilde b_{ij} \rangle =0 \biggr\}.
\end{align*}
For every $j =1,...,d$, let us consider $q_j {\rm I}$, where ${\rm I }$ denotes the $d \times d$ identity matrix, and let us denote by $b_j \in B$ the $L^2$-projection  of $q_j {\rm I}$ onto B. This implies that
\begin{align}\label{Q4}
\langle |b_j|^2 \rangle \leq \langle |q_j{\rm I}|^2 \rangle \stackrel{(\ref{Q1})}{\lesssim} |E|^2.
\end{align}
We now argue that the third-order tensor $b= b_{jkm}$ satisfies the identities
\begin{align}
\sum_k b_{jkk}= q_j, \qquad \sum_k b_{kkj}=0. \label{Q5}
\end{align}
To prove the first identity in (\ref{Q5}) we observe that since, $\{D^2\zeta : \zeta\in H^2(\Omega)\} \subset B$, it follows by orthogonality, integration by parts and the definition of $B$ that
\begin{align*}
0 &= \langle D^2 \zeta : (b_j - q_j{\rm I}) \rangle = \sum_{k,l}\langle D_{k}D_l\zeta (b_{jkl} - q_j\delta_{kl})\rangle 
= \sum_{k,l}\langle  D_{k}D_l\zeta \, b_{jkl} \rangle - \langle \text{trace}{(D^2\zeta)} q_j \rangle\\
& = -\sum_{k,l} \langle  D_l\zeta  D_{k}b_{jkl} \rangle - \langle\text{trace}{(D^2\zeta)} q_j \rangle = -\sum_{k,l} \langle  D_l\zeta  D_{l}b_{jkk} \rangle - \langle\text{trace}{(D^2\zeta)} q_j \rangle \\
&= \langle \text{trace}{(D^2\zeta)}(\text{trace}(b_j) - q_j) \rangle.
\end{align*}
Therefore, as by ergodicity $\{$trace$(D^2\zeta) \, \zeta \in H^2(\Omega) \}$ is dense in $B$ and since both $b_j$ and $q_j$ have vanishing expectation, identity (\ref{Q5}) follows. Similarly, the remaining identity in  \eqref{Q5} is implied once we argue that
\begin{align*}
\langle \text{trace}(D^2\zeta) (\sum_k b_{kkj}) \rangle = - \langle D D_j \zeta \cdot q \rangle \stackrel{(\ref{Q1})}{=} 0.
\end{align*}
This is obtained by integrating by parts the left-hand side in the line above and by observing that the curl-freeness and the symmetry conditions in the definition of $B$ combined with the first identity in (\ref{Q5}) yield
\begin{align*}
D_l D_l b_{kkj} = D_l D_j b_{kkl} = D_j D_l b_{kkl}= D_j D_l b_{klk}= D_{j}D_k b_{kll}.
\end{align*}
We now may extend the random tensor $b$ to a stationary random tensor field $b(a, x) := b(a( \cdot + x))$ such that, thanks to the relationship between horizontal derivatives $D_i$ and spatial derivatives $\partial_i$, it satisfies $\langle \cdot \rangle$- almost surely in the distributional sense
\begin{align*}
\partial_l b_{jkm}= \partial_m b_{jkl}.
\end{align*}
Therefore, there exists a field $\Psi= \Psi(a,x)$, defined up to a skew-symmetric constant tensor, such that 
\begin{align}\label{Q8}
\partial_l\Psi_{jk}= b_{jkl} - b_{kjl} \qquad \text{for all $j,k= 1,...,d$.}
\end{align}
Condition \eqref{Q0} trivially follows from \eqref{Q4} while equations (\ref{PsiE}) and (\ref{PsiE2}) are implied by the following calculations: By symmetry of $b$ it indeed holds
\begin{align*}
(\nabla \cdot \Psi)_j\stackrel{(\ref{Q8})}{=} \sum_{l} \partial_l\Psi_{jl} = \sum_l (b_{jll} - b_{ljl})=\sum_l (b_{jll} - b_{llj})\stackrel{(\ref{Q5})}{=} q_j,
\end{align*}
and by the symmetry and curl-freeness of $b$
\begin{align*}
-\Delta \Psi_{jk}&= -\sum_{l}\partial_{l}\partial_l \Psi_{jk} \stackrel{(\ref{Q8})}{=} -\sum_l (\partial_l b_{jkl}- \partial_l b_{kjl}) = -\sum_l (\partial_l b_{jlk}- \partial_l b_{klj})\\
&= -\sum_l (\partial_k b_{jll}- \partial_j b_{kll})\stackrel{(\ref{Q5})}{=} \partial_jq_k - \partial_k q_j.
\end{align*}
\end{proof}

\addtocontents{toc}{\protect\setcounter{tocdepth}{0}} 

\section{Proof of \texorpdfstring{\eqref{li01}}{} in Remark \texorpdfstring{\ref{remark.tensor.C}}{}.}
From (\ref{i5}) we infer the pointwise identity
\begin{align*}
\lefteqn{\big((x_i+\phi_i)a^* (e_k+\nabla\phi_k^*)-(x_k+\phi_k^*)a (e_i+\nabla\phi_i)\big)
            -\big( x_i        \ah^*e_k                - x_k          \ah e_i
 \big)}\nonumber\\
&=\big(\phi_i  \ah^* e_k+x_i\nabla\cdot\sigma_k^*+\phi_i  \nabla\cdot\sigma_k^*\big)
-\big(\phi_k^*\ah  e_i+x_k\nabla\cdot\sigma_i  +\phi_k^*\nabla\cdot\sigma_i  \big).
\end{align*}
We test the right-hand side with $\eta_R(x) = R^{-d} \eta(x/R)$, for a smooth non-negative cut-off function $\eta$ with $\int \eta = 1$. 
\BR5
Integrating by parts in the terms which include $\sigma$ while using sublinearity of $\phi$ and $\sigma$ to argue that the terms involving $\nabla \eta_R$ vanish in the limit $R \to \infty$, by sending $R \to \infty$ we see that by all three statements in Corollary~\ref{ergodicity} we obtain~\eqref{li01}. 
\ER



\section{Proof of estimate \texorpdfstring{\eqref{equi.v.full}}{}. \ } 

\BR5
Throughout this proof we use the notation $\lesssim$ for $\leq C$ with the constant $C=C(d, \lambda) < +\infty$. We will prove that any $\ah$-harmonic function $v$ in the annulus $\{ \frac 12 < |x| < 4 \}$ satisfies
\begin{equation}\label{eqPf1}
\sup_{1 < |x|< 2} ( |\nabla v| + |\nabla^2 v| + |\nabla^3 v|) \lesssim  \biggl(\fint_{\frac 12 <|x| < 4} |\nabla v|^2\biggr)^\oh,
\end{equation}
from which~\eqref{equi.v.full} will immediately follows by scaling. 

\medskip
To show~\eqref{eqPf1}, we first observe that by the Sobolev inequality we have for any function $v$ and for integer exponent $m > d/2$ 
\begin{equation}\nonumber
 \sup_{1 < |x|< 2} |v| \lesssim \sum_{k=0}^m \biggl( \int_{1 < |x| < 2} |\nabla^k v|^2 \biggr)^{\frac{1}{2}}
\end{equation}
and so to obtain~\eqref{eqPf1} it is enough to show that
\begin{equation*}
 \biggl( \int_{1 < |x| < 2} |\nabla^k v|^2 \biggr)^{\frac{1}{2}} \lesssim 
 \biggl( \int_{1/2 < |x| < 4} |\nabla v|^2 \biggr)^{\frac{1}{2}}
\end{equation*}
for $1 \le k \le d/2+5$. Since $\ah$ is constant, in particular any derivative of $v$ is again $\ah$-harmonic, the last estimate follows from iterating standard Caccioppoli's inequality. 
\ER



\def\cprime{$'$}
\providecommand{\bysame}{\leavevmode\hbox to3em{\hrulefill}\thinspace}
\providecommand{\MR}{\relax\ifhmode\unskip\space\fi MR }
\providecommand{\MRhref}[2]{%
  \href{http://www.ams.org/mathscinet-getitem?mr=#1}{#2}
}
\providecommand{\href}[2]{#2}



\end{document}